\documentclass[fullpage, 11pt]{article}

\usepackage{xr}
\usepackage[dvips]{graphicx}
\usepackage{easybmat}
\usepackage{color}
\usepackage[usenames,dvipsnames,svgnames,table]{xcolor}

\definecolor{lgray}{gray}{0.9}

\usepackage[english]{babel}
\usepackage{amsfonts}
\usepackage{amssymb}
\usepackage{amsmath}
\usepackage{latexsym}
\usepackage{amsthm}
\usepackage{epsfig}
\usepackage{graphicx}
\usepackage{color}
\usepackage{subfigure}
\usepackage{enumerate}

\usepackage{ifpdf}

\DeclareMathOperator    \aff                    {aff}

\DeclareMathOperator    \argmax         {arg\,max}
\DeclareMathOperator    \cl                     {cl}
\DeclareMathOperator    \conv           {conv}
\DeclareMathOperator    \cone           {cone}
\DeclareMathOperator    \dist           {dist}
\DeclareMathOperator    \im                    {im}
\DeclareMathOperator    \intr                   {int}
\DeclareMathOperator    \rank                {rank}
\DeclareMathOperator    \relint         {rel\,int}
\DeclareMathOperator    \verts          {vert}

\newcommand{\old}[1]{{}}

\newcommand{\bb}{\mathbb}

\newcommand{\R}{\bb R}
\newcommand{\Q}{\bb Q}
\newcommand{\Z}{\bb Z}
\newcommand{\N}{\bb N}

\newcommand{\shiftedp}{\bar{p}}
\newcommand{\y}{\bar{y}}			
\renewcommand{\a}{\bar{a}}		

\newcommand\st{\mid}
\newcommand\bigst{\mathrel{\big|}}

\newcommand\IH{{\mathrm{I}}}

\newtheorem{theorem}{Theorem}[section]
\newtheorem{corollary}[theorem]{Corollary}
\newtheorem{lemma}[theorem]{Lemma}
\newtheorem{proposition}[theorem]{Proposition}
        \newtheorem{prop}[theorem]{Proposition}

        \newtheorem{obs}[theorem]{Observation}

\newtheorem{definition}[theorem]{Definition}

\makeatletter
\def\th@claim{%
  \let\thm@indent\indent 
  \thm@headfont{\itshape}%
  \normalfont 
  \thm@preskip\topsep \divide\thm@preskip2
  \thm@postskip\thm@preskip
}
\makeatother

\theoremstyle{claim}
\newtheorem{claim}{Claim}

\makeatletter 
\newcommand{\greek}[1]{%
  \expandafter\@greek\csname c@#1\endcsname } 
\newcommand{\@greek}[1]{$\ifcase#1\or\alpha\or\beta\or\gamma\or\delta\or\varepsilon
  \or\zeta\or\eta\or\theta\or\iota\or\kappa\or\lambda \or\mu\or\nu\or\xi\or
  o\or\pi\or\varrho\or\sigma \or\tau\or\upsilon\or\phi\or\chi\or\psi\or\omega
  \else\@ctrerr\fi$} 
\makeatother

\makeatletter
\renewenvironment{proof}[1][\proofname]{\par
  \pushQED{\qed}%
  \normalfont \topsep6\p@\@plus6\p@\relax
  \trivlist
  \item[\hskip\labelsep
        \itshape\bfseries
    #1\@addpunct{.}]\ignorespaces
}{%
  \popQED\endtrivlist\@endpefalse
}
\newenvironment{claimproof}[1][\proofname]{\par
  \pushQED{\qedsymbol}%
  \normalfont \topsep1\p@\@plus1\p@\relax
  \trivlist
  \item[\hskip2em
        \itshape
    #1\@addpunct{.}]\ignorespaces
}{%
  \popQED\endtrivlist\@endpefalse
}
\makeatother

\makeatletter
\newcommand\Step[1]{\textbf{Step~#1.}
  \def\@currentlabel{#1}}
\makeatother

\newcommand{\Rf}{R_f}           

\newcommand{\T}{\mathcal{T}}                            
\newcommand{\Y}{\mathcal{Y}}                    
\renewcommand{\Psi}{\gamma}
\newcommand{\B}{\mathfrak B}
\newcommand{\M}{M}
\newcommand{\Ny}{\mathcal N(\Y)}
\newcommand{\Ty}{\T(\Y)}

\newcommand{\rvec}[1]{r^{j_#1}}
\newcommand{\pvec}[1]{p^{j_#1}}

\def\st{\,|\,}

\old{ \addtolength{\oddsidemargin}{-30pt}
\addtolength{\evensidemargin}{-30pt} \addtolength{\textwidth}{60pt}

\addtolength{\topmargin}{-22pt} \addtolength{\textheight}{32pt} }

\renewcommand{\texttt}[1]{{\tt{\textcolor{blue}{#1}}}}

\begin{document}

\title{The Triangle Closure is a Polyhedron}

\author{
Amitabh Basu\thanks{Dept. of Mathematics, University of California, Davis, 
{\tt{abasu@math.ucdavis.edu}}} \and
Robert Hildebrand\thanks{Dept. of Mathematics, University of California, Davis,
{\tt{rhildebrand@math.ucdavis.edu}}} \and
Matthias K\"oppe\thanks{Dept. of Mathematics, University of California, Davis, 
{\tt{mkoeppe@math.ucdavis.edu}} }
}

\date{\today\thanks{$\relax$Revision: 716 $ - \ $Date: 2013-01-08 16:15:15 -0700 (Tue, 08 Jan 2013) $ $}}

\maketitle

\begin{abstract}
Recently, cutting planes derived from maximal lattice-free convex sets have
been studied intensively by the integer programming community. An important
question in this research area has been to decide whether the closures
associated with certain families of lattice-free sets are polyhedra. For a
long time, the only result known was the celebrated theorem of Cook, Kannan
and Schrijver who showed that the split closure is a polyhedron. Although some
fairly general results were obtained by Andersen, Louveaux and Weismantel
[{\em An analysis of mixed integer linear sets based on lattice point free
  convex sets}, Math. Oper. Res. \textbf{35} (2010), 233--256]
and Averkov [\emph{On finitely generated closures in the theory of cutting
  planes}, Discrete Optimization \textbf{9} (2012), no.~4, 209--215], some basic questions have
remained unresolved. For example, maximal lattice-free triangles are the
natural family to study beyond the family of splits and it has been a standing
open problem to decide whether the triangle closure is a polyhedron. In this
paper, we show that when the number of integer variables $m=2$ the triangle
closure is indeed a polyhedron and its number of facets can be bounded by a
polynomial in the size of the input data. The techniques of this proof
are also used to give a refinement of necessary conditions for valid
inequalities being facet-defining due to Cornu\'ejols and
Margot [\emph{On the facets of mixed integer programs
  with two integer variables and two constraints}, Mathematical Programming
\textbf{120} (2009), 429--456] and obtain polynomial complexity results about
the mixed integer hull. 
\end{abstract}

\section{ Introduction}

We study the following system, introduced by Andersen et al.~\cite{alww}:
\begin{equation}
\begin{aligned}\label{M_fk}
&x  =  f + \sum_{j = 1}^k r^js_j \\
&x  \in  \mathbb{Z}^m \\
&s_j \geq  0 \quad \textrm{for all } j=1, \dots, k.
\end{aligned}
\end{equation}

This model has been studied extensively with the purpose of providing a unifying theory for cutting planes and exploring new families of cutting planes~\cite{alww, bbcm, rohtua,bc, cm, dl, dw2008}.
In this theory, an interesting connection is explored between valid inequalities for the convex hull of solutions to~\eqref{M_fk} (the \emph{mixed integer hull}) and maximal lattice-free convex sets in $\R^m$. A {\em lattice-free} convex set is a convex set that does not contain any integer point in its interior. A~{\em maximal} lattice-free convex set is a lattice-free convex set that is maximal with respect to set inclusion. Since the $x$ variables are uniquely determined by the $s_j$ variables, only the values of the $s_j$ variables need to be recorded for the system~\eqref{M_fk}, as done with the following notation:

\begin{equation}\label{eq:R_f}
R_f = \{ s \in \R^k_+ \st f + \sum_{j = 1}^k r^js_j \in \Z^m\}
\end{equation}
where $\R^k_+$ denotes the nonnegative orthant in $\R^k$. The mixed integer
hull is then denoted by $\conv(R_f)$ and can be obtained by intersecting all
valid inequalities derived using the Minkowski functional of maximal
lattice-free convex sets containing $f$ in their interior \cite{alww, bal,
  corner_survey, zambelli}. 
We explain this more precisely after introducing some notation.

Let $B \in \R^{n\times m}$ be a matrix with $n$ rows $b^1,\dots,b^n\in\R^m$.
We write $B = (b^1; \dots; b^n)$. Let 
\begin{equation}\label{eq:M(B)}
M(B) =\{\,x \in \R^m \st B\cdot (x-f) \leq e\,\},
\end{equation}
where $e$ is the vector of all ones. This is a polyhedron with $f$ in its
interior. We will denote the set of its vertices by $\verts(B)$. In fact, any
polyhedron with $f$ in its interior can be given such a description. We will
mostly deal with matrices $B$ such that $M(B)$ is a maximal lattice-free
convex set in $\R^m$.  
The Minkowski functional for the set $M(B)$ can be defined as $$\psi_B(r) = \max_{i \in \{1,\ldots, n\}} b^i\cdot r \quad\text{for $r
  \in \R^m$}.$$ 

\begin{prop}
If $B \in \R^{n\times m}$ is a matrix such that $M(B)$ is a lattice-free
convex set in $\R^m$ with $f$ in its interior, then the inequality $\sum_{j=1}^k\psi_B(r^j)s_j \geq 1$ is a valid inequality for~\eqref{M_fk}.
\end{prop}
\begin{proof}
Let $s \in R_f$.  Then $x = f + \sum_{j=1}^k r^j s_j \in \Z^m$.  Consider $\psi_B(x - f)$ and let $b \in \{b^1, \dots, b^n\}$ such that $\psi_B(x-f) = b \cdot (x-f)$.  Since $x \in \Z^m$ and $M(B)$ is lattice-free, by definition of $M(B)$, we have $b \cdot (x - f) \geq 1$.  Thus
$$
1 \leq b \cdot (x-f) = b \cdot \sum_{j=1}^k r^j s_j  = \sum_{j=1}^k(b \cdot r^j) s_j \leq \sum_{j=1}^k \psi_B(r^j) s_j.
$$
Therefore, the inequality $\sum_{j=1}^k\psi_B(r^j)s_j \geq 1$ holds for~$s$.
Since $s \in R_f$ was chosen arbitrarily, it is a valid inequality
for~\eqref{M_fk}. 
\end{proof}

We define the vector of coefficients as
\[
\gamma(B) = (\psi_B(r^j))_{j=1}^k
\]
and therefore can write the mixed integer hull as 
\begin{equation}\label{eq:mixed-int-hull-by-lattice-free}
  \conv(R_f) = \bigg\{\,s
\in \R^k_+ \mathrel{\bigg|} \begin{array}{@{}l@{}} \gamma(B)\cdot s \geq 1
  \textrm{ for all } B\in\R^{n\times m}\textrm{ such that } \\ M(B)\textrm{ is a maximal }\textrm{lattice-free convex set }\end{array}\,\bigg\}.
\end{equation}

\paragraph{Motivation.} All maximal lattice-free convex sets are polyhedra~\cite{rohtua, lovasz}. 
The most primitive type of maximal lattice-free convex set in $\R^m$ is the {\em split}, which is of the form $\pi_0 \leq \pi\cdot x \leq \pi_0 + 1$ for some $\pi \in \Z^m$ and $\pi_0 \in \Z$. A famous theorem due to Cook, Kannan and Schrijver~\cite{cook-kannan-schrijver} implies that the intersection of all valid inequalities for~\eqref{M_fk} derived from splits, known as the \emph{split closure}, is a polyhedron. 
The split closure result has been used repeatedly as a theoretical as well as practical tool in many diverse settings within the integer programming community. This motivates the following question: For which families of lattice-free convex sets is the associated closure a polyhedron?
 Not much was known about this question until very recently when some elegant results of a more general nature were obtained in~\cite{andersen-louveaux-weismantel} and~\cite{averkov}. Even so, some basic questions have remained open. Consider the case $m=2$. For this case, the different types of maximal lattice-free convex sets have been classified quite satisfactorily. Lov\'asz characterized the maximal lattice-free convex sets in $\R^2$ as follows.


\begin{theorem}[Lov\'asz \cite{lovasz}]
\label{mlfcb2}
In the plane, a maximal lattice-free convex set with non-empty interior is one of the following:
\begin{enumerate}[\rm(a)]
\item A split $c \leq a x_1 + b x_2 \leq c+1$ where $a$ and $b$ are co-prime integers and $c$ is an integer;
\item A triangle with an integral point in the interior of each of its edges;
\item A quadrilateral containing exactly four integral points, with exactly one of them in the interior of each of its edges.  Moreover, these four integral points are vertices of a parallelogram of area 1.
\end{enumerate}
\end{theorem}
Following Dey and Wolsey \cite{dw2008}, the maximal lattice-free triangles can be further partitioned into
three canonical types:
\begin{itemize}
\item \emph{Type 1 triangles}: triangles with integral vertices and exactly one integral point in the
relative interior of each edge;
\item \emph{Type 2 triangles}: triangles with at least one fractional vertex $v$, exactly one integral
point in the relative interior of the two edges incident to $v$ and at least two integral
points on the third edge;
\item \emph{Type 3 triangles}: triangles with exactly three integral points on the boundary, one in
the relative interior of each edge.
\end{itemize}

Figure \ref{fig:lattice_free_sets} shows these three types of triangles as well as a maximal lattice-free quadrilateral and a split satisfying the properties of Theorem \ref{mlfcb2}.

\begin{figure}
\label{fig:lattice_free_sets}
\centering
\scalebox{.5}{%
\ifpdf
\includegraphics{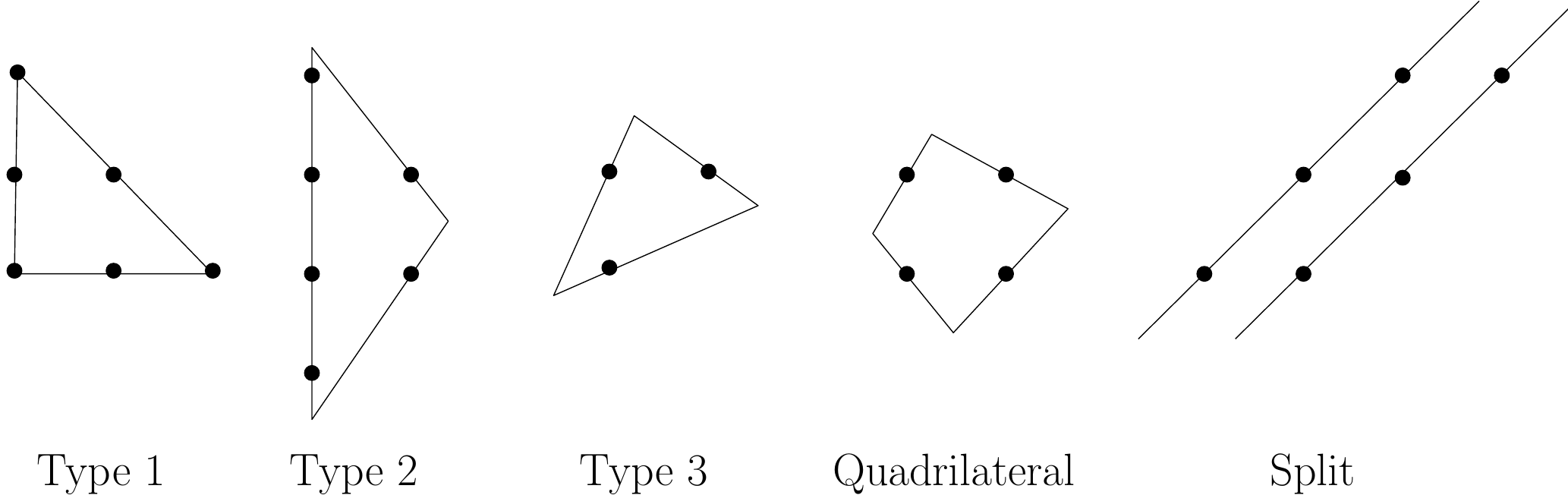}
\else
\includegraphics{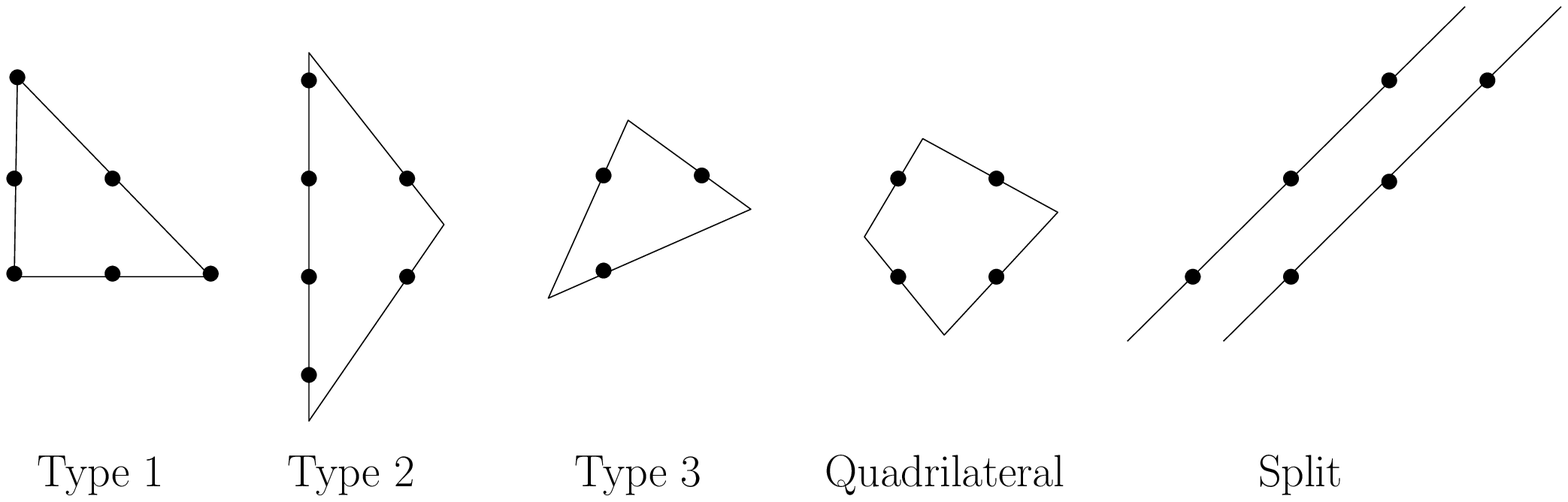}
\fi
}
%
\caption{Types of maximal lattice-free convex sets in $\R^2$.}
\end{figure}

For this simple case of $m=2$, it was not even known whether the triangle closure (the convex set formed by the intersection of all inequalities derived from maximal lattice-free triangles) is a polyhedron. The results from~\cite{andersen-louveaux-weismantel} and~\cite{averkov} cannot be used as they use an assumption of the so-called bounded lattice-width, which is not applicable here. In this paper, we settle this question in the affirmative under the assumption of rationality of all the data. The techniques used are substantially different from those in~\cite{andersen-louveaux-weismantel} and~\cite{averkov}. 


\label{sec:gen-lift}

\paragraph{Statement of Results.} Given a matrix $B \in \R^{3 \times 2}$, if $M(B)$ is a lattice-free set, then it will be either a triangle or a split in $\R^2$ (not necessarily maximal); the latter case occurs when one row of~$B$ is a scaling of another row. 

We define the {\em split closure}  as 
$$S = \{\,s \in \R^k_+ \st \gamma(B)\cdot s \geq 1 \textrm{ for all } B\in\R^{3\times 2}\textrm{ such that }M(B)\textrm{ is a lattice-free split\,}\}.$$
 Note that we are using a redundant description of convex sets that are splits, i.e., using 3 inequalities to describe it, instead of the standard 2 inequalities.
It follows from the result of Cook, Kannan and Schrijver~\cite{cook-kannan-schrijver}  that the split closure is a polyhedron.  We are interested in the closure using all inequalities derived from lattice-free triangles.

We define the \emph{triangle closure}, first defined in~\cite{bbcm}, as 
\[
T = \{\,s \in \R^k_+ \st \gamma(B)\cdot s \geq 1 \textrm{ for all } B\in\R^{3\times 2}\textrm{ such that }M(B)\textrm{ is a lattice-free triangle}\,\}.
\]
It is proved in~\cite{bbcm} that $T \subseteq S$, and therefore, $T = T \cap S$. This is because we can write a sequence of triangles whose limit is a split, and therefore all split inequalities are limits of triangle inequalities. Hence, using the fact that $T=T\cap S$, we can write the triangle closure as 
\begin{equation}\label{eq:generic_rep}
T= \{\,s \in \R^k_+ \st \gamma(B)\cdot s \geq 1 \textrm{ for all } B\in\R^{3\times 2}\textrm{ such that }M(B)\textrm{ is lattice-free\,}\}.
\end{equation}
The reason we  describe split sets using 3 inequalities is to write the triangle closure in a uniform manner using $3\times 2$ matrices as in \eqref{eq:generic_rep}. We note here that in the definition of $T$, we do not insist that the lattice-free set $M(B)$ is maximal.

We will prove the following theorem.

\begin{theorem}\label{thm:main}
Let $m=2$. Suppose that the data in~\eqref{M_fk} is rational, i.e., $f \in \Q^2$ and $r^j \in \Q^2$ for all $j =1, \ldots, k$. Then the triangle closure $T$ is a polyhedron with only a polynomial number of facets with respect to the binary encoding sizes of $f, r^1, \dots, r^k$.

\end{theorem}

We will first use convex analysis in Section \ref{preliminaries} to
illuminate the convex geometry of~$T$ by studying a set obtained from the defining inequalities of~$T$.  We will then demonstrate in Lemma \ref{lem:extreme_rep}
that it suffices to show that an associated convex set has finitely many extreme
points. In Section \ref{sec:finiteness}, we prove that there are indeed only
finitely many such extreme points, and in Section \ref{final}, we complete the
proof of Theorem \ref{thm:main}.  

The tools developed in Section~\ref{sec:finiteness} for proving Theorem~\ref{thm:main} will then be used to prove the following result about the mixed integer hull in Section~\ref{sec:polynomiality}.

\begin{theorem}
\label{THM:POLYFACETS}
For $m=2$, the number of facets of the mixed integer hull $\conv(R_f)$ is polynomial in the size of the
binary encoding sizes of $f, r^1, \dots, r^k$.
\end{theorem}

We prove the following result in Section~\ref{sec:polynomiality} as a direct consequence of our proof for Theorem~\ref{THM:POLYFACETS}.
\begin{theorem}\label{thm:enumerate}
There exists a polynomial time algorithm to enumerate all the facets of $\conv(\Rf)$ when $m=2$.
\end{theorem}

Apart from being the main machinery behind Theorems~\ref{thm:main},
\ref{THM:POLYFACETS} and \ref{thm:enumerate}, the results in
Section~\ref{sec:finiteness} also shed light on the classification results of
Cornu\'ejols and Margot for the facets of the mixed integer hull~\cite{cm}. In
Section~\ref{sec:finiteness}, we provide a more detailed set of necessary
conditions for a maximal lattice-free convex set to give a facet-defining
inequality. This avoids the use of an algorithm for a statement of such
necessary conditions (as was done in~\cite{cm} via the {\em Reduction
  Algorithm}) and also provides a completely different proof technique for
such classifications. This might also help towards obtaining such
results in dimensions higher than two, i.e., $m \geq 3$.  On the other
hand, we do not provide sufficient conditions, as was done in~\cite{cm}.

We make a remark about the proof structure of Theorem~\ref{thm:main} here. In this context, the most important result from Section~\ref{sec:finiteness} is Theorem~\ref{thm:polynomial}. Theorem~\ref{thm:polynomial} can be viewed as
the bridge between Section \ref{preliminaries} and Section \ref{final}. The
reader can follow the proof of Theorem~\ref{thm:main} by reading only
Sections~\ref{preliminaries} and~\ref{final}, if Theorem~\ref{thm:polynomial}
is assumed true. One can then return to
Section~\ref{sec:finiteness} to see the proof of Theorem~\ref{thm:polynomial},
which is rather technical. 
 
\section{\boldmath Preliminaries: Convex Analysis and the  Geometry of $T$}
\label{preliminaries}
 We will  prove several preliminary convex analysis lemmas relating to the  geometry of $T$.
We show that we can write the triangle closure $T$\ using a smaller set of inequalities.  We begin by defining the set of vectors which give the inequalities defining $T$,
\[
\Delta = \{\,\gamma(B) \st B\in \R^{3\times 2}\textrm{ such that }M(B)\textrm{ is lattice-free (not necessarily maximal)\,} \}.
\]
It is easily verified that for any matrix $B\in \R^{3\times 2}$, if $M(B)$ is lattice-free, then $\psi_B(r) \geq 0$ for all $r\in \R^2$ and therefore  $\Delta \subseteq \R^k_+$.

Let $\Delta' = \cl(\conv(\Delta)) + \R^k_+$ where $ \cl(\conv(\Delta))$ denotes the closed convex hull of $\Delta$ and $+$ denotes the Minkowski sum. Then $\Delta \subseteq \Delta'$ and $\Delta'$ is convex as it is the Minkowski sum of two convex sets (see Theorem 3.1 in \cite{rock}). In general the Minkowski sum of two closed sets is not closed (for example, $X = \{\,(x,y) \st y \geq 1/x, x > 0\,\}, Y = \{\,(x,y) \st x=0\,\}, X + Y = \{\,(x,y)  \st x > 0\,\}$). However, in this particular case, we show now that $\Delta'$ is closed. We will use the well-known fact that the Minkowski sum of two compact sets is indeed closed.  We prove the following more general result.

\begin{lemma}
\label{lem:closed}
Let $X,Y \subseteq \R^k_+$ be closed subsets of $\R^k_+$. Then $X + Y$ is closed.
\end{lemma}

\begin{proof}
Let $Z = X + Y$ and let $(z^n) \in Z$ such that $z^n \to z \in \R^k$.  
We want to show that $z \in Z$. 
Let $A = \{\, x \in \R^k_+ \st \|x\|_\infty \leq \|z\|_\infty + 1 \,\}$.  Since $\|z^n - z\|_\infty \to 0$ as $n\to \infty$, for some $N \in \mathbb{N}$, we must have that $\|z^n\|_\infty \leq \|z\|_\infty + 1$, that is, $z^n\in Z \cap A$, for all $n\geq N$.  
Since $X,Y \subseteq \R^k_+$, we see that $Z\cap A\subseteq (X \cap A) + (Y \cap A)$.  
Since $(X \cap A) + (Y \cap A)$ is a Minkowski sum of two closed and bounded subsets of $\R^k$, i.e., compact, $(X \cap A) + (Y \cap A)$ is closed. Therefore, the tail of $(z^n)$ is contained in a closed set, so it must converge to a point in the set, that is, $z \in (X \cap A) + (Y \cap A)$.  Since  $(X \cap A) + (Y \cap A) \subseteq X + Y = Z$,  we have that $z\in Z$. Therefore, $Z$ is closed.
\end{proof}

\begin{lemma}\label{lem:T_and_P'}
$T = \{\,s \in \R^k_+ \st \gamma\cdot s \geq 1 \textrm{ for all } \gamma \in \Delta'\,\}$.
\end{lemma}
\begin{proof}
Since $\Delta \subseteq \Delta'$, we have that $$\{\,s \in \R^k_+ \st
\gamma\cdot s \geq 1  \textrm{ for all } \gamma \in \Delta'\,\} \subseteq
\{\,s \in \R^k_+ \st \gamma\cdot s \geq 1 \textrm{ for all } \gamma \in \Delta\,\} = T.$$ We now show the reverse inclusion. Consider any $s \in T$ and $\gamma \in \Delta'$. We show that $\gamma \cdot s \geq 1$.

Since $\Delta' = \cl(\conv(\Delta)) + \R^k_+$, there exists $r\in \R^k_+$ and $a \in \cl(\conv(\Delta))$ such that $\gamma = a + r$. Moreover, there exists a sequence $(a^n)$ such that $(a^n)$ converges to $a$ and $(a^n)$ is in the convex hull of points $q^j \in \Delta$, $j\in J$. Since $q^j \cdot s \geq 1$ for all $j\in J$, we have that $a^n \cdot s \geq 1$ for all $n \in \mathbb{N}$. Therefore $a\cdot s = \lim_{n \to \infty} a^n\cdot s \geq 1$. Since $r \in \R^k_+$, $r\cdot s \geq 0$ and so $\gamma\cdot s = (a+r)\cdot s \geq a\cdot s \geq 1$.
\end{proof}

We say that $a\in \Delta'$ is a {\em minimal} point if there does not exist $x \in \Delta'$ such that $a - x \in \R^k_+\setminus\{0\}$. If such an $x$ exists then we say that $a$ is {\em dominated} by~$x$.
We introduce some standard terminology from convex analysis. Given a convex set $C \subseteq \R^k$, a {\em supporting hyperplane} for~$C$ is a hyperplane $H = \{\,x \in \R^k \st h\cdot x = d\,\}$ such that $h\cdot c \leq d$ for all $c\in C$ and $H \cap C \neq \emptyset$. A point $x \in C$ is called {\em extreme} if there do not exist $y^1$ and~$y^2$ in $C$ different from $x$ such that $x = \frac12(y^1 + y^2)$. If such $y^1\neq y^2$ exist, we say that $x$~is a \emph{strict convex combination} of $y^1$ and~$y^2$.  A point~$x$ is called {\em exposed} is there exists a supporting hyperplane~$H$ for~$C$ such that $H \cap C = \{x\}$. We will denote the closed ball of radius~$r$ around a point~$y$ as~$ \B(y,r)$. We denote the boundary of this ball by~$\partial \B (y,r)$.

\begin{lemma}\label{lem:rec_cone}
$\Delta'$ is a closed convex set with $\R_+^k$ as its recession cone.
\end{lemma}

\begin{proof}
Recall that $\Delta \subseteq \R_+^k$. Since $\R^k_+$ is closed and convex, $\cl(\conv(\Delta)) \subseteq \R^k_+$ and so $\Delta' = \cl(\conv(\Delta)) + \R_+^k$ is closed by Lemma~\ref{lem:closed}. Since the Minkowski sum of two convex sets is convex, $\Delta'$ is convex. Moreover since $\Delta' \subseteq \R^k_+$, the recession cone of $\Delta'$ is $\R^k_+$.
\end{proof}

\begin{lemma}\label{lem:extreme_rep}
Let $C$ be the set of extreme points of $\Delta'$. Then $$T =  \{\,s \in \R^k_+ \st a\cdot s \geq 1  \textrm{ for all } a \in C\,\}.$$
\end{lemma}

\begin{proof}
Let $\hat T = \{\,s \in \R^k_+ \st a\cdot s \geq 1 \text{ for all } a \in C\,\}$. Since $C \subseteq \Delta'$, we have that $T \subseteq \hat T$. We show the reverse inclusion. Consider any $s\in \hat T$.

By Lemma~\ref{lem:rec_cone}, $\Delta'$ is a closed convex set with $\R_+^k$ as its the recession cone. Therefore, $\Delta'$ contains no lines. This implies that any point $a \in \Delta'$ can be represented as $a = z + \sum_j \lambda_jv^j$ where $z$ is a recession direction of $\Delta'$, $v^j$'s are extreme points of $\Delta'$, $\lambda_j \geq 0$ and $\sum_j \lambda_j = 1$ (see Theorem 18.5 in \cite{rock}). Moreover, since the $v^j$'s are extreme points, $v^j \in C$ and therefore $v^j\cdot s \geq 1$ for all $j$ because $s \in \hat T$. Since $z \in \R^k_+$, $s\in \R^k_+$, $\lambda_j \geq 0$ for all $j$ and $\sum_j \lambda_j = 1$, $a\cdot s = z\cdot s + \sum_j \lambda_j (v^j\cdot s) \geq 1$. Therefore, for all $a \in \Delta'$, $a\cdot s \geq 1$. By Lemma~\ref{lem:T_and_P'}, $s \in T$.
\end{proof}

\begin{obs}\label{obs:minimal}
Since the recession cone of $\Delta'$ is $\R^k_+$ by Lemma~\ref{lem:rec_cone}, every extreme point of $\Delta'$ is minimal.
\end{obs}

We end this section with the following technical lemma.

\begin{lemma}\label{lem:seq_extreme} Let $A$ be any subset of $\R^k$ and let $A' = \cl(\conv(A))$. Then for any extreme point $x$ of $A'$, there exists a sequence of points $(a^n) \in A$ converging to $x$. 
\end{lemma}

\begin{proof}
We first show the following claim.

\begin{claim}\label{claim:exposed} For any exposed point $a$ of $A'$, there exists a sequence of points $(a^n) \in A$ converging to $a$.\end{claim}
\begin{claimproof}
Let $H = \{\,x \in \R^k \st h \cdot x = d\,\}$ be a supporting hyperplane for $A'$ such that $H \cap A' = \{a\}$. 
Suppose to the contrary that there does not exist such a sequence in $A$. 
This implies that there exists $\epsilon > 0$ such that $\B(a,\epsilon) \cap A = \emptyset$. Let $D = \partial \B(a,\epsilon) \cap H$. 
Since $H \cap A' = \{a\}$, for any point $c\in D$, $\dist(c, A') > 0$. Since $D$ is a compact set and the distance function is a Lipschitz continuous function, there exists $\delta > 0$ such that $\dist(c, A') > \delta$ for all $c \in D$. 
We choose $\delta'$ such that for any $y \in \partial \B(a,\epsilon)$ satisfying $d \geq h\cdot y > d - \delta'$, there exists $c \in D$ with $\dist(c,y) < \delta$.

Since $a \in \cl(\conv(A))$, there exists a sequence of points $(b^n) \in \conv(A)$ converging to $a$. 
This implies that $(h\cdot b^n)$ converges to $h\cdot a = d$. Therefore, we can choose $b$ in this sequence such that $h\cdot  b > d - \delta'$ and $b \in \B(a, \epsilon)$. 
Since $b \in \conv(A)$ there exist $v^j \in A, j = 0, \ldots, k$ such that $b = \conv(\{v_0, \ldots, v_k\})$. Therefore, for some $j$, $h\cdot v^j > d- \delta'$.
 Moreover, since $v^j \in A$ and $\B(a,\epsilon) \cap A = \emptyset$, $v^j \not\in \B(a,\epsilon)$. 
 Since $b \in \B(a, \epsilon)$ and $v^j \not\in \B(a,\epsilon)$, there exists a point $p \in \partial \B(a,\epsilon)$ such that $p$ is a convex combination of $b$ and $v^j$. 
 Since $h\cdot  b > d - \delta'$ and $h\cdot v^j > d- \delta'$, we have that $h\cdot p > d- \delta'$. 
Moreover $b \in \conv(A)$ implying $b \in A'$ and $v^j \in A'$, so we have $p \in A'$ and so $d \geq h\cdot p$ since $H$ is a supporting hyperplane for $A'$. So by the choice of $\delta'$, we have that there exists $c \in D$ with $\dist(c, p) < \delta$.
 However, $\dist(c, A') > \delta$ for all $c \in D$ which is a contradiction because $p \in A'$.
\end{claimproof}

By Straszewicz's theorem (see for example Theorem 18.6 in \cite{rock}), for any extreme point $x$ of $A'$, there exists a sequence of exposed points converging to $x$. 
So for any $n \in \mathbb{N}$, there exists an exposed point $e^n$ such that $\dist(e^n, x) < \frac{1}{2n}$ and using Claim~\ref{claim:exposed}, there exists $a^n \in A$ such that $\dist(e^n, a^n) < \frac{1}{2n}$. Now the sequence $(a^n)$ converges to $x$ since $\dist(a^n, x) < \frac{1}{n}$.
\end{proof}


\section{\boldmath Polynomially Many Extreme Points in $\Delta'$}
\label{sec:finiteness}
In this section, we will introduce certain tools and use them to prove the following proposition. 

\begin{proposition}\label{prop:almost-extreme}
There exists a finite set $\Xi{} \subseteq \Delta$, such that if $\gamma \in \Delta\setminus \Xi{}$, then $\gamma$ is dominated by some $\gamma' \in \Delta$, or $\gamma$ is the strict convex combination of $\gamma^1$ and $\gamma^2 \in \Delta$. Furthermore, the cardinality of $\Xi$ is bounded polynomially in the binary encoding size of $f, r^1, \ldots, r^k$.
\end{proposition}

This proposition will be proved by carefully counting the
triangles and splits $M(B)$ such that $\gamma(B)$~is not dominated by a point
in $\Delta$ and is not a strict convex combination of points in~$\Delta$. To this end, we define the following subsets of $\R^k_+$:
\begin{align*}
  \Delta_i &= \{\, \gamma(B) \st B\in\R^{3\times2},\, M(B)\text{ is a Type $i$
    triangle}\,\}\quad\text{for $i=1,2,3$}\\\intertext{and}
  \Pi &= \{\, \gamma(B) \st B\in\R^{3\times2},\, M(B)\text{ is a maximal lattice-free split}\,\}.
\end{align*}
Note that these sets are not disjoint, as the same vector~$\gamma$ can be
realized by maximal lattice-free convex sets of different kinds.

We first develop some important concepts in Subsection~\ref{sec:tilt-space},
followed by the main counting arguments in Subsection~\ref{sec:counting}. In
Subsection~\ref{sec:proofs}, we prove Proposition~\ref{prop:almost-extreme}.  
It has the following theorem as a consequence, which is the most important
ingredient for the triangle closure result. 
\begin{theorem}\label{thm:polynomial}
There  exists a finite set $\Xi{} \subset \Delta$ such that if $\gamma$ is an extreme point of $\Delta '$, then $\gamma \notin \Delta \setminus \Xi{}$.
Furthermore, the cardinality $\#\Xi{}$ is polynomial in the encoding sizes of $f, r^1, \dots, r^k$.
\end{theorem}
\begin{proof}
  Let $\Xi$ be the set from  Proposition~\ref{prop:almost-extreme}. 
  Since $\Delta \subseteq \Delta'$, together with
  Observation~\ref{obs:minimal} and the definition of extreme point, this
  implies that $\Delta\setminus \Xi{}$ does not contain any extreme points of
  $\Delta'$. 
\end{proof}

\subsection{Tools}\label{sec:tilt-space} 


For fixed $f$,  given a set of rays $R \subseteq \R^m$ and a matrix $B=(b^1 ; \dots ; b^n)
\in \R^{n\times m}$, we refer to the set of \emph{ray intersections} $$ P(B,R) = \bigl\{\,p(B,r) \in \R^m\bigst
 r \in R,\ \psi_B(r) > 0 \bigr\}$$ 
where $p(B,r) =  f + \tfrac{1}{\psi_B(r)} r$ is the point where $r$ meets the boundary of the set $M(B)$.    Therefore, $P(B,R) \subset \partial M(B)$.  Furthermore, if $P(B^1, \{r^1, \dots, r^k\}) = P(B^2, \{r^1, \dots, r^k\})$, then $\gamma(B^1) = \gamma(B^2)$.  
If $p(B,r)$ is a vertex of $M(B)$, then we call $r$ a \emph{corner ray} of $M(B)$. 

Whenever
$\psi_B(r) > 0$, the set
$I_B(r) = \argmax_{i=1, \dots, n} b^i\cdot  r$ is the index set of all defining inequalities
of the polyhedron $M(B)$ that the ray intersection $p(B,r)$
satisfies with equality. In particular, for $m=2$, when all the inequalities corresponding to the rows
of $B$ are different facets of $M(B)$, we have $\#I_B(r)= 1$ when $r$ points to the
relative interior of a facet, and $\#I_B(r)= 2$ when $r$ points to a vertex of
$M(B)$, where $\# X$ denotes the cardinality of the set~$X$. 

Let $F_i(B) = M(B)\cap\{x \in \R^m \st b^i\cdot(x-f) = 1 \}$ for each $i = 1, \ldots, n$ (in what follows, $F_i(B)$ will usually be a facet of $M(B)$). Let $Y(B)$ be the set of integer points contained in $M(B)$.  Recall that if $M(B)$ is a maximal lattice-free convex set, then each facet $F_i(B)$ contains at least one integer point in its relative interior.  
In our proofs, it is convenient to choose, for every $i=1, \ldots, n$, 
a certain subset $Y_i\subseteq Y(B)\cap F_i(B)$ of the integer points on~$F_i(B)$. 

\begin{definition}
Let $\Y$ denote the tuple $(Y_1, Y_2, \ldots, Y_n)$.  The \emph{tilting
  space}  $\T(B, \Y, R)\subset \R^{n\times m}$ is defined as  
the set of matrices $A = (a^1; a^2; \ldots ;a^n)\in\R^{n\times m}$ that satisfy the
following conditions:
\begin{subequations}
  \begin{align}
    a^i\cdot (y - f) &= 1 && \text{for} \ y \in Y_i,\ i=1, \dots, n,\label{eq:tilting-1}\\
    a^{i}\cdot  r &= a^{i'}\cdot  r   && \text{for} \ i,i' \in I_B(r), \text{and all } r \in R\label{eq:tilting-2}\\
    a^{i}\cdot r&> a^{i'}\cdot r && \text{for} \ i \in I_B(r),\ i' \notin I_B(r), \text{and  all }r \in R
    .\label{eq:tilting-3}
  \end{align}
\end{subequations}
\end{definition}

Note that if $R' \supseteq R$, then $\T(B,\Y,R') \subseteq \T(B, \Y, R)$.   The tilting space~$\T(B, \Y, R)$ is defined for studying perturbations of the lattice-free set~$M(B)$. This is done by changing or tilting the facets of~$M(B)$ subject to certain constraints to construct a new lattice-free set $M(A)$. Constraint~\eqref{eq:tilting-1} requires that $F_i(A)$ must contain subset~$Y_i$ of integer points. Constraints~\eqref{eq:tilting-2} and~\eqref{eq:tilting-3} together imply that for any $r \in R$, the  ray intersection $p(A,r) = f + \frac{1}{\psi_A(r)} r$ lies on~$F_i(A)$ of~$M(A)$ if and only if the ray intersection $p(B,r) = \smash[t]{f + \frac{1}{\psi_B(r)} r}$ for~$M(B)$ lies on the corresponding $F_i(B)$ of~$M(B)$. Thus we have $I_A(r) = I_B(r)$.  In particular, this
means that if $r\in R$~is a corner ray of $M(B)$, then $r$ must also be a
corner ray for~$M(A)$ if $A \in \T(B,\Y, R)$.

\smallbreak

Note that $\T(B,\Y,R)$ is defined by linear equations and strict linear inequalities and,
since  $B\in \T(B,\Y, R)$, it is non-empty.  Thus it is a convex set
whose dimension is the same as that of the affine space
given by the equations \eqref{eq:tilting-1}~and~\eqref{eq:tilting-2} only. This motivates the following definition.

\begin{definition}
$\mathcal{N}(B, \Y, R)$ will denote the nullspace of the equations~\eqref{eq:tilting-1} and~\eqref{eq:tilting-2}, i.e., $$\mathcal{N}(B, \Y, R) = \bigg\{A =(a^1; \dots; a^n) \in \R^{n\times m} \mathrel{\bigg|}  \begin{array}{@{}cc@{}l@{}} a^i\cdot (y - f) = 0 && \text{for} \ y \in Y_i,\ i=1, \dots, n,\\
    a^{i}\cdot  r = a^{i'}\cdot  r   && \text{for} \ i,i' \in I_B(r) \text{ and all } r \in R\end{array}\,\bigg\}.$$ 
\end{definition}

If $R' \supseteq R$, then $\mathcal N(B,\Y, R') \subseteq N(B,\Y,R)$.  For many cases when $\gamma(B)$ is not extreme, we will find a matrix $\bar A \in \mathcal N(B,\Y, R)$   such that $\gamma(B)$ can be expressed as the convex combination of $\gamma(B + \epsilon \bar A)$ and $\gamma(B - \epsilon \bar A)$ and $M(B + \epsilon \bar A), M(B - \epsilon \bar A)$ are lattice-free polytopes.  If $\gamma(B) \in \Delta$, then $\gamma(B + \epsilon \bar A), \gamma(B - \epsilon \bar A)$ will also be in $\Delta$.   


For convenience, we shorten notation in the following way: for the rest of the paper we fix the input data $f$ and the set of rays $R = \{r^1, \dots, r^k\}$.  Furthermore, whenever the matrix $B$ is clear from context, we write $\Ny = \mathcal N(B, \Y, R)$, $\Ty = \T(B, \Y, R)$,  $P = P(B, R)$,
$p^j = p(B, r^j)$  for $j = 1, \dots, k$, and
$F_i = F_i(B)$.

Next we introduce a tool that helps to ensure that sets $M(B \pm \epsilon \bar A)$ are lattice-free. This will be done by utilizing now-classic results in the theory of parametric linear
programming.  Specifically, consider a parametric linear program, \begin{equation}\label{eq:param_lp}\sup\{\,
c(t)\cdot x : A(t) x \leq b(t)\,\} \in\R\cup\{\pm\infty\},\end{equation}
where all coefficients
depend continuously on a parameter vector $t$ within some parameter region
$\mathcal R \subseteq \R^q$.  
\begin{theorem}[D.~H.~Martin~\cite{martin-1975-continuity-maximum}, Lemma~3.1]
  \label{thm: locally-bounded-solution-set}
  Suppose that the solution set of \eqref{eq:param_lp} for $t=t_0$ is non-empty and bounded.  Then,
  in parameter space, there is an open neighborhood~$\mathcal O$ of~$t_0$ such that the
  union of all solution sets for $t\in \mathcal O$ is bounded. 
\end{theorem}
We use this theorem of parametric linear programming to prove the following lemma.
\begin{lemma}
\label{lemma: S(B) full-dimensional}
Let $B\in \R^{n\times m}$ be such that $\M(B)$ is a bounded maximal
lattice-free set.  Then for every $\bar A \in \R^{n\times m}$, there exists
$\delta > 0$ such that for all $0 < \epsilon < \delta$, the set $Y(B+\epsilon\bar A)$ of integer points contained
in $\M(B + \epsilon\bar A)$ is a subset of $Y(B)$.
\end{lemma}
\begin{proof}
  Consider the parametric linear program
  $$ \max \{\,0 \st a^i\cdot (x-f) \leq 1,\ i=1, \dots, n\,\}$$
  with parameters $t = A = (a^1;\dots;a^n) \in \R^{n\times m}$. 
  By the assumption of the lemma, the solution set for $t_0 = B = (b^1;\dots;b^n)$
  is bounded.  Let $\mathcal O$ be the open neighborhood of~$t_0$ from
  Theorem~\ref{thm: locally-bounded-solution-set},
  and let $\hat S$ be the union of all solution sets for $t\in \mathcal O$, which is by
  the theorem a bounded set.  
  
  For each of the finitely many lattice points $y \in \hat S \setminus \M(B)$, let 
  $i(y)\in\{1,\dots,n\}$ be an index of an inequality that cuts off~$y$, that
  is, $b^{i(y)}\cdot (y - f) > 1$.  Then, given any $\bar A \in \R^{n\times m}$, there exists $\delta > 0$ such that for all $0 < \epsilon < \delta$ such that $(b^{i(y)} + \epsilon \bar a^{i(y)})\cdot(y - f) > 1$ for all $y \in \hat S \setminus \M(B)$. Therefore, for all such $\epsilon$, we have $Y(B + \epsilon A) \subseteq Y(B)$.\end{proof}
 

\begin{lemma}[General tilting lemma] \label{obs:dimension}
Let $B\in\R^{n\times m}$ be such that $M(B)$ is a bounded lattice-free
set.
Suppose $\Y  = (Y_1, \dots, Y_n)$ is a covering of $Y(B)$, i.e., $Y(B)
\subseteq Y_1\cup\dots\cup Y_n$. For any $\bar A\in\Ny \setminus \{0\}$, there exists $\delta > 0$ 
such that for all $0 < \epsilon < \delta$ the following statements hold:
\begin{enumerate}[\rm(i)]
\item 
$I_B(r^j) = I_{B+ \epsilon \bar A}(r^j) =  I_{B- \epsilon \bar A}(r^j)$ for all $j = 1, \ldots, k$.
\item 
$\gamma(B)= \frac{1}{2}\gamma({B + \epsilon \bar A}) + \frac{1}{2} \gamma({B - \epsilon \bar A})$.
\item
Both $M(B \pm \epsilon \bar A)$ are lattice-free.
\item 
Suppose $m=2$ and let $\bar a^1$, \dots, $\bar a^n\in\R^2$ denote the rows
of~$\bar A$. Suppose there exists an index $i \in\{1,\dots, n\}$ such that $\bar
a^i \neq 0$, $\# Y_i = 1$, and $(F_i \cap P) \setminus \Z^2 \neq \emptyset$. 
Then  $\gamma(B)$ is a strict convex combination of $\gamma(B + \epsilon \bar A)$ and $\gamma(B - \epsilon \bar A)$. 
\end{enumerate}
\end{lemma}
\begin{proof}
Since $\bar A \in \Ny \setminus \{0\}$, $B \pm \epsilon
\bar A$ satisfies the equations~\eqref{eq:tilting-1} and~\eqref{eq:tilting-2}
for any~$\epsilon$, and there exists $\delta_1 > 0$ such that
$B \pm \epsilon \bar A$ satisfies the strict inequalities~\eqref{eq:tilting-3} 
for $0\leq \epsilon< \delta_1$.
Thus $B \pm \epsilon \bar A \in \mathcal T(\mathcal Y)$ for $0\leq \epsilon<
\delta_1$.
Let $\delta_2 > 0$ be
obtained by applying Lemma~\ref{lemma: S(B) full-dimensional}. Choose $\delta
= \min\{\delta_1, \delta_2\}$. 

\emph{Part (i).} This follows from the fact that $B \pm
\epsilon \bar A \in \mathcal T(\mathcal Y)$, 
and thus \eqref{eq:tilting-2} and \eqref{eq:tilting-3} hold, for all $0 < \epsilon < \delta$.

\emph{Part (ii).} This is a consequence of (i), since $\gamma(B)_j = \psi_B(r^j) = b^i\cdot r^j$ for any $i\in I_B(r^j)$, and $\gamma(B\pm\epsilon \bar A)_j = \psi_{B\pm\epsilon\bar A}(r^j) = (b^i\pm\epsilon\bar a^i)\cdot r^j$ for any $i\in I_{B\pm\epsilon\bar A}(r^j)$. 

\emph{Part (iii).} This follows from Lemma~\ref{lemma: S(B) full-dimensional} and the fact
that $Y(B)$ remains on the boundary of $M(B\pm \epsilon\bar A)$, due to
constraint~\eqref{eq:tilting-1}.

\emph{Part (iv).}
Since $(F_i \cap P) \setminus \Z^2 \neq \emptyset$, there exists a ray $r^j\in
R$ such that $p^j \in (F_i \cap P) \setminus \Z^2$. 
We will show that $\bar a^i \cdot r^j \neq 0$.
Suppose for the sake of contradiction that 
$\bar a^i \cdot r^j = 0$.  Let $y \in Y_i$.  By definition of $\Ny$,  $\bar
a^i \cdot (y - f) = 0$.  Since $r^j$ does not point to an integer point from
$f$, the vectors  $r^j$ and $y-f$ are not parallel.  Therefore, the system
$\bar a^i \cdot r^j = 0$, $\bar a^i \cdot (y-f) = 0$ has the unique solution
$\bar a^i =0$, which is a contradiction since we assumed $\bar a^i \neq 0$.
Hence, $\bar a^i \cdot r^j \neq 0$ and therefore $\psi_{B + \epsilon \bar A}
(r^j) \neq \psi_{B - \epsilon \bar A}(r^j)$.  Therefore, $\gamma(B)$ is a
strict convex combination of $\gamma(B + \epsilon \bar A)$ and $\gamma(B -
\epsilon \bar A)$. 
\end{proof}

\begin{figure}
\centering
\ifpdf
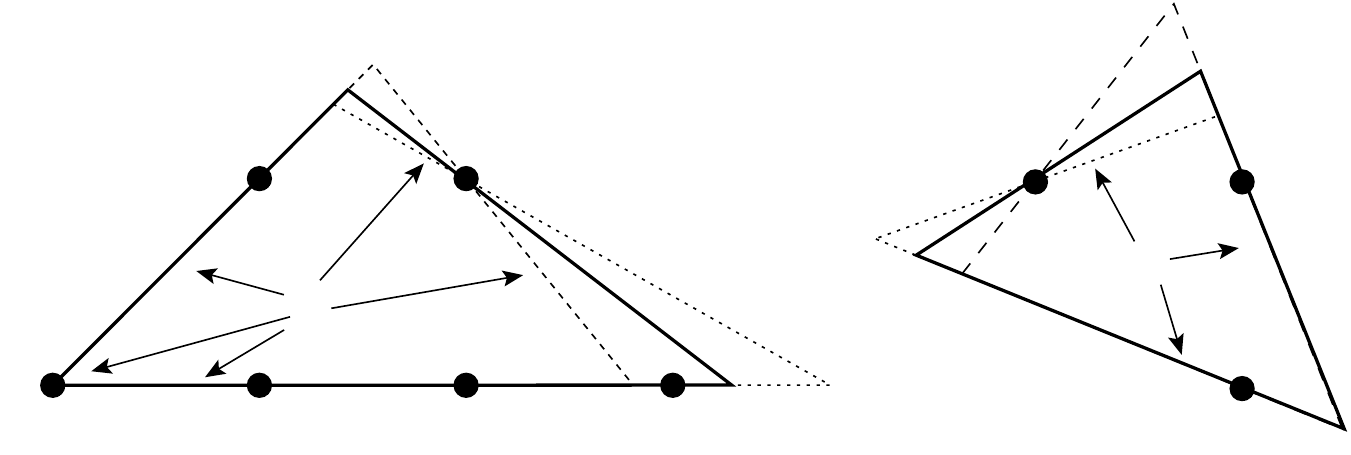
\else
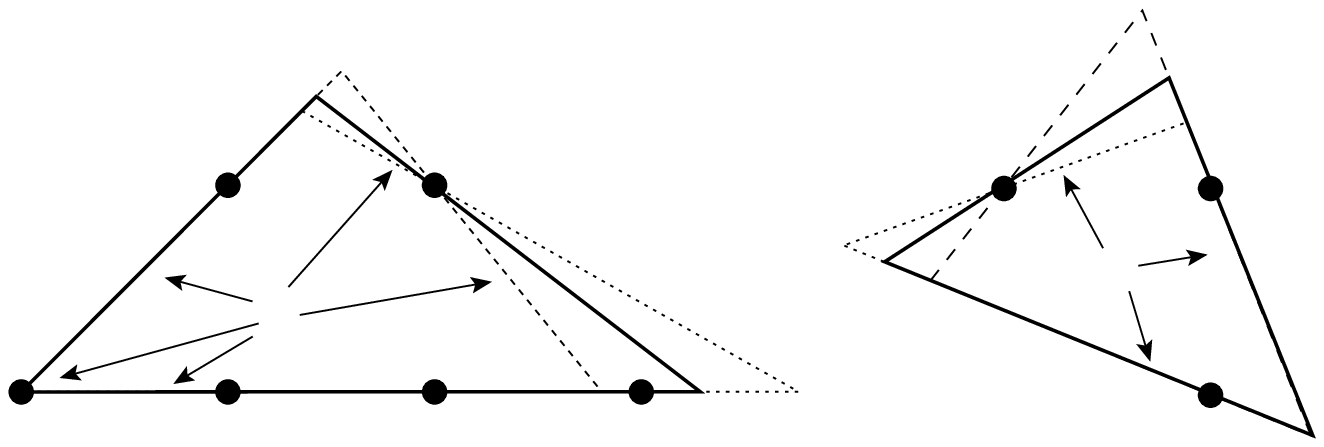
\fi
\caption{Single facet tilts: Tilting one facet of a polytope to generate new inequalities.  In
  both examples, there is a ray pointing to a non-integer point on the interior
  of the facet being tilted.  This ensures that the inequalities from the
  tilted sets are distinct, and therefore we see that $\gamma(B)$ is the strict convex combination of other points in $\Delta$. This is the assertion of
  Lemma~\ref{lemma:simple_tilts}.} 
\label{figure:simple}
\end{figure}

We will now apply this principle to obtain two lemmas for the case
$m=2$. 
The first simple application is to tilt a single facet of a polytope to
show that the corresponding inequality is a strict convex combination of other
inequalities, as shown in Figure
\ref{figure:simple}.  This is summarized in the following lemma. 
 
\begin{lemma}[Single facet tilt lemma]\label{lemma:simple_tilts}
Let $M(B)$ be a  lattice-free polytope for some matrix $B\in \R^{n\times 2}$. 
Suppose that $F_1(B) \cap \Z^2 = \{y^1\}$, $y^1 \in \relint( F_1(B))$, and $P
\cap F_1(B) \subset \relint(F_1(B))$, i.e., there are no ray intersections on
the lower-dimensional faces of $F_1(B)$, and  
$P\cap F_1(B) \setminus \Z^2\neq\emptyset$.  Then there exist $\bar a^1 \neq 0$ and $\epsilon> 0$
such that $\bar A = (\bar a^1; 0; \dots; 0)$ has the property that $\gamma(B)$ is a strict convex combination of $\gamma(B + \epsilon \bar A)$ and $\gamma(B - \epsilon \bar A)$  and $M(B + \epsilon \bar A)$ and $M(B - \epsilon \bar A)$ are lattice-free.
\end{lemma}

\begin{proof}
Let $Y_1 = \{y^1\}$ and $Y_i=Y(B)\cap F_i$, $i = 2, \ldots, n$, so 
that $\Y=(Y_1,\dots,Y_n)$ is a covering of the set $Y(B)$ of integer points in~$\M(B)$.  

Let $\bar a^1 \neq 0$ such that $\bar a^1\cdot (y^1 - f) = 0$ and observe that
$\bar A = (\bar a^1; 0; \dots; 0) \in \Ny$ because there are no
constraints~\eqref{eq:tilting-2} involving~$\bar a^1$, as there are no ray intersections on
the lower-dimensional faces of $F_1(B)$.  By Lemma~\ref{obs:dimension}, for some $\epsilon   > 0$, $\gamma(B)$ is a strict convex combination of $\gamma(B + \epsilon \bar A)$ and $\gamma(B - \epsilon \bar A)$ and $M(B + \epsilon \bar A)$ and $M(B - \epsilon \bar A)$ are lattice-free.
\end{proof}

We give another application of the perturbation arguments to Type 3 triangles and quadrilaterals, before moving on to more specific counting arguments in the next subsection.

\begin{lemma}
\label{lemma:corner_rays_almost}
Let $B \in \R^{n \times 2}$ be such that $M(B)$ is a Type 3 triangle ($n=3$) or a
maximal lattice-free quadrilateral ($n=4$). Let $\relint(F_i) \cap
\Z^2 = \{y^i\}$ and set $Y_i = \{y^i\}$; then $(Y_1, \ldots, Y_n)$ form a covering of $Y(B)$. If $P \not\subset \Z^2$ and $M(B)$ has fewer than $n$ corner rays,
then $\gamma(B)$ is a strict convex combination of $\gamma(B^1)$ and $\gamma(B^2)$, where $B^1$, $B^2 \in \R^{n \times 2}$ are matrices such that $M(B^1)$, $M(B^2)$ are both lattice-free.  
\end{lemma}
\begin{proof}
Let $Y_i = \{y^i\}$ for $i=1, \dots, n$.  Let $R' \supseteq R = \{r^1, \dots,
r^k\}$ such that $R'$ contains exactly $n-1$ corner rays pointing to different
vertices of $M(B)$ from $f$.  We examine the null space $\mathcal N(B, \Y,
R')$.  
With $n-1$ corner rays, $\mathcal N(B,\Y, R')$ is the set of matrices $A=(a^1;\dots;a^n)$ satisfying the following system of equations,
where, for convenience, we define $\y^i:= y^i - f$:
\begin{equation*}
a^i \cdot \y^i = 0  \ \text{for} \ i=1,\dots,n  \hspace{.5cm} \text{ and } \hspace{.5cm}    a^i \cdot r^i = a^{i+1} \cdot r^i \ \text{for} \  i = 1,\dots, n-1.
\end{equation*}
We have assumed, without loss of generality, that the corner rays and facets are numbered such that we have corner rays $r^i \in F_i \cap F_{i+1}$ for $i=1,\dots, n-1$. Note that $\bar y^i$ is linearly
independent from $r^i$ for $i=1, \dots, n-1$ and linearly independent from
$r^{i-1}$ for $i=2, \dots, n$, because $y^i$ lies in the relative interior of $F_i$ and the rays point to the vertices.   

There are $2n-1$ equations and $2n$ variables, so $\dim  \mathcal N(B,\Y, R') \geq 1$. Choose $\bar A = (\a^1; \dots ; \a^n) \in \mathcal N(B,\Y, R')\setminus\{0\}$. Notice that for $i=1,\dots, n-1$, if $\a^i = 0$, then $\a^{i+1}$ must satisfy $\a^{i+1} \cdot r^i =0$ and $\a^{i+1} \cdot \y^{i+1} = 0$, which implies that $\a^{i+1} = 0$, since $\bar y^{i+1}$ and $r^i$ are linearly independent. Similarly, for $i=2,\dots, n$, if $\a^i = 0$, then $\a^{i-1}$ must satisfy $\a^{i-1} \cdot r^{i-1} = 0$ and $\a^{i-1} \cdot \y^{i-1} = 0$, which implies that $\a^{i-1} = 0$. By induction, this shows that if $\a^i = 0$ for any $i=1,\dots, n$, then $\bar A = 0$, which contradicts our assumption.  Hence, $\bar a^i \neq 0$ for all $i=1,\dots, n$. 

Since $R' \supseteq R$, we have that $\mathcal N(B, \Y, R') \subseteq \mathcal N(B, \Y, R)$; thus, $\bar A \in \mathcal N(B, \Y, R) \setminus \{0\}$. Since $P \not\subset \Z^2$, there exists $r \in R$ that
does not point to an integer point.  Let $i \in I_B(r)$.  Since $\bar a_i \neq 0$, by applying Lemma~\ref{obs:dimension} with $\bar A$ we obtain $\epsilon > 0$ such that $\gamma(B)$ is a strict convex combination of $\gamma(B + \epsilon \bar A)$ and $\gamma(B - \epsilon \bar A)$ and $M(B + \epsilon \bar A), M(B - \epsilon \bar A)$ are lattice-free. 
\end{proof}



\subsection{Counting Arguments}\label{sec:counting}

An important ingredient in the proof of Proposition~\ref{prop:almost-extreme} is the following consequence of the Cook--Hartmann--Kannan--McDiarmid theorem on
the polynomial-size description of the integer hulls of polyhedra in fixed
dimension \cite{cook-hartmann-kannan-mcdiarmid-1992} combined with an algorithm by Hartmann~\cite{hartmann-1989-thesis} for
enumerating all the vertices, which runs in polynomial time in fixed
dimension.  
\begin{lemma}\label{rem:ray_cone}
Given two rays $r^1$ and $r^2$ in $\mathbb{R}^2$, 
we define the affine cone
$$C(r^1, r^2) = \{\,x \in \mathbb{R}^2 \st x = f + s_1r^1 + s_2r^2 \text{ for } 
s_1,s_2 \geq 0\,\}.$$ The number of facets and vertices of the integer hull
$$(C(r^1,r^2))_\IH = \conv(C(r^1, r^2)\cap\Z^2)$$ is bounded by a polynomial in the
binary encoding sizes of $f, r^1,r^2$. Furthermore, the facets and vertices of the integer hull can be enumerated in polynomial time in the binary encoding sizes of $f, r^1, r^2$.
\end{lemma}

In the following, the closed line segment between two points $x^1$ and $x^2$ will be denoted by $[x^1, x^2]$, and the open line segment will be denoted by $(x^1, x^2)$. 

\begin{lemma}\label{rem:cone_vertex}
Consider any lattice-free convex set $M(B)$ for $B\in\R^{n\times 2}$. 
Suppose there exist two rays $r^{j_1}, r^{j_2}$ such that the corresponding
ray intersections $p^{j_1}, p^{j_2}$ are distinct and lie on a facet $F$ of $M(B)$.  
\begin{enumerate}[(i)]
\item If $[p^{j_1}, p^{j_2}] \cap \Z^2 = \{y\}$ and $y \in (p^{j_1},
  p^{j_2})$, then $y$ is a vertex of the integer hull 
  $(C(r^{j_1}, r^{j_2}))_\IH$. Moreover, the line $\aff(F)$ is a supporting hyperplane
  for $(C(r^{j_1}, r^{j_2}))_\IH$, i.e., $(C(r^{j_1}, r^{j_2}))_\IH$ lies on
  one side of this line.
\item 
  If $[p^{j_1}, p^{j_2}] \cap \Z^2$ contains at least two points, then the line
  $\aff(F)$ contains a facet of the integer hull $(C(r^{j_1}, r^{j_2}))_\IH$.
\end{enumerate}
\end{lemma}
\begin{proof}
Suppose $H$ is the halfspace corresponding to $F$ that contains $f$. Then $H \cap C(r^{j_1}, r^{j_2}) \subset M(B)$ and since $M(B)$ does not contain any integer points in its interior, neither does $H \cap C(r^{j_1}, r^{j_2})$. Since we assume $[p^{j_1}, p^{j_2}] \cap \Z^2$ is non-empty and $p^{j_1}, p^{j_2}$ lie on the line defining $H$ (and also $F$), this line is a supporting hyperplane for $(C(r^{j_1}, r^{j_2}))_\IH$. 

If $[p^{j_1}, p^{j_2}] \cap \Z^2$ contains the single point $y$ and $y \in (p^{j_1}, p^{j_2})$, then clearly $y$\ is an extreme point of $(C(r^{j_1}, r^{j_2}))_\IH$. 

If $[p^{j_1}, p^{j_2}] \cap \Z^2$ contains two (or more) points, then the line defining $H$ (and also $F$) defines a facet of $(C(r^{j_1}, r^{j_2}))_\IH$.
\end{proof}

\begin{obs}[Integral ray intersections]\label{obs:int_intersections}
Let $R=\{r^1, \ldots, r^k\}$. Then there is a unique $\gamma \in \R^k$ such that $\gamma = \gamma(B)$,  $M(B)$ is a lattice-free convex set and $P(B, R) \subset \Z^2$. (Note that there may be multiple matrices $B \in \R^{n\times 2}$ yielding $\gamma$.)
\end{obs}\begin{proof}
  Let $\mathcal B$ be the family of matrices $B \in \R^{n \times 2}$ such that
  $P(B,R) \subset \Z^2$ and $M(B)$ is a lattice-free set.  Then $P(B_1,R) =
  P(B_2, R)$ for all $B_1,B_2\in\mathcal B$ because $M(B_1)$ and $M(B_2)$ are
  lattice-free.   Hence, $\gamma(B)$ is the same vector for all $B\in\mathcal B$.
\end{proof}

\begin{proposition}[Counting Type 3 triangles]\label{prop:counting_type3} 
There exists a finite subset $\Xi_3\subseteq\Delta_3$ such that for any $\gamma\in\Delta_3\setminus\Xi_3$, 
  there exist $\gamma^1, \gamma^2 \in \Delta$ such that $\gamma$ is a strict
  convex combination of $\gamma^1$ and $\gamma^2$. Moreover, the cardinality
  of $\Xi_3$ is bounded polynomially in the binary encoding sizes of $f, r^1,
  \dots, r^k$.  Specifically, $\Xi_3$ can be chosen as the set of all
  $\gamma(B)$ such that $M(B)$ is a Type 3 triangle and one of
  the following holds, where $R = \{r^1, \ldots, r^k\}$:\\ 
\textbf{Case a.} $P(B , R) \subset \Z^2$.\\
\textbf{Case b.} $M(B)$ has  three corner rays, that is, $\verts(B) \subseteq P(B, R)$.
\end{proposition}

\begin{proof}
\Step{1}\ 
Let $\gamma = \gamma(B)$ for some fixed $B\in \R^{3 \times 2}$ such that $M(B)$ is a Type 3  triangle. 
By Lemma~\ref{lemma:corner_rays_almost}, if $\gamma \in \Delta_3 \setminus \Xi_3$, then there exist $\gamma^1, \gamma^2 \in \Delta$ such that $\gamma$ is a strict convex combination of $\gamma^1$ and $\gamma^2$.  Therefore, we are left to determine the cardinality of $\Xi_3$.\medskip

\noindent \Step{2}\ 
We now bound the cardinality of $\Xi_3$ by considering each case. \smallskip

\noindent \textit{Case a.} Observation~\ref{obs:int_intersections} shows that
there is a unique $\gamma\in\Xi_3$ corresponding to this case.\smallskip

\noindent \textit{Case b.} $M(B)$ has three corner rays.  
First we pick any triplet of pairwise distinct rays, say $r^{j_1},r^{j_2},r^{j_3} \in R$, as the corner rays; there
are $O(k^3)$ such triplets. 
Because we are constructing a Type 3 triangle, there needs to be an integer
point~$y_3\in(p^{j_1},p^{j_2})$.  By Lemma~\ref{rem:cone_vertex}~(i), $y_3$ is
a vertex of $(C(r^{j_1},r^{j_{2}}))_\IH$.  By the same argument, there exist
integer points $y_2 \in (p^{j_1},p^{j_3})$ and $y_1 \in (p^{j_2},p^{j_3})$
that are vertices of $(C(r^{j_1},r^{j_{3}}))_\IH$ and
$(C(r^{j_2},r^{j_{3}}))_\IH$, respectively.  
By Proposition~\ref{prop:unique-triangle} in the Appendix, 
a triangle is uniquely determined by three corner rays and and one point on the
relative interior of each facet.
Thus, we can use a triplet of rays and a
vertex from each integer hull of the three cones spanned by consecutive rays
to uniquely define the triangle. These are polynomially bounded in number by
Lemma~\ref{rem:ray_cone}. 

Thus the number of elements of $\Xi_3$ corresponding to each case has a
polynomial bound, and the result is proved. 
\end{proof}

\begin{proposition}[Counting splits]\label{prop:counting_splits}
There exists a finite subset $\Xi_0\subseteq\Pi$ such that for any $\gamma\in\Pi\setminus\Xi_0$, 
  there exist $\gamma^1, \gamma^2 \in \Delta$ such that $\gamma$ is a strict
  convex combination of $\gamma^1$ and $\gamma^2$. Moreover, the cardinality
  of $\Xi_0$ is bounded polynomially in the binary encoding sizes of $f, r^1,
  \dots, r^k$. Specifically, $\Xi_0$ can be chosen as the set of all
  $\gamma(B)$ such that $M(B)$ is a maximal lattice-free split and one of the
  following holds, where $R = \{r^1, \ldots, r^k\}$:

\noindent\textbf{Case a.} $P(B, R) \subset \Z^2$.\\
\textbf{Case b.} There exists $j \in \{1, \ldots, k\}$ such that $r^j$ lies in the recession cone of the split.\\
\textbf{Case c.} $\#(\conv(P(B, R)\cap F_1) \cap \Z^2) \geq 2$.
\end{proposition}

\begin{proof}
\noindent \Step{1}\ 
Consider $\gamma(B) \in \Pi \setminus \Xi_0$, and so $M(B)$ is a maximal lattice-free split such that none of Case a, Case b, or Case c hold.
So we suppose that, possibly by exchanging the rows of~$B$, no
ray in $R$ lies in the recession cone of the split, $ P(B, R)\cap F_1
\setminus \Z^2 \neq \emptyset$, and $\#(\conv(P(B, R)\cap F_1) \cap \Z^2) \leq
1$.

We will first construct a lattice-free quadrilateral $M(\hat B)$ such that $\gamma(B) = \gamma(\hat B)$.
We will consider the sub-lattice of~$\Z^2$ contained in the linear space
parallel to $F_1$. We use the notation $v(F_1)$ to denote a primitive lattice
vector which generates this one-dimensional lattice. 
Choose $y^1 \in \Z^2$ such that $P(B,R) \cap F_1 \subset (y^1 - v(F_1), y^1 + v(F_1))$. 
%
Pick any $x^1, x^2 \in F_1$ such that  
$$(P(B, R) \cap F_1) \cup \{y^1\} \subsetneq (x^1, x^2) \subsetneq (y^1 - v(F_1), y^1 + v(F_1)). $$

We can assume that $F_2$ corresponds to the facet opposite $F_1$, with another exchange of the rows of $B$ if necessary. Next, choose distinct integer points $x^3, x^4 \in F_2 \cap
\Z^2$ such that $P(B, R) \cap F_2 \subset (x^3, x^4)$ and $f \in \intr(\conv(\{x^1, x^2, x^3, x^4\}))$.  
Now, let $\hat B = (\hat b^1;\hat b^2;\hat b^3;\hat b^4)\in \R^{ 4 \times 2}$ such that $M(\hat B) = \conv( \{x^1, x^2, x^3, x^4 \})$  and let $F_1(\hat B) = [x^1, x^2], F_2(\hat B) = [x^3, x^4], F_3(\hat B) = [x^1, x^3]$ and $F_4(\hat B) = [x^2, x^4]$ (see Figure \ref{figure:SplitTilts}~(a)).  
The set $M(\hat B)$ is a lattice-free quadrilateral with no corner rays.  By construction, the ray intersections are the same for $M(\hat B)$ and $M(B)$, i.e., $P(\hat B, R) = P(B, R)$, and therefore $\gamma(\hat B) = \gamma(B)$.  Furthermore, $F_1(\hat B) \cap P(B, R) \subset \relint(F_1(\hat B))$, that is, all ray intersections on $F_1(\hat B)$ are contained in its relative interior.     
Since $P(B, R) \cap F_1(\hat B) \not \subset \Z^2$ and is non-empty, by Lemma~\ref{lemma:simple_tilts}, $\gamma(\hat B) = \gamma(B)$ is a strict convex combination of $\gamma(\hat B + \epsilon \bar A)$ and $\gamma(\hat B - \epsilon \bar A)$, where $ \bar A = (\bar a^1; 0;0;0)$ and $\bar a^1 \neq 0$.   

Lastly, we need to show that there exist $A,A' \in \R^{3\times 2}$ such that $\gamma(A) = \gamma(\hat B + \epsilon \bar A)$ and $\gamma(A') = \gamma(\hat B - \epsilon \bar A)$, and $M(A), M(A')$ are lattice-free triangles. Since the cases are similar, we will just show that such a matrix $A$ exists. More concretely, we want to exhibit a matrix $A$ such that $P(A, R) = P(\hat B + \epsilon \bar A, R)$ and that $M(A)$ is lattice-free. As $\bar A$ comes from Lemma \ref{lemma:simple_tilts}, $\bar a^1 \cdot v(F_1) \neq 0$.  Suppose, without loss of generality, that $\bar a^1 \cdot v(F_1) > 0$.     Let $a^1 = \hat b^1 + \epsilon \bar a^1$,  $a^2 = \hat b^2$, and $a^3 = \hat b^3$ and $A = (a^1;a^2;a^3)$ (see Figure \ref{figure:SplitTilts} (b)).    Let $\alpha > 0$.  Then 
$$(\hat b^1 + \bar a^1)\cdot (y^1 + \alpha v(F_1) - f) 
=  (\hat b^1 + \bar a^1)\cdot (y^1 - f) + \alpha \hat b^1 \cdot v(F_1)+ \alpha \bar a^1\cdot v(F_1)   = 1 + 0 + \alpha \bar a^1 \cdot v(F_1) 
> 1,$$
 and therefore, $y^1 + \alpha v(F_1) \notin M(A)$ for all $\alpha > 0$.   
 Recalling that $M(B)$ is a split, it
 follows that $\intr(M(A))\setminus \intr(M(B)) \subseteq \intr(M(\hat B +
 \epsilon \bar A))$ and $\intr(M(A)) \cap \intr(M(B)) \subseteq \intr(M(B))$.
 Therefore, $\intr(M(A)) \subseteq \intr(M(\hat B +
 \epsilon \bar A)) \cup \intr(M(B))$.  Because of this inclusion, $M(A)$ is lattice-free.

\begin{figure}[p]
\ifpdf
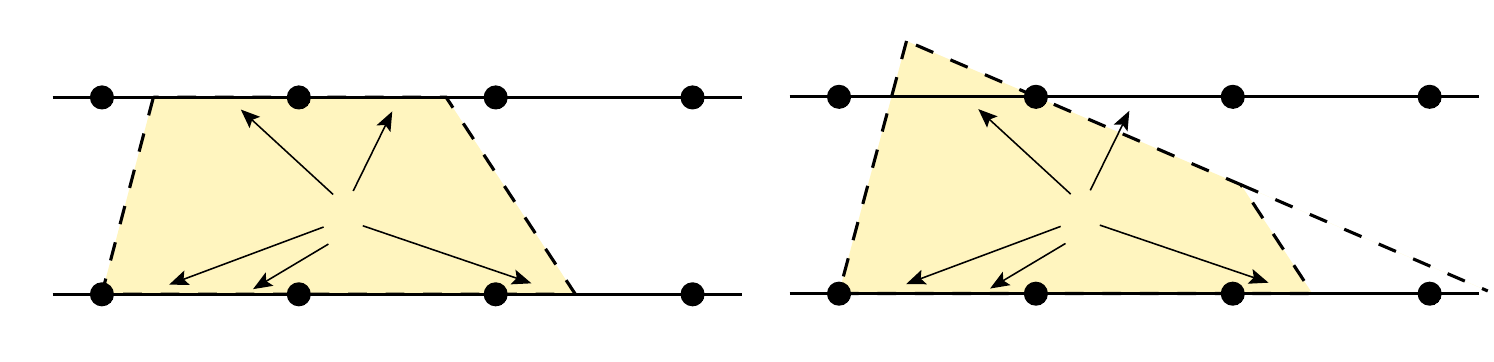
\else
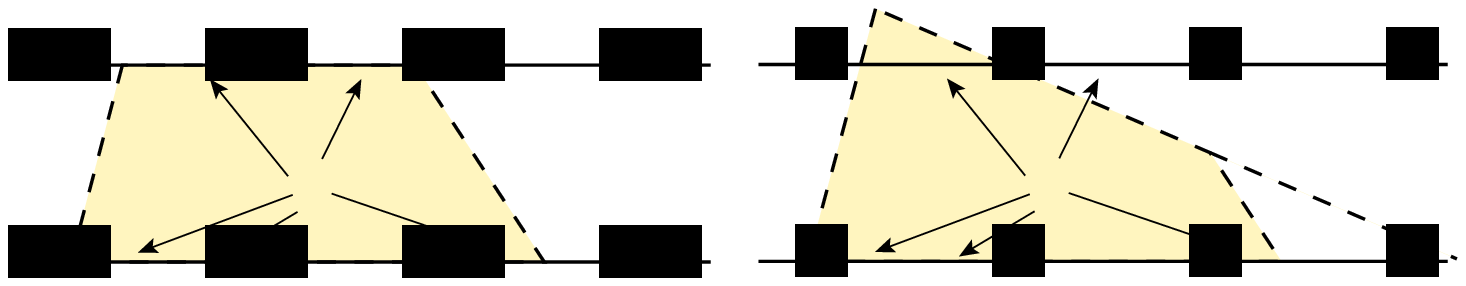
\fi
\caption{ Geometry of necessary conditions for splits in proof of Proposition \ref{prop:counting_splits}.  (a) depicts the lattice-free quadrilateral $M(\hat B)$ that is created such that $\gamma(B) = \gamma(\hat B)$.  (b) shows the lattice-free quadrilateral $M(\hat B + \epsilon \bar A)$ and the lattice-free triangle $M(A)$.  We see that $P(A, R) = P(\hat B + \epsilon \bar A, R)$ and hence $\gamma(A) = \gamma(\hat  B + \epsilon \bar A)$. }
\label{figure:SplitTilts}
\end{figure}
\medbreak

\noindent \Step{2}\ We now bound the cardinality of $\Xi_0$ by considering each case.  \\
\indent \textit{Case a.} $P(B, R) \subset \Z^2$.
Observation~\ref{obs:int_intersections} shows that 
there is a unique $\gamma\in\Xi_0$ corresponding to this case.\smallskip

The two remaining cases are illustrated in Figure~\ref{fig:count-splits}.\smallskip  
\begin{figure}[p]
\centering
\ifpdf
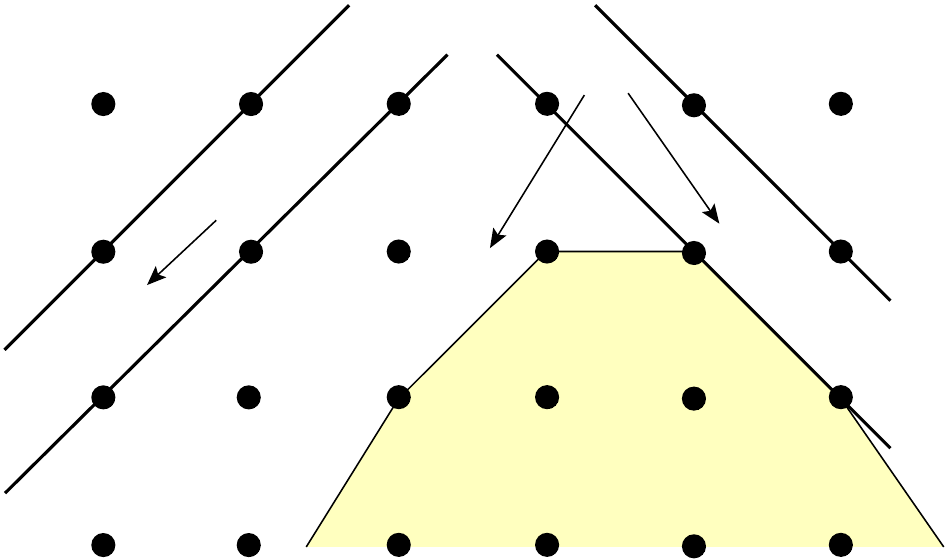
\else
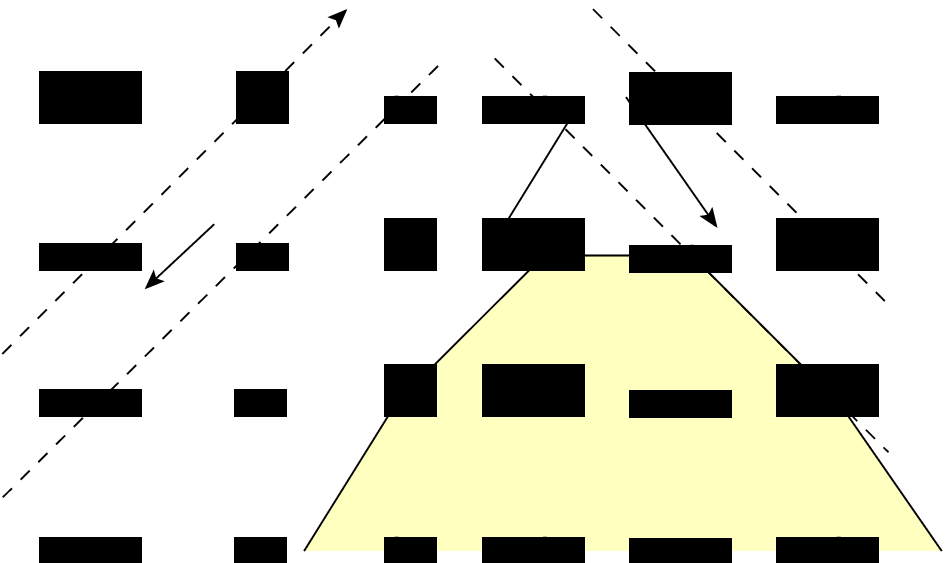
\fi
\caption{Step 3 of the proof of Proposition~\ref{prop:counting_splits}: Uniquely determining the split $M(B)$ for Cases (b) and (c).  Case (b) requires only the ray $r^j$ which is in the recession cone of the split, whereas Case (c) is determined by a facet of $(C(r^{j_1}, r^{j_2}))_\IH$.}\label{fig:count-splits}
\end{figure}

\indent \textit{Case b.}  A ray
direction~$r^j$ is parallel to the split.  There are at most $k$ such
ray directions, and thus at most~$k$ splits in this case.\smallbreak

\indent \textit{Case c.} There exist $p^{j_1}, p^{j_2}\in P(B, R)$  such that
$\# ([p^{j_1}, p^{j_2}] \cap \Z^2) \geq 2$, and therefore, the split must run
parallel to a facet of $(C(r^{j_1},r^{j_2}))_\IH$ by
Lemma~\ref{rem:cone_vertex}~(ii), of which there are only polynomially many. There
are only ${ k \choose 2}$ ways to choose two rays for this possibility.\smallbreak 

Since each case has a polynomial bound, we conclude that $\# \Xi_0$ is polynomially bounded as well.
\end{proof}

\begin{proposition}[Counting Type 1 triangles]\label{prop:counting_type1}
  There exists a finite subset $\Xi_1\subseteq\Delta_1$, such that for any $\gamma\in\Delta_1\setminus(\Xi_1\cup\Pi\cup\Delta_2)$, 
  there exist $\gamma^1, \gamma^2 \in \Delta$ such that $\gamma$ is a strict convex combination of $\gamma^1$ and $\gamma^2$
  or there exists $\gamma' \in \Delta$ such that $\gamma$ is dominated by $\gamma'$. Moreover, the cardinality of $\Xi_1$ is bounded polynomially in the binary encoding sizes of $f, r^1, \dots, r^k$.
Specifically, $\Xi_1$ is chosen as the set of all $\gamma(B)$ such that $M(B)$
is a Type~1 triangle, there exist distinct points $p^{j_1},
p^{j_2} \in \verts(B) \cap F_3 \cap P(B, R)$, i.e., $F_3$ has two corner rays, and one of the following holds:\\
\textbf{Case a.} $f \notin M(S_3)$.\\
\textbf{Case b.} $f \in M(S_3)$, and $P(B, R) \not\subset M(S_3)$.

\noindent Here $R=\{r^1, \ldots, r^k\}$ and $S_3 \in \R^{3\times 2}$ is a matrix such that $M(S_3)$ is a maximal lattice-free split with the property that one facet of $M(S_3)$ contains $F_3$ and $M(S_3) \cap \intr(M(B)) \neq \emptyset.$
\end{proposition}
 Figure~\ref{fig:type1} illustrates these two cases.
\begin{proof}
Consider any $\gamma\in\Delta_1\setminus(\Xi_1\cup\Pi\cup\Delta_2)$ and let $\gamma = \gamma(B)$ for some $B \in \R^{3 \times 2}$ such that $M(B)$ is a Type 1 triangle. For the sake of brevity, we use $P$ to denote $P(B, R)$ in the remainder of this proof. \old{We first establish the following claim.
\begin{claim}\label{claim:Type1} Suppose $\gamma\in \Delta_1\setminus (\Pi \cup \Delta_2)$, there {\em do not} exist $\gamma^1, \gamma^2 \in \Delta$ such that $\gamma$ is a strict convex combination of $\gamma^1$ and $\gamma^2$, and there {\em does not} exist $\gamma' \in \Delta$ such that $\gamma$ is dominated by $\gamma'$. Then $\gamma \in \Xi_1$.
\end{claim}}

\begin{figure}[p]%
\centering
\ifpdf
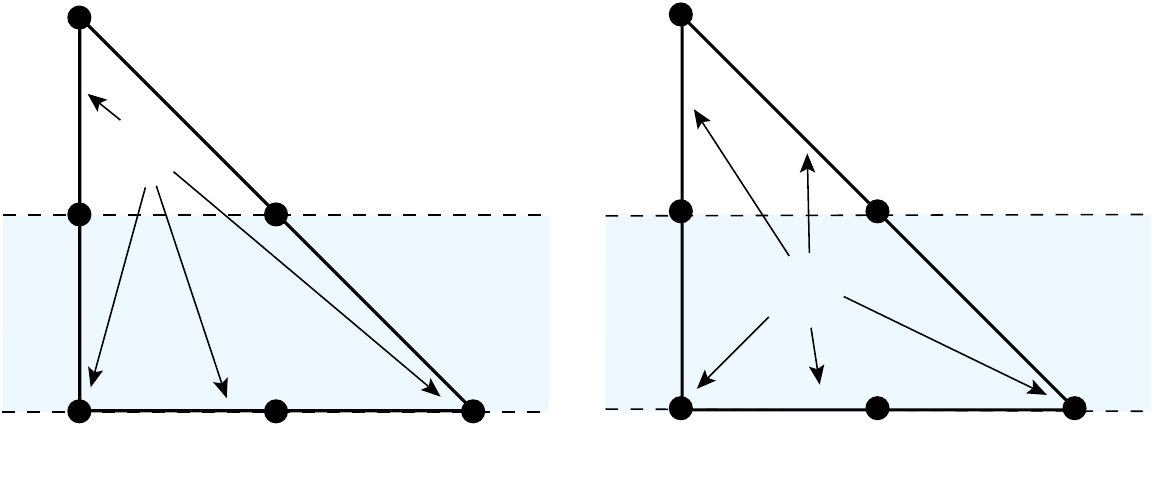
\else
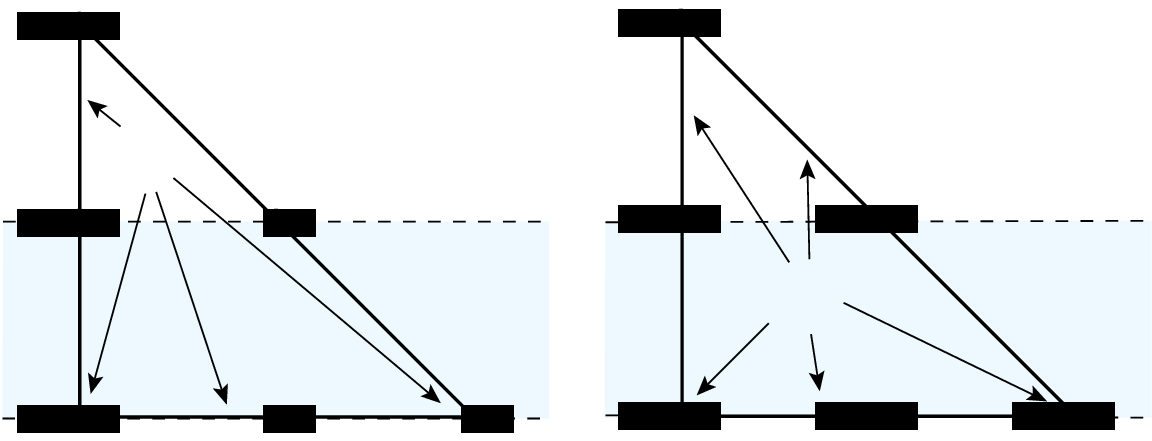
\fi
\caption{Proposition~\ref{prop:counting_type1}, Step 2: Uniquely determining
  Type 1 triangles that are in Cases a and b.  In both cases, the triangle is
  uniquely determined by $y^1, y^2$ and the ray intersections $p^{j_1},
  p^{j_2} \in \verts(B) \cap F_3$.}
\label{fig:type1}
\end{figure}%



\noindent \Step{1}\ Suppose $\#(\verts(B) \cap P) \leq 1$. 
This implies that some facet has no corner rays; without loss of generality, let this be $F_1$.  Thus $F_1\cap \verts(B) \cap P  = \emptyset$.
Let $y^1$ be the integer point in $\relint(F_1)$. 
If $P \cap F_1 \setminus \Z^2 = \emptyset$, then we can tilt $F_1$ slightly in either direction without making new ray intersections on $F_1$.  This creates a Type~2 triangle that realizes $\gamma$ (see
Figure~\ref{figure:simple2}), which contradicts the hypothesis that $\gamma \notin \Delta_2$.  Therefore, we can assume $P \cap \relint(F_1) \setminus\Z^2 \neq
\emptyset$.
Under this assumption, Lemma \ref{lemma:simple_tilts} shows that there exist $\gamma^1, \gamma^2 \in \Delta$ such that $\gamma$ is a strict convex combination of $\gamma^1$ and $\gamma^2$.
%

\smallbreak
\begin{figure}
\centering
\ifpdf
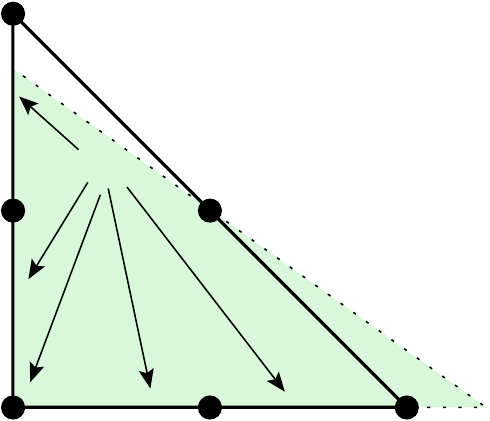
\else
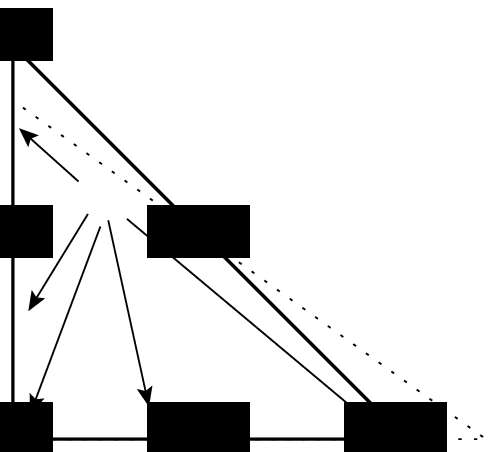
\fi
\caption{In the proof of Proposition~\ref{prop:counting_type1}, Step~1, a Type 1 triangle
  can be replaced by a Type 2 triangle (dotted) that gives the same inequality.%
}\label{figure:simple2}
\end{figure}
\noindent \Step{2}\ Suppose $\#(\verts(B) \cap P) \geq 2$. 
Thus there exist $p^{j_1}, p^{j_2} \in \verts(B) \cap P$, and we assume, by
possibly exchanging rows of~$B$, that $F_3$ is a facet 
containing $p^{j_1}, p^{j_2}$. Since $\gamma \not\in \Xi_1$, neither Case a
nor Case b holds. Thus $P \cup \{f\} \subset M(S_3)$. But then $\gamma$ is
dominated by or equal to~$\gamma(S_3)$. Since $\gamma \not\in \Pi$, $\gamma$ must be dominated by $\gamma(S_3)$. \smallskip

Thus we have shown that either $\gamma$ is a strict convex combination of $\gamma^1$ and $\gamma^2$ with $\gamma^1, \gamma^2 \in \Delta$ or there exists $\gamma' \in \Delta$ such that $\gamma$ is dominated by $\gamma'$.\smallbreak

\noindent \Step{3}\ 
We next bound the cardinality of $\Xi_3$, i.e., the number of Type 1 triangles
with two corner rays~$r^{j_1}$, $r^{j_2}$ on $F_3$ such that Case a or Case b
holds. There are $O(k^2)$ ways to choose the two corner rays on $F_3$, which
uniquely determine $F_3$.\smallbreak 

\noindent\textit{Case a.} Since $f$ does not lie in the split~$S_3$, the integer
points $y^1, y^2$ are uniquely determined.\smallbreak

\noindent\textit{Case b.} In this case $f$ lies in the split~$S_3$ and there exists a
ray intersection~$p^{j_3}$ outside the split.  After choosing the ray $r^{j_3}\in R$, 
the integer points $y^1, y^2$ are uniquely determined; there are at most~$k$
choices for~$r^{j_3}$.\smallbreak

Since the triangle is uniquely determined from the points $y^1, y^2$ and the
two corner rays $r^{j_1}$, $r^{j_2}$ on $F_3$, there are only polynomially
many Type 1 triangles which give vectors in~$\Xi_1$. 
\end{proof}

%


We next consider Type 2 triangles, which are the most complicated to handle.  For this, we first establish some notation and an intermediate lemma.

Consider a matrix $B=(b^1;b^2;b^3)\in \R^{3\times 2}$ such that $M(B)$ is a
Type 2 triangle.  For $i=1,2,3$, we denote $F_i = F_i(B)$. 
Without loss of generality, we assume that the facet containing multiple
integer points is $F_3$. 
We label the unique integer points in the relative interiors of $F_1$ and $F_2$ as $y^1$ and $y^2$, respectively. Within the case analysis of some of the proofs, we will refer to certain
points lying within splits.  For convenience, for $i=1,2,3$, we define $S_i\in \R^{3 \times 2}$ such that $M(S_i)$ is the maximal lattice-free split with the properties that one facet of $M(S_i)$ contains~$F_i$ and $M(S_i)\cap \intr(M(B)) \neq \emptyset$.

\begin{lemma}[Type 3 dominating Type 2 lemma]\label{lemma:Type3-domination}
Let $R=\{r^1, \ldots, r^k\}$. 
Consider any $B\in \R^{3\times 2}$ such that $M(B)$ is a Type 2 triangle. Denote the vertex $F_1\cap F_3$ by $v$ and let $y^3 \in F_3$ be the integer point in $\relint(F_3)$ closest to $v$. Suppose $P(B, R) \cap F_3$ is a subset of the line segment connecting $v$ and $y^3$. Then there exists a matrix $B'\in \R^{3\times 2}$ such that $M(B')$ is a Type 3 triangle and either $\gamma(B)$ is dominated by $\gamma(B')$, or $\gamma(B) = \gamma(B')$.
\end{lemma}

\begin{figure}
\centering
\ifpdf
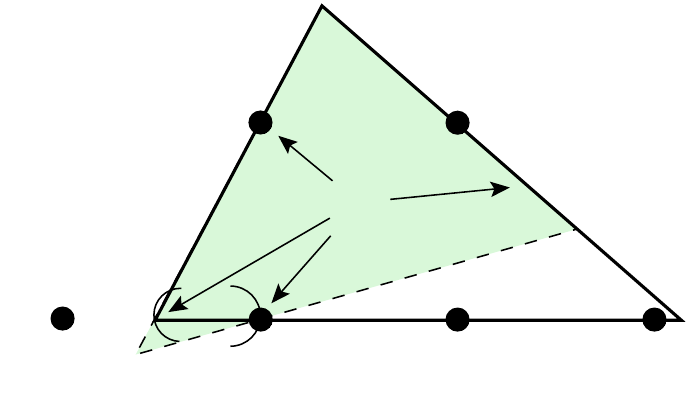
\else
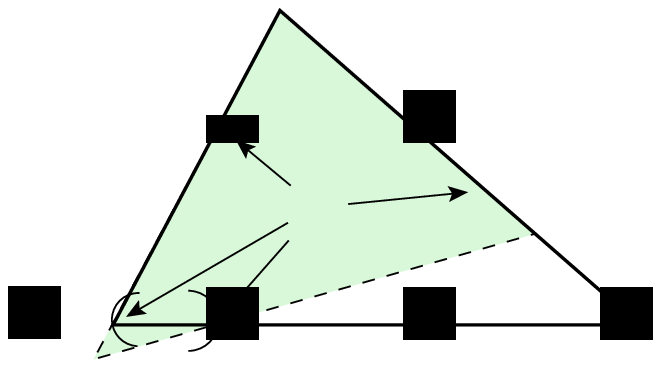
\fi
\caption{ We depict the geometry of Lemma \ref{lemma:Type3-domination} and show how we can change a facet of $M(B)$ to find a new matrix $B' \in \R^{3\times 2}$ such that $M(B')$ is a Type 3 triangle and $\gamma(B')$ either dominates $\gamma(B)$, or $\gamma(B) = \gamma(B')$.}
\label{fig-Lemma310}
\end{figure}

\begin{proof} 
  Choose $\bar a^3$ such that $\bar a^3 \cdot (y^3-f)
= 0$ and $\bar a^3 \cdot (y^3 - v) > 0$. Consider tilting $F_3$ by adding
$\epsilon \bar A  = \epsilon(0; 0; \bar a^3)$ to $B$ for some small enough
$\epsilon > 0$ so that the following two conditions are met. First,
$\epsilon$ is chosen small enough such that the set of integer points
contained in $M(B + \epsilon \bar A)$ is a subset of $Y(B)$; this can be done
by Lemma~\ref{lemma: S(B) full-dimensional}. Second, since $P(B,R) \cap F_3
\subset [y^3, v]$, we know that there is no corner ray pointing to $F_2 \cap
F_3$, and therefore we can choose $\epsilon$ small enough such that for all rays $r^j$ such that $2 \in I_B(r^j)$, $I_{B +\epsilon \bar A}(r^j) = I_B(r^j)$. This means that $p^j = p(B,r^j)\in F_2$ if and only if $p(B+\epsilon \bar A, r^j) \in F_2(B + \epsilon \bar A)$.
 
Now suppose $r^j$ is a ray pointing  from $f$ to $F_3$.  
Since $P(B,R) \cap F_3 \subset [y^3, v]$, we can describe~$r^j$ as the linear
 combination $r^j = 
 \alpha_1(y^3 - f) - \alpha_2(y^3 - v)$  for some $\alpha_1, \alpha_2 \geq
 0$.   Observe that $\psi_{B + \epsilon \bar A}(r^j) = \max\{b^1\cdot r^j,
 (b^3+\epsilon\bar a^3)\cdot r^j\}$
 and 
 \begin{equation}\label{eq:ineq}
   (b^3+\epsilon\bar a^3)\cdot r^j = b^3 \cdot r^j + \epsilon \bar a^3 \cdot
   \bigl(\alpha_1(y^3 - f) - \alpha_2(y^3 - v)\bigr) \leq b^3 \cdot r^j  = \psi_B(r^j).
 \end{equation}
 By definition of~$\psi_B(r^j)$, we also have $b^1\cdot r^j \leq
 \psi_B(r^j)$. 
 Therefore, $\psi_{B + \epsilon \bar A}(r^j) \leq \psi_B(r^j)$.
 
 Finally suppose  $r^j$ is such that $p^j \in P(B,R) \cap ((F_1 \cup F_2) \setminus F_3)$.
Then $\psi_{B + \epsilon \bar A}(r^j) = \psi_B(r^j)$ since by construction $I_B(r^j) = I_{B+\epsilon\bar A}(r^j)$ for all such rays. 
 Also, note that for any $y \in F_3\cap \Z^2$, $y = y^3 + \beta (y^3 - v)$ for some $\beta \geq 0$.  
 Therefore, 
 $$(b^3 + \epsilon \bar a^3) \cdot (y - f) \geq b^3 \cdot (y-f) = 1,$$
  meaning that none of these integer points are contained in the interior of
  $M(B + \epsilon \bar A)$. Since the set of integer points contained in $M(B
  + \epsilon \bar A)$ is a subset of $Y(B)$ and facets $F_1$ and $F_2$ were
  not tilted, $M(B + \epsilon \bar A)$ is lattice-free; in fact, it is a Type
  3 triangle. See Figure \ref{fig-Lemma310}. 

Thus, we can choose $B' = B +\epsilon\bar A$. The vector $\gamma(B)$ is dominated by $\gamma(B')$ when the inequality (\ref{eq:ineq}) is strict for some $r^j$; otherwise, $\gamma(B) = \gamma(B')$.
\end{proof}

\begin{proposition}[Counting Type 2 triangles]\label{prop:type2-abstractly}
There exists a finite subset $\Xi_2\subseteq\Delta_2$ such that if $\gamma\in\Delta_2\setminus(\Xi_2\cup\Delta_3\cup \Pi)$, then
  there exist $\gamma^1, \gamma^2 \in 
  \Delta$ such that 
  $\gamma$ is a strict convex combination of $\gamma^1$ and $\gamma^2$
  or there exists $\gamma' \in 
  \Delta$ such that $\gamma$ is dominated by $\gamma'$. Moreover, the cardinality of $\Xi_2$ is bounded
  polynomially in the binary encoding sizes of $f, r^1, \dots,
  r^k$. Specifically, $\Xi_2$ can be chosen as the set of all $\gamma(B)$
  such that $M(B)$ is a Type 2 triangle
  satisfying one of the following, where
  $P = P(B, R)$ and $F_i = F_i(B)$ such that $F_3$ is the facet of~$M(B)$
  containing multiple integer points:\\
\textbf{Case a.} $P \subset \Z^2$. \\
\textbf{Case b.} $P \not\subset \Z^2$ and there exist $ p^{j_1}\in P \cap F_1\cap
F_3$ (i.e., there is a corner ray pointing from~$f$ to~$F_1\cap
F_3$) and $ p^{j_2} \in P \cap F_3$ with $\#([p^{j_1}, p^{j_2}] \cap
\Z^2) \geq 2$. Moreover, if $P \cap \relint(F_2)\setminus\Z^2 \neq \emptyset$, then
there is a corner ray of $M(B)$ pointing to a vertex different from
$F_1\cap F_3$. Also, one
of the following holds:\\  
\indent \textbf{Case b\oldstylenums{1}.} $f \notin M(S_3)$.\\
\indent \textbf{Case b\oldstylenums{2}.} $f \in M(S_3)$ and $P \not\subset M(S_3)$.\\
\textbf{Case c.}  $P \not\subset \Z^2$ and there exist $p^{j_1}\in P\cap F_1\cap
F_3 \cap \Z^2$ (i.e., there is a corner ray pointing from~$f$ to~$F_1\cap
F_3\subset\Z^2$) and $p^{j_2} \in P\cap F_1$ with $\#([p^{j_1},p^{j_2}] \cap \Z^2) \geq
2$. Moreover, if $P\cap \relint(F_2) \setminus \Z^2 \neq \emptyset$, then
$p^{j_2}$ can be chosen such that $p^{j_2} \in F_1 \cap F_2$ (i.e., there is a corner ray pointing from~$f$ to~$F_1\cap F_2$).  
Also, one of the following holds:\\
\indent \textbf{Case c\oldstylenums{1}.} $f \notin M(S_1)$.\\
\indent\textbf{Case c\oldstylenums{2}.} $f \in M(S_1)$ and $P \not\subset M(S_1)$.  \\
\textbf{Case d.}  $P \not\subset \Z^2$, for all $i \in \{1,2,3\} $ and all $p^{j_1}, p^{j_2} \in P\cap F_i $  we have $\#([p^{j_1},p^{j_2}] \cap \Z^2 )\leq 1$, there
exists a corner ray pointing from~$f$ to $F_1\cap F_3$, and $F_1\cap F_3
\not\subset \Z^2$. Let $y^3, y^4 \in F_3$ such that $y^3$ is the closest integer point in $\relint(F_3)$ to $F_1 \cap F_3$, and $y^4$ is the next closest integer point.  Let $H_{2,4}$ be the half-space adjacent to $[y^2, y^4]$ and containing $y^1$. 

Then, we further have $P\cap (y^3, y^4) \neq \emptyset$. Moreover, one of the following holds:

\indent \textbf{Case d\oldstylenums{1}.} $f \notin H_{2,4}$,  there exists a
corner ray pointing from $f$ to $F_1 \cap F_2$.\\
\indent \textbf{Case d\oldstylenums{2}.} $f \notin H_{2,4}$,  there exists a ray pointing from $f$ through $(y^1, y^2)$ to $F_1$ and there are no rays pointing from $f$ to $\relint(F_2)\setminus \Z^2$.\\
\indent \textbf{Case d\oldstylenums{3}.} $f \in H_{2,4}$,  $P \not\subset H_{2,4}$, and there exists a corner ray pointing from $f$  to $F_1 \cap F_2$.

\end{proposition}

\begin{proof}
 
Consider any $\gamma\in \Delta_2$. By definition of $\Delta_2$, there exists a matrix $B\in \R^{3\times 2}$ such that $\gamma(B) = \gamma$ and $M(B)$ is a Type 2 triangle. Recall the labeling of the facets of $M(B)$ as $F_1, F_2, F_3$ with corresponding labels for the rows of $B$. For the sake of brevity, let $P$ denote the set $P(B, R)$ of the ray intersections in $M(B)$ for the rest of this proof. If $P \subset \Z^2$, we are in Case a. Therefore, in the remainder of the proof, we always assume $P \not\subset \Z^2$.

\subsubsection*{Proof steps 1 and 2: Dominated, convex combination, or Case d.}
Suppose $P \not\subset \Z^2$ and for all $i\in\{1,2,3\}$ and all $p^{j_1}, p^{j_2} \in
P \cap F_i$, we have $\#([p^{j_1},p^{j_2}] \cap \Z^2 )\leq 1$.  We will show  that at least one of the following occurs:
\begin{enumerate}[(i)]
 \item $\gamma(B)$ is dominated by some $\gamma' \in \Delta$, or is     a strict convex combination of some $\gamma^1, \gamma^2 \in \Delta$,  or there exists  a maximal lattice-free split or Type 3 triangle  $M(B')$ such that $\gamma(B') =~\gamma(B)$. \item Either Case d\oldstylenums{1}, Case d\oldstylenums{2}, or Case d\oldstylenums{3} occurs.
\end{enumerate}
   
First note that if there exist distinct $p^{j_1}, p^{j_2} \in P \cap \verts(B) \cap F_3$, then $[p^{j_1}, p^{j_2}] = F_3$ and $\#([p^{j_1}, p^{j_2}] \cap \Z^2) \geq 2$ since $F_3$ contains multiple integer poitns, and thus violating the assumptions.  Therefore $\#(P \cap \verts(B) \cap F_3) \leq 1$.


Recall that $F_3$ is the facet of $M(B)$ that contains at least 2 integer points and consider the sub-lattice of~$\Z^2$ contained in the linear space parallel to $F_3$. We use the notation $v(F_3)$ to denote the primitive lattice vector which generates this one-dimensional lattice and lies in the same direction as the vector pointing from $F_1 \cap F_3$ to $F_2 \cap F_3$. Since $\#([p^{j_1},p^{j_2}] \cap \Z^2 )\leq 1$ for all $p^{j_1}, p^{j_2} \in P \cap F_3$, there exists $y^3 \in F_3\cap\Z^2$ such that $P \cap F_3
\subset (y^3 - v(F_3), y^3 + v(F_3))$. Let $y^4 = y^3 + v(F_3)$ and let $y^5 = y^3 - v(F_3)$ and so $P \cap F_3$ is a subset of the {\em open} segment $(y^5, y^4)$. Note that $y^4, y^5$ are not necessarily contained in $F_3$. In Step~1 we will analyze the case with no corner rays on $F_3$ and see that we always arrive in conclusion~(i), whereas in Step \ref{step:Step2} we will analyze the case with a corner ray on $F_3$ and see that we will also arrive in conclusion~(i), except for the last step, Step \ref{step:Step2d},  where we arrive in conclusion~(ii).   
\smallbreak

\begin{figure}
\centering
\ifpdf
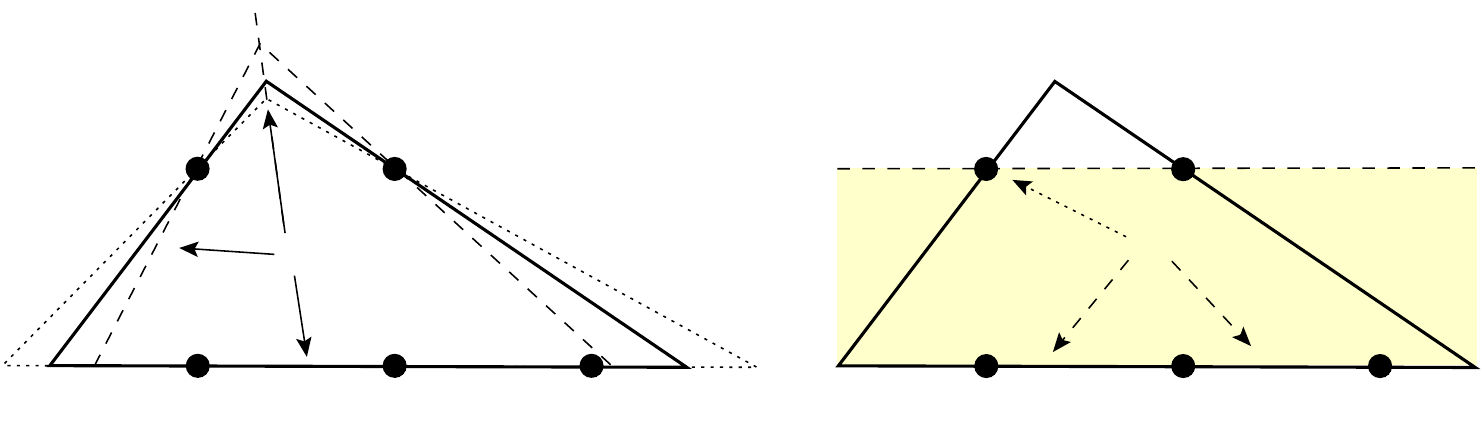
\else
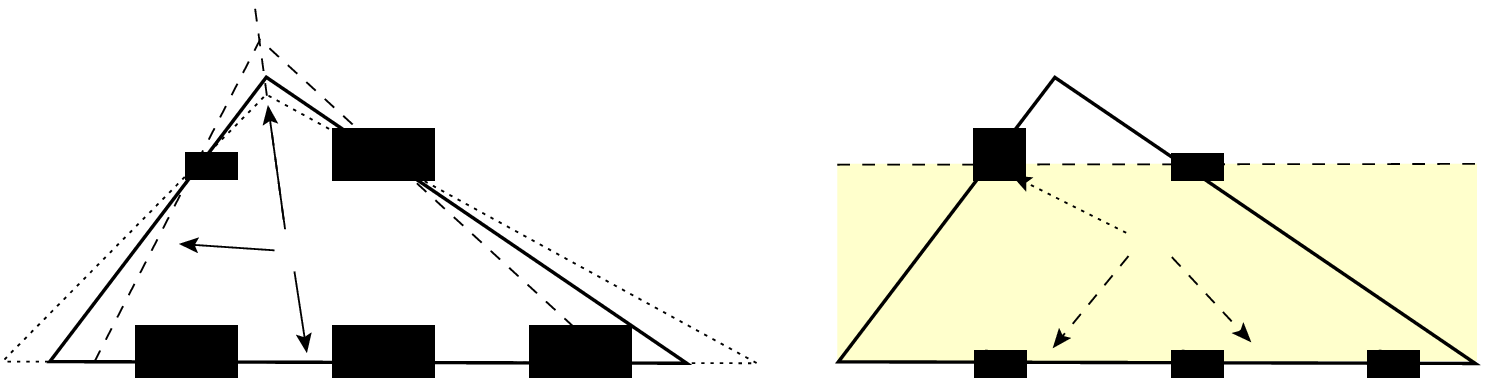
\fi
\caption{ Steps \ref{step:Step1a} and \ref{step:Step1b1}. The left figure
  depicts Step 1a where $P\cap (F_1 \cup F_2) \neq \emptyset$ and there are no
  corner rays on $F_3$, and shows that $\gamma(B)$ is a strict convex
  combination of other points in $\Delta$ by finding two lattice-free
  triangles through tilting the facets $F_1$ and $F_2$.     The right figure
  depicts Step~\ref{step:Step1b1}, where we find the split $M(S_3)$ such that
  $\gamma(S_3)$ dominates $\gamma(B)$.} 
\label{fig:1aIandII}
\end{figure}

\begin{figure}
\centering
\ifpdf
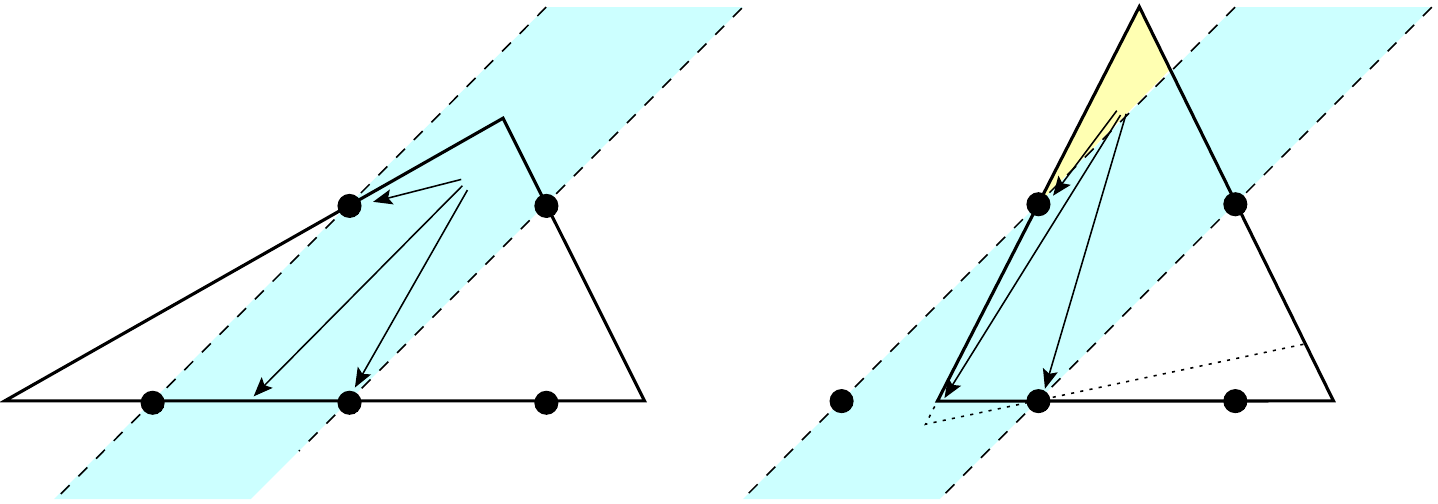
\else
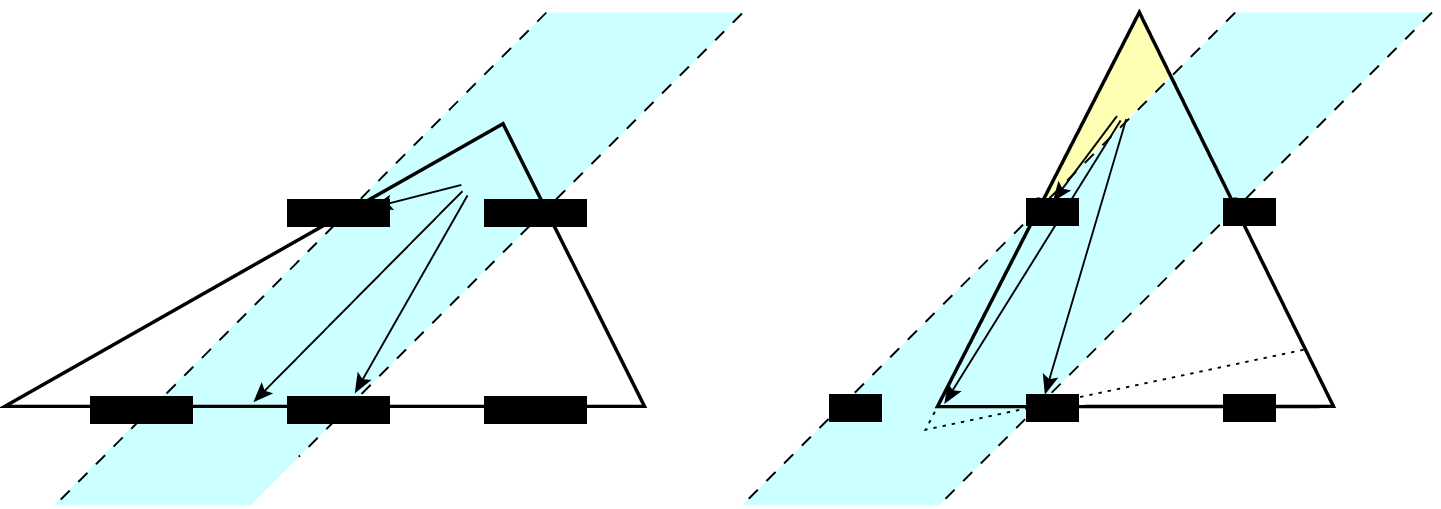
\fi
\caption{Step \ref{step:Step1b2}. In this step we consider $f \notin M(S_3)$.  On the left we see that $P\cap F_3 \subset [y^5,y^3]$ and both $y^5,y^3 \in F_3$, which allows $\gamma(B)$ to be dominated by $\gamma(S)$.  
This split satisfies $\gamma(S) \in \Delta$ because $f\notin M(S_3)$, meaning that $f$ is located somewhere on the top of the triangle, which is completely contained by $M(S)$.  
On the right, $y^5 \notin F_3$, which means that the split $S$ cuts off the top corner of the triangle, potentially leaving $f$ outside the split.  
This is problematic, so instead, we use Lemma \ref{lemma:Type3-domination} to create a new Type 3 triangle $M(B')$ such that $\gamma(B')$ dominates $\gamma(B)$.}
\label{fig:1atiltsIII}
\end{figure}

\noindent \Step{1}\label{step:Step1}\ Suppose  that $F_3$ has no corner rays, i.e.,
$\verts(B) \cap P \cap F_3 = \emptyset$.  \smallbreak


\indent  \Step{1a}\label{step:Step1a}  
Suppose  $P \cap (F_1 \cup F_2)\setminus \Z^2 \neq \emptyset$, i.e., there exists an index $j \in \{1,\dots , k\}$ such that $p^j \in (F_1 \cup F_2) \setminus (F_3 \cup \Z^2)$.  We will use the tilting space to show that $\gamma(B)$ is a strict convex combination of points in $\Delta$.  
Let $ Y_1 = \{y^1\}, Y_2 = \{y^2\}, Y_3 = F_3 \cap\Z^2$, $\Y = \{Y_1, Y_2,
Y_3\}$.  Let $r'$ be a corner ray that points from $f$ to the vertex $F_1 \cap
F_2$ and let $R' = \{r^1, \dots, r^k, r'\}$.  Since $R' \supseteq R$, we have $\mathcal N(B, \Y, R' )
\subseteq \mathcal N(B, \Y, R)$. 

We first count the equations that define $\mathcal N(B,\Y, R')$.  The equation
$a^3= b^3$ is implicit in $\mathcal T(B,\mathcal Y, R')$ since there are multiple integer points on $F_3$. There are two other equations for integer points on $F_1$ and $F_2$. The null space $\mathcal N(B, \Y, R')$ is given by the equations  

$$ 
a^1 \cdot (y^1-f) = 0, \qquad \ a^2 \cdot (y^2-f) = 0, \qquad\ a^1\cdot r' = a^2\cdot  r',\qquad a^3 = 0.
$$
Since $\mathcal{N}(B,\Y, R') \subset \R^{3\times2}$ and there are 5 equations
(note that $a_3 = 0$ is actually two equations), we see $\dim(\mathcal{N}(B,\Y,R))
\geq \dim(\mathcal{N}(B,\Y,R')) \geq 1$.
Let $\bar A = (\bar a^1; \bar a^2; \bar a^3) \in \mathcal N(B,\Y,R)\setminus
\{0\}$. Since $\Y$ is a covering of the lattice points in $M(B)$, by
Lemma~\ref{obs:dimension}, there exists $\epsilon > 0$ such that
$\gamma(B)= \frac{1}{2}\gamma({B + \epsilon \bar A}) + \frac{1}{2} \gamma({B -
  \epsilon \bar A})$ and $M(B \pm \epsilon \bar A)$ are both lattice-free. See
Figure \ref{fig:1aIandII} for these possible triangles.  

We next show that $\gamma(B - \epsilon\bar A) \neq \gamma(B +\epsilon\bar A)$. Observe that $\bar a^3 = 0$ since we are restricted by the equation $a^3 = b^3$. If $\bar a^1 = 0$, then $\bar a^2$ must satisfy $\bar a^2\cdot  r' = 0$ and $\bar a^2\cdot  (y^2-f) =0$, which implies that $\bar a^2 = 0$ since $r'$ and $y^2-f$ are linearly independent (since $y^2 \in \relint(F_2)$ and $r'$ points to a corner of $F_2$).  
Similarly, if $\bar a^2 = 0$, then $\bar a^1 = 0$. Since $\bar A \neq 0$, we must have both $\bar a^1, \bar a^2 \neq 0$.  
Then, since $\# Y_1 = \# Y_2 = 1$ and $1,2 \in I_B(r')$ and $p^j \in (F_1 \cup F_2) \setminus (F_3 \cup \Z^2)$, by Lemma~\ref{obs:dimension} $\gamma(B)$ is a strict convex combination of $\gamma(B + \epsilon \bar A)$ and $\gamma(B - \epsilon \bar A)$ and $M(B \pm \epsilon \bar A)$ are lattice-free triangles.

\smallbreak

\indent  \Step{1b}\label{step:Step1b}
 Suppose  $P \cap (F_1 \cup F_2)\setminus \Z^2 = \emptyset$, i.e., there only
 exist rays pointing  from $f$ to $F_3, y^1, y^2$. Therefore, $P \subset
 M(S_3)$.  We now analyze further subcases.

\indent  \Step{1b\oldstylenums{1}}\label{step:Step1b1}  
Suppose $f \in M(S_3)$.  Then $\gamma(B)$ is either dominated by or equal to $\gamma({S_3})$.  If $P \cap F_3 = \emptyset$, then $P\subset \{y^1, y^2\} \subset \Z^2$, which is a contradiction with the assumption of Step 1 that $P\not\subset\Z^2 $.
See Figure~\ref{fig:1aIandII}.

\indent  \Step{1b\oldstylenums{2}}\label{step:Step1b2} 
   Suppose that $f \notin M(S_3)$ and $P \cap F_3 \neq \emptyset$.
Suppose further that either $P\cap F_3 \subset [y^3,y^4]$ or $P\cap F_3
\subset [y^5, y^3]$.  These two cases are very similar, so we will just show
the case with $P \subset [y^5, y^3]$, which is illustrated in Figure~\ref{fig:1atiltsIII}.

If both $y^5, y^3 \in F_3$, then $\gamma(B)$ is dominated by $\gamma(S)$ where
$S$ is the maximal lattice-free split with its two facets along $[y^3, y^2]$
and $[y^5, y^1]$. This is because $P \not\subset \Z^2$ and so there exists a
ray intersection lying in the open segment $(y^5, y^3)$.  

Otherwise,  suppose $y^5 \notin F_3$. Note that $y^3 \notin \verts(B)$, because otherwise since $P\cap F_3 \neq \emptyset$, we find that  $P \cap F_3 \subset [y^5, y^3]\cap F_3 = \{y^3\}$, and therefore, $y^3 \in P$, contradicting the fact that there are no corner rays on $F_3$. Thus $y^3$ is the integer point in $\relint(F_3)$ closest to $F_1 \cap F_3$. This implies that $M(B)$ satisfies the hypotheses of Lemma~\ref{lemma:Type3-domination}. Hence there exists $B'$ such that $M(B')$ is a Type 3 triangle and either $\gamma(B)$ is dominated by $\gamma(B')$ or $\gamma(B) = \gamma(B')$.

\begin{figure}
\centering
\ifpdf
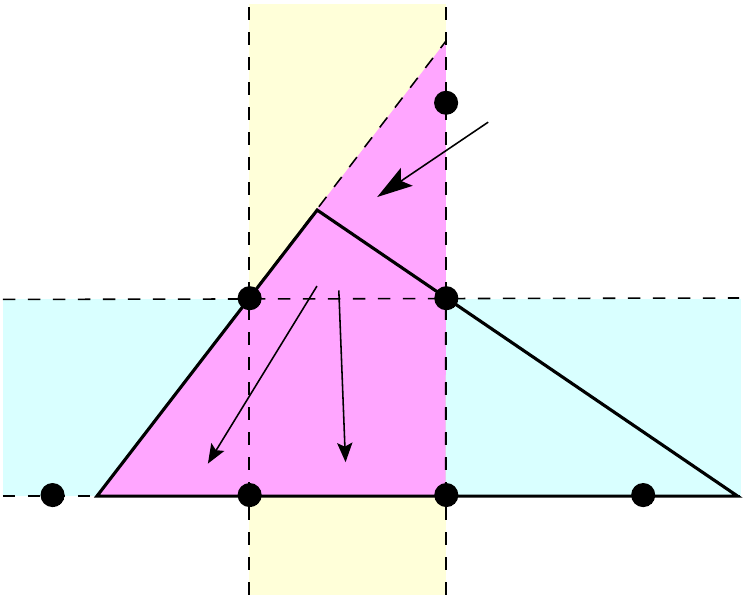
\else
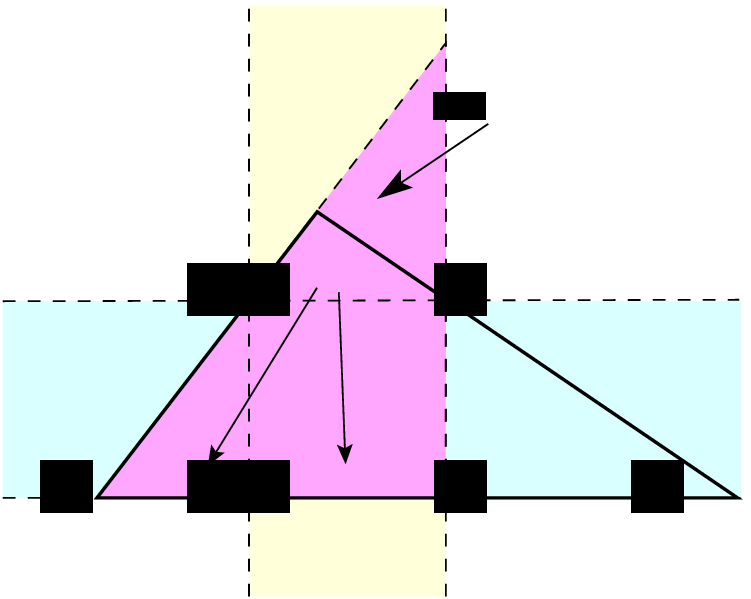
\fi
\caption{Step 1b\oldstylenums{3}. This figure demonstrates a new triangle $M(B')$ that yields the same inequality as $M(B)$.  
 }
\label{fig:figure4}
\end{figure}

\indent  \Step{1b\oldstylenums{3}}\label{step:Step1b3} 
  Suppose that $f \notin M(S_3)$, $P \cap F_3 \neq \emptyset$, and  $P \not\subset [y^3, y^4]$, $P \not\subset [y^5, y^3]$, i.e., $\conv(P \cap F_3)$ contains the integer point $y^3$ in its relative interior. 

We define a new triangle $M(B')$ by choosing its facets $F'_i = F_i(B')$, and thus, uniquely defining the matrix $B' \in \R^{3 \times 2}$.
 Let $F'_3 = F_3 \cap [y^5, y^4]$.  Next, let $F'_1$ and $F'_2$ be given by lines from the endpoints of $F'_3$ through $y^1$ and $y^2$, respectively. See Figure~\ref{fig:figure4}.
We will show $\gamma(B') = \gamma(B)$ and then that $\gamma(B')$ is a strict convex combination of other points in $\Delta$. 

\begin{claim}$M(B')$ is lattice-free.
\end{claim}
\begin{claimproof}  First note that $\#(F_3 \cap [y^5, y^4] \cap \Z^2) \geq 2 $ since $y^3$ is in the relative interior of $F_3$ and $F_3$ contains multiple integer points.  Without loss of generality, suppose $ y^4 \in F_3$.  Let $S\in \R^{3 \times 2}$ such that $M(S)$ is the maximal lattice-free split with facets on $[y^1, y^3]$ and $[y^2, y^4]$.  Then $M(B')$ is lattice-free since $M(B') \subset M(S) \cup M(B)$, which are both lattice-free sets.
\end{claimproof}
\begin{claim} $f\in M(B')$ and $\gamma(B') = \gamma(B)$.
\end{claim}
\begin{claimproof}  Since $f \in M(B) \cap M(S)$ and $M(B) \cap M(S) \subset
  M(B') \cap M(S)$, it follows that $f \in M(B')$. Recall that we are under
  the assumption that all rays point to $y^1, y^2$ or $F_3$. Moreover, all ray
  intersections on $F_3$ are contained in $(y^5, y^4)$ and hence the ray
  intersections are contained in $F'_3$. Therefore, the set $P(B', R)$ of ray intersections with respect to $M(B')$ is the same as $P$, and therefore $\gamma(B) = \gamma(B')$.
 \end{claimproof}
 Since $\conv(P \cap F_3)$ contains $y^3$ in its relative interior and $P\cap F_3$ is contained in the open segment $(y^5, y^4)$, we must have $P \cap \relint(F_3) \setminus \Z^2 \neq \emptyset$.  Furthermore, let $P' = P(B',R)$ be the set of ray intersections for $M(B')$, then $P' \cap F_3' = P \cap F_3$ by definition of $F'_3$.  Therefore, $P' \cap \relint(F_3') \setminus \Z^2 \neq \emptyset$.  Furthermore, $P' \cap \verts(B') \cap F_3' = \emptyset$ since $M(B)$ has no corner rays pointing to $F_3$ and there cannot exist rays pointing to $y^4$ or $y^5$ since $P\cap F_3$ is contained in the open segment $(y^5, y^4)$. Moreover, $\relint(F'_3) \cap \Z^2 = \{y^3\}$. Therefore, Lemma \ref{lemma:simple_tilts} can be applied to $M(B')$ with $F = F'_3$, which shows that $\gamma(B') = \gamma(B)$ is a strict convex combination of other points in $\Delta$.

 \medbreak

\noindent  \Step{2}\label{step:Step2} 
Suppose there is a corner ray on $F_3$ and, if necessary, relabel
the facets of $M(B)$ (and the rows of $B$) such that this corner ray points from~$f$ to the intersection $F_1 \cap
F_3$. Recall that we label the integer points $y^1\in  F_1$,  $y^2 \in
F_2$. Since $F_1\cap F_3\subseteq P$, observe that $y^3\in F_3$ (as defined in the paragraph before Step~1) is the closest integer point in $F_3$ to $F_1
\cap F_3$, and since $M(B)$ is a Type 2 triangle, we have $y^4\in F_3$.  Let $H_{2,4}$ be the half-space with
boundary containing the segment  $[y^2, y^4]$ and with interior
containing~$y^1$.  See Figure  \ref{fig:1ctiltsIII}. 
\smallbreak

\begin{figure}
\centering
\ifpdf
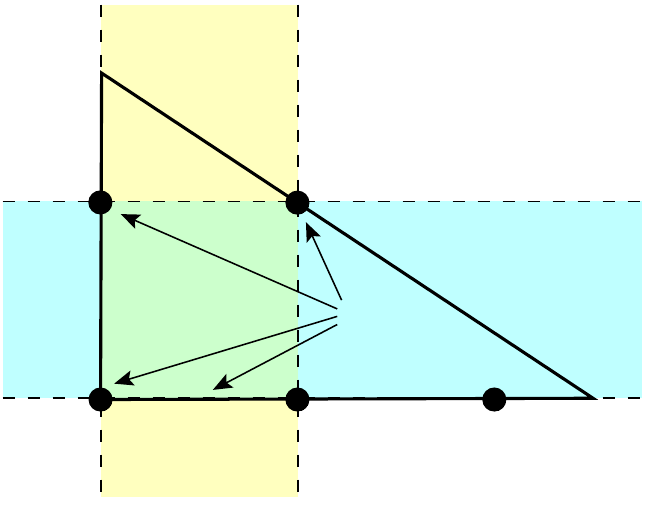
\else
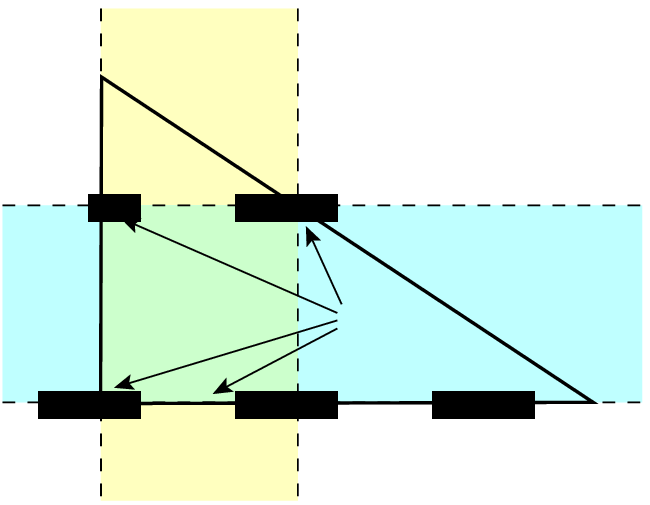
\fi
\caption{Step 2a. 
Depending on where $f$ is located in the triangle, at least  one of $\gamma(S_3)$ or $\gamma(S_1)$ either dominates or realizes $\gamma(B)$.  In this picture, $\gamma(S_3)$ realizes $\gamma(B)$, while $\gamma(S_1)\notin\Delta$ since $f \notin\intr(M(S_1))$.
  }
\label{fig:Type2Step1}
\end{figure}

\begin{figure}
\centering
\ifpdf
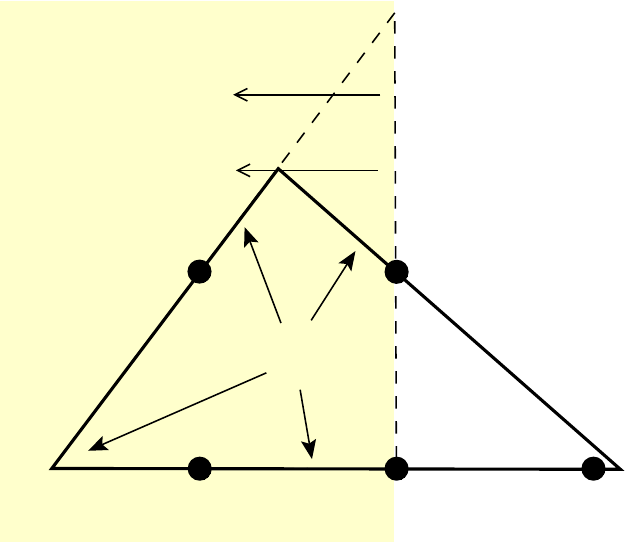
\else
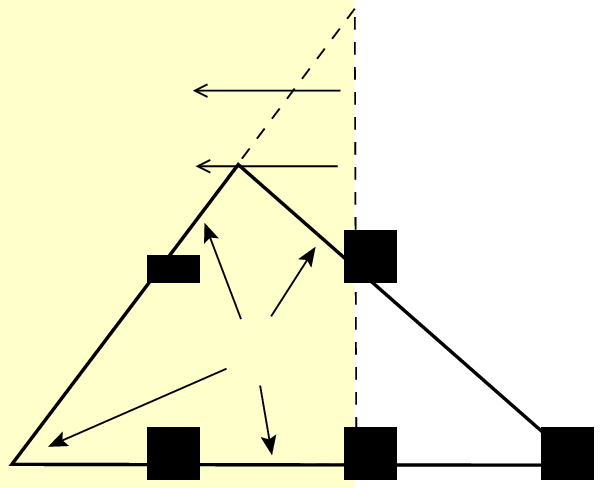
\fi
\caption{ Step 2b.  We show that if $P \cup \{f\} \subset H_{2,4}$, then we can create a different Type~2 triangle~$M(B')$ such that $\gamma(B')$  dominates~$\gamma(B)$.  If $\gamma(B) = \gamma(B')$, i.e., the ray pointing from $f$ to the facet~$F_2$ does not exist in the above picture, then the new triangle is a Type 2 triangle that was considered in Step 1a.   
}
\label{fig:1ctiltsIII}
\end{figure}

\smallbreak

\indent  \Step{2a}\label{step:Step2a} 
Suppose $y^3 \in F_1 \cap F_3$ and recall that there is a corner ray pointing
from $f$  to $F_1\cap F_3$. Note that this implies that $P \cap F_2 \cap
\verts(B) = \emptyset$, because $\#([p^{j_1}, p^{j_2}] \cap \Z^2) \leq 1$ for all
$p^{j_1}, p^{j_2} \in P \cap F_i$ for any $i \in \{1,2,3\}$ and including any corner
ray pointing from $f$ to $F_1 \cap F_2$ or $F_2 \cap F_3$ would contradict
this.  Therefore,  $P \cap F_2 \subset \relint(F_2)$. 

If $P \cap F_2 \setminus \Z^2  =P \cap \relint(F_2)\setminus \Z^2\neq \emptyset$, then $M(B)$ satisfies the hypotheses of Lemma~\ref{lemma:simple_tilts} with $F = F_2$ and $\gamma(B)$ is a strict convex combination of points in $\Delta$.

 If instead $P \cap F_2
\setminus \Z^2 = \emptyset$, then $P \cap F_2 \subset \{y^2\}$, and since $F_1 \cap F_3 \subseteq P$ and no two ray intersections within a facet can contain two integer points between them,  we must have $P \cap F_1 \subseteq [y^3, y^1]$ and $P \cap F_3 \subseteq [y^3, y^4]$.  Therefore, $P\subset \conv(\{y^1, y^2, y^3, y^4\})$. 
  Since $M(S_1) \cup M(S_3) \supset M(B)$, we must have 
$P\cup \{f\} \subset M(S_i)$ for $i=1$ or~$3$, and hence $\gamma(B)$ is either dominated by or equal to $\gamma(S_i)$.
See Figure \ref{fig:Type2Step1}.
\smallbreak

\indent  \Step{2b}\label{step:Step2b} 
Suppose $y^3 \notin F_1 \cap F_3$ and $P\cup \{f\} \subset H_{2,4}$.  Let $B' \in \R^{3 \times 2}$ such that $M(B')$ is the lattice-free Type 2 triangle with base $F'_3$ along $[y^2, y^4]$, the facet $F'_1$ given by the line defining $F_1$ for $M(B)$ and the facet $F'_2$ given by the line defining $F_3$ for $M(B)$. 
Let $P'$ be the set of ray intersections for $M(B')$.  See Figure \ref{fig:1ctiltsIII}.

If $P\cap \relint(F_2)\setminus \{y^2\} \neq \emptyset$, then $\gamma(B')$ dominates $\gamma(B)$ because $P\cup \{f\} \subset H_{2,4}$. 

Otherwise, $\gamma(B) = \gamma(B')$ and $P\cap \relint(F_2)\setminus \{y^2\} = \emptyset$. This implies that no ray points from $f$ to the corner $F'_1 \cap F'_3$ of $M(B')$. Recall that $P \cap F_3$ is a subset of the open segment $(y_5, y_4)$, therefore, $y^4 \not\in P$.
Hence, $M(B')$ has no corner rays on $F'_3$.  Also, since there exists a corner ray pointing  from $f$ to $F_1 \cap F_3 = F'_1 \cap F'_2$, we see that $P'\cap (F_1' \cup F_2') \setminus \Z^2 \neq \emptyset$.  Hence, $M(B')$ is a Type 2 triangle satisfying the conditions considered in Step~\ref{step:Step1a}, and using the same reasoning from that step, $\gamma(B') = \gamma(B)$ can be shown to be a strict convex combination of points in $\Delta$.
\smallbreak

\indent  \Step{2c}\label{step:Step2c} 
  Suppose $P \cap F_3 \subset [F_1 \cap F_3, y^3]$, $y^3 \notin F_1 \cap F_3$.
  Recall that $y^3$ is the closest integer point in $F_3$ to the corner $F_1\cap F_3$. Then $M(B)$ satisfies the hypotheses of Lemma~\ref{lemma:Type3-domination} and we can find a Type 3 triangle $M(B')$ such that $\gamma(B)$ is dominated by $\gamma(B')$ or $\gamma(B) = \gamma(B')$.

\begin{figure}[top]
\centering
\ifpdf
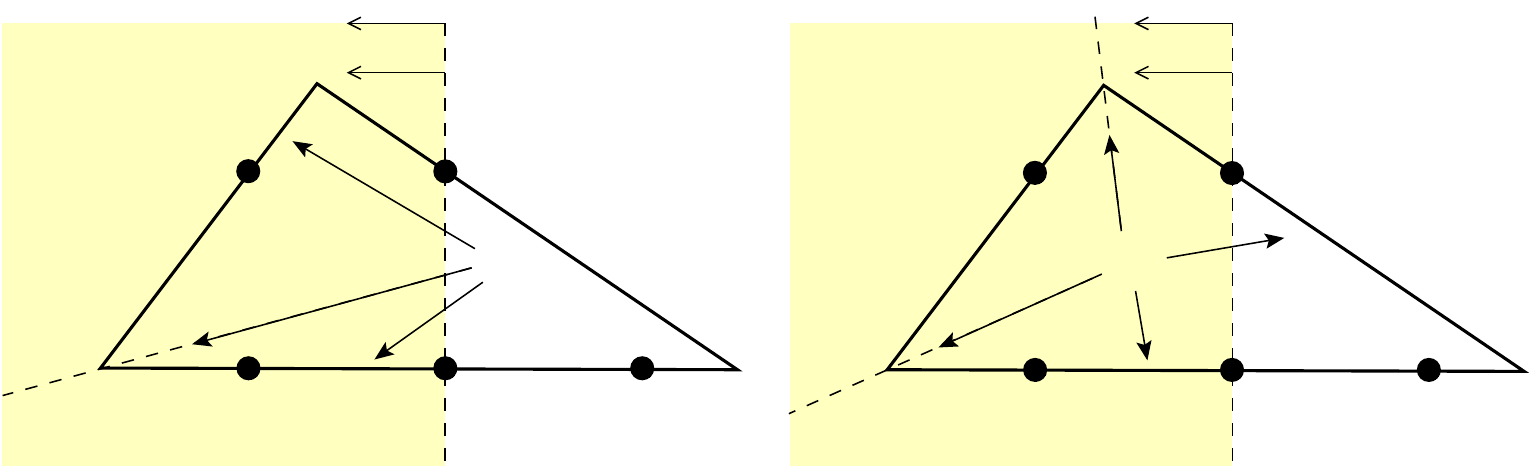
\else
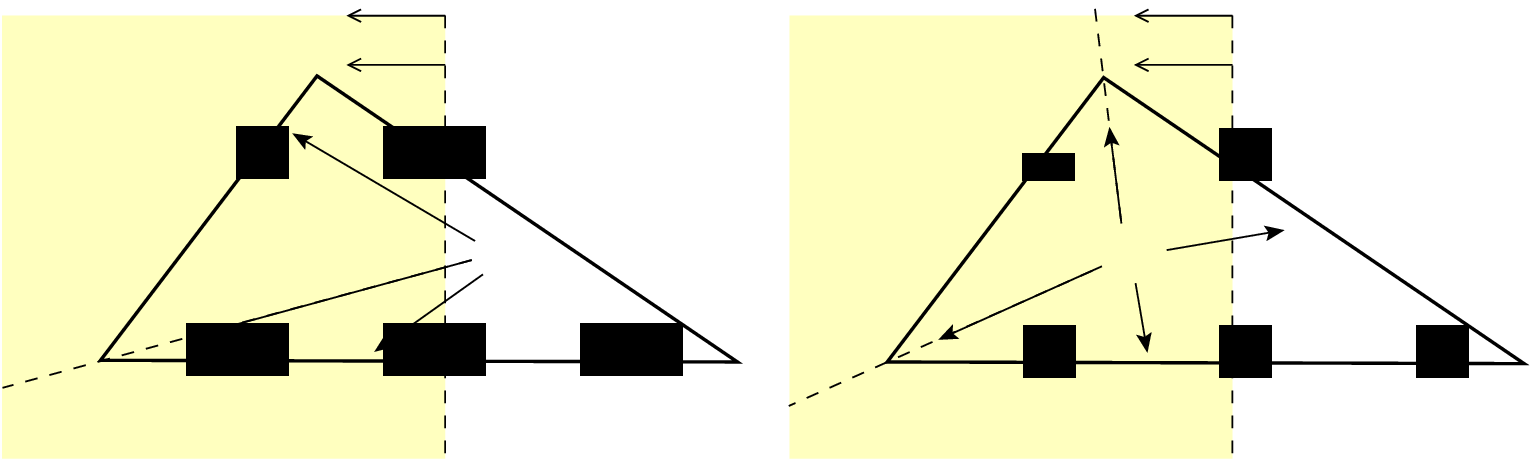
\fi
\caption{ Steps \ref{step:Step2d1}  and \ref{step:Step2d2}.  These steps arrive at  Cases d\oldstylenums{1}, d\oldstylenums{2},  and d\oldstylenums{3}, which are depicted here.  Other rays may also exist.  Cases d\oldstylenums{1} and d\oldstylenums{2}, because of their commonalities, are represented in one picture on the left, and Case d\oldstylenums{3} is on the right.}
\label{fig:Exceptions}
\end{figure}
\smallbreak

\indent  \Step{2d}\label{step:Step2d} 
         We can now assume that $P \not\subset \Z^2$ (the assumption for Steps 1 and 2), there is a corner ray pointing from $f$ to  $F_1 \cap F_3$ (assumption in beginning of Step 2), $y^3 \notin F_1 \cap F_3$ (negation of the assumption in Step 2a), $P \cup \{f\} \not\subset H_{2,4}$ (negation of the second assumption in Step 2b), and $(y^3, y^4) \cap P \neq \emptyset$ (negation of the assumption in Step 2c), which implies $y^3 \in \intr(\conv(P \cap F_3))$. Furthermore, we may be in one of the following subcases.
         
\indent  \Step{2d\oldstylenums{1}}\label{step:Step2d1} 
$f\notin H_{2,4}$.  This implies $f \in M(S_3)$ since $M(B)\setminus H_{2,4} \subset M(S_3)$. If $P$ is also contained in $M(S_3)$, then either $\gamma(B)$ is dominated by $\gamma({S_3})$, or $\gamma(B) = \gamma(S_3)$. Therefore, we assume $P \not\subset M(S_3)$, and so there must exist a ray $r$ pointing  from $f$ through $(y^1, y^2)$.  

Suppose there is a ray that points from $f$ to $\relint(F_2)$. If there is no
corner ray pointing from $f$ to $F_1 \cap F_2$, then $M(B)$ would satisfy the
hypotheses of Lemma~\ref{lemma:simple_tilts} with $F = F_2$ since no ray
points to $F_2\cap F_3$. Therefore, $\gamma(B)$ can be expressed as the strict
convex combination of points from $\Delta$. On the other hand, if there is a corner ray pointing to $F_1\cap F_2$, then we satisfy the statement of Case~d\oldstylenums{1}.  See Figure \ref{fig:Exceptions}.

Suppose now that no ray points from $f$ to $\relint(F_2)$. This implies that
the ray $r$ points from $f$ to $F_1$ through $(y^1, y^2)$ and $P \cap \relint(F_2)\setminus\Z^2 = \emptyset$. This is Case d\oldstylenums{2}.  

\indent  \Step{2d\oldstylenums{2}}\label{step:Step2d2} 
$f \in H_{2,4}$ and $P \not\subset H_{2,4}$. Because also $P\cap F_3 \subseteq
H_{2,4}$, this implies that $P \cap \relint(F_2) \setminus \Z^2 \neq \emptyset$. 

If there is no corner ray pointing from $f$ to $F_1\cap F_2$, then $M(B)$ satisfies the hypotheses of Lemma~\ref{lemma:simple_tilts} with $F = F_2$\ because there is no ray intersection in $F_2 \cap F_3$. Then $\gamma(B)$ can be expressed as the strict convex combination of points from $\Delta$. 

On the other hand, if there is a corner ray pointing from $f$ to $F_1\cap F_2$, then we satisfy the statement of Case d\oldstylenums{3}.

\bigskip

From the analysis of Steps 1 and 2, we conclude that
if $P \not\subset \Z^2$,  $\gamma$ is not dominated by any $\gamma' \in \Delta$, is not a strict convex combination of any $\gamma^1, \gamma^2 \in \Delta$, and there does not exist a maximal lattice-free split or Type 3 triangle $M(B')$ such that $\gamma(B') = \gamma$, then one of the following holds:
\begin{enumerate}
\item[(i)] There exist ray intersections $p^{j_1}, p^{j_2} \in P \cap F_i$ with $\#([p^{j_1},p^{j_2}] \cap \Z^2) \geq 2$ for some $i \in \{1,2,3\}$.
\item[(ii)] We are in Case d.
\end{enumerate}

\subsubsection*{Proof steps 3 and 4: Remaining cases.}
We now assume that $\gamma = \gamma(B)$ is not dominated by any $\gamma' \in \Delta$, is not a strict convex combination of any $\gamma^1, \gamma^2 \in \Delta$,  and  there does not exist a maximal lattice-free split or Type~3 triangle $M(B')$ such that $\gamma(B') = \gamma(B)$, and we are not in Case~d, and we are not in Case~a (so $P \not\subset \Z^2$). 
Therefore, from our previous analysis, there exist ray intersections $p^{j_1}, p^{j_2} \in P \cap F_i$ with $\#([p^{j_1},p^{j_2}] \cap \Z^2) \geq 2$ for some $i \in \{1,2,3\}$. We will show that either Case b\oldstylenums{1}, b\oldstylenums{2}, c\oldstylenums{1}, or c\oldstylenums{2} occurs. In Step 3 below, we analyze the case when $i=3$, and in Step~4, we analyze the case when $i=1$ or $i=2$. \medbreak

\noindent  \Step{3}\label{step:Step3} 
Suppose $P \not\subset \Z^2$ and there exist  $p^{j_1}, p^{j_2} \in P \cap
F_3$ with $\#([p^{j_1},p^{j_2}] \cap \Z^2) \geq 2$. Let $p^{j_1},p^{j_2}$ be such that $P \cap F_3 \subseteq [p^{j_1},p^{j_2}]$.\smallbreak

\begin{figure}
\centering
\ifpdf
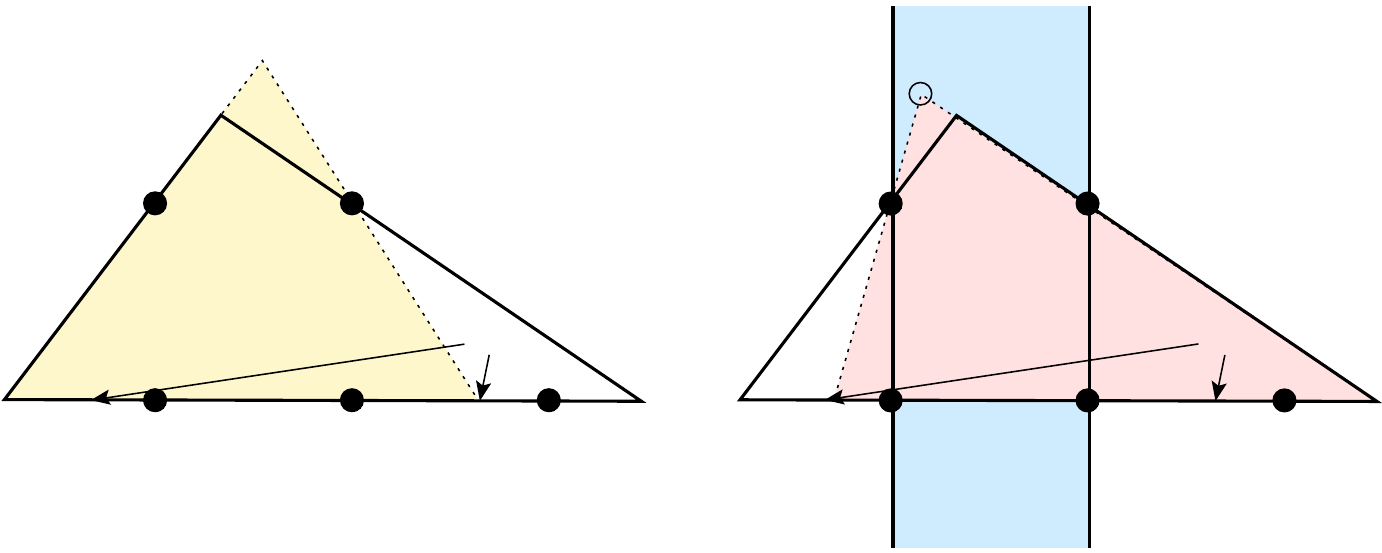
\else
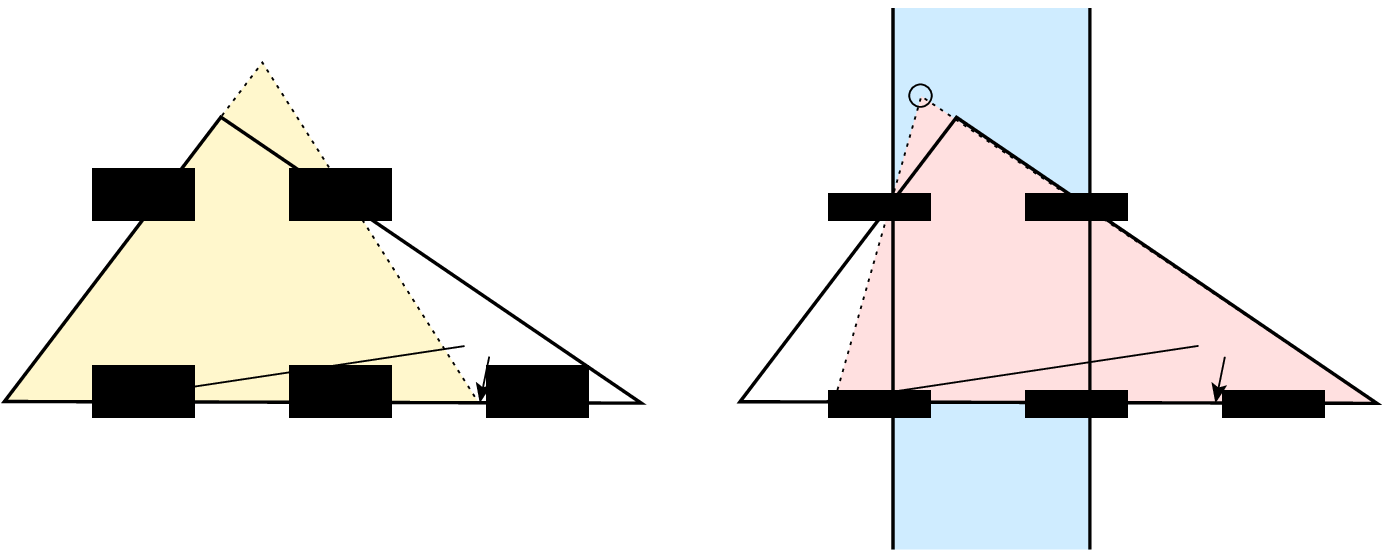
\fi
\caption{Step~3a.  Either
  $F_1$ or $F_2$ is tilted to give a new triangle~$M(B')$ (dotted).   
  (a)~Here $F_2$ cannot be used because tilting would remove $f$ from the
  interior.
  (b)~Instead, $F_1$ needs to be used.}
\label{fig:type2-step2a}
\end{figure}
\indent \Step{3a}\label{step:Step3a} 
We first show that there exists a matrix $B'$ such that $M(B')$ is a lattice-free Type 2 triangle that has a corner ray on $F_3$, and $\gamma(B) = \gamma(B')$.\smallbreak

If either $p^{j_1}$ or $p^{j_2}$ is a vertex of $M(B)$, then we let $B' = B$ and move to Step \ref{step:Step3b}. 

We now deal with the case that $p^{j_1}, p^{j_2} \notin \verts(B)$, i.e., there are no corner rays pointing from $f$ to $F_3$.

\begin{claim}
\label{claim:gamma}
If there exists a ray $r^j \in R$ such that $p^j \in F_1 \cup F_2 \setminus (F_3 \cup \Z^2)$, then $\gamma(B)$ is a strict convex combination of other points in $\Delta$.
\end{claim}
\begin{claimproof}
We define $\Y = (Y_1,
Y_2, Y_3)$ as $Y_1 = \{y^1\}$, $Y_2 = \{y^2\}$ and $Y_3 = F_3 \cap
\Z^2$. Hence, $\Y$ is a covering of $Y(B)$.  Define a new ray $r'$ to be a corner ray pointing from $f$ to the intersection $F_1 \cap F_2$ and let $R' = \{r^1, \dots, r^k, r'\}$.  Then $\mathcal{N}(B,\Y, R') \subseteq \Ny$.  

Since there are no corner rays pointing to $F_3$, there is only one independent equation
coming from a corner ray condition in the system
defining~$\mathcal{N}(B,\Y, R')$. The integer point sets $Y_1$ and $Y_2$ each contribute one
equation. Since $Y_3$ contains two integer points, it contributes a system of
equalities involving $a^3$ with rank $2$. Therefore, $\dim \mathcal{N}(B,\Y, R') =
6 - 5 = 1$.  

Note that the equations from $Y_3$ impose that $\bar a^3 = 0$.   Therefore, either $\bar a^1 \neq 0$ or $\bar a^2 \neq 0$.  Since $p^j \in \relint(F_i) \setminus \Z^2$ for either $i=1,2$, we have that $i \in I_B(r^j)$ for either $i=1,2$.  Then, since $\#Y_1 = \#Y_2 = 1$, by Lemma~\ref{obs:dimension}, there exists $\epsilon > 0$ such that $\gamma(B)$ is a strict convex combination of $\gamma({B + \epsilon \bar A})$ and $\gamma(B - \epsilon \bar A)$ and $M(B + \epsilon \bar A)$ and $M(B - \epsilon \bar A)$ are lattice-free.
\end{claimproof}



We now show that either $F_1$ or $F_2$ can be tilted to create a new Type 2
triangle $M(B')$ that has a corner ray pointing to the facet that contains
multiple integer points, $F_3(B')$, and $\gamma(B') = \gamma(B)$.  
Let $y^1, y^2$ be the integer points in $\relint(F_1)$, $\relint(F_2)$, respectively.
Since, by Claim~\ref{claim:gamma}, there are no rays pointing to $F_1 \cup F_2
\setminus \Z^2$, tilting $F_i$ with fulcrum~$y^i$ does not change $\gamma(B)$
unless $f$ is no longer on the interior of the perturbed set.      
For $i=1, 2$, consider changing $F_i$ to
now lie on the line through $p^{j_i}$ and $y^i$. At most one of these facet tilts puts $f$ outside the perturbed set, and therefore, at least one of these changes is possible.  This is illustrated in
Figure~\ref{fig:type2-step2a}. We can assume that the
tilt of facet~$F_1$ is possible  (by exchanging the rows of $B$, if necessary).  Let $B' \in \R^{3 \times 2}$ such that $M(B')$  is the new set after this tilting operation. The facets $F_i' = F_i(B')$ of $M(B')$ are chosen such that $F_1'$ corresponds to the new tilted $F_1$ and $F_2', F_3'$ correspond to $F_2, F_3$ respectively. 

We claim that $M(B')$ is lattice-free.  To
see this, we will show that $\intr(M(B'))$ is a subset of the union of
interiors of two lattice-free sets.  Let $y^3, y^4\in [p^{j_1}, p^{j_2}] \cap \Z^2$ be distinct integer points
adjacent to each other.  Then consider the maximal lattice-free split $M(S)$,
where $S \in \R^{2\times 2}$,  whose facets contain the segments $[y^3, y^1]$ and
$[y^4, y^2]$, respectively.  
Since $[y^3, y^4] \subset (p^{j_1}, F_2 \cap F_3)$ is a strict
subset, the former and new intersections $F_1 \cap F_2$ and $F_1' \cap F_2'$
are contained in the split $M(S)$.
Observe that $\intr(M(B')) \cap M(S_3) \subseteq \intr(M(B))$ and 
$\intr(M(B')) \setminus M(S_3) \subseteq \intr(M(S))$. 
Hence $\intr(M(B')) \subset \intr(M(B)) \cup \intr(M(S))$.  
Therefore $M(B')$ is lattice-free since $M(S)$ and $M(B)$ are both lattice-free. \smallbreak 

Therefore, we have shown that $\gamma(B) = \gamma(B')$ where $M(B')$ is a
Type~2 triangle that has a corner ray pointing to the facet $F'_3$.  Note that
since $\gamma(B) = \gamma(B')$, the sets of ray intersections coincide, that
is, $P(B', R) = P$.

\indent  \Step{3b}\label{step:Step3b} 
After Step \ref{step:Step3a}, we now focus on the new triangle $M(B')$ with facets $F_i' = F_i(B')$ for $i=1, 2,3$ that has a corner ray pointing from $f$ to the vertex at $F'_1 \cap F'_3$.   
We will show that the conditions of Case b are satisfied. 

If $P(B', R)\cap \relint(F'_2) \setminus \Z^2 \neq \emptyset$ and there are no corner rays on $F'_2$, then 
$M(B')$ satisfies the hypotheses of
Lemma~\ref{lemma:simple_tilts} and $\gamma(B) = \gamma(B')$ could be expressed
as a strict convex combination of points in $\Delta$.  Therefore, if $P(B', R)\cap \relint(F'_2) \setminus \Z^2 \neq \emptyset$, then there must be a corner ray pointing from $f$ to $F'_2$ (and thus pointing to a vertex different from $F'_1\cap F'_3$).


Hence, the conditions of Case b are
met for $B'$ (instead of~$B$). Furthermore, if $P\cup \{f\}\subset M(S_3)$, then $\gamma(B) =
\gamma(B')$ is dominated by or equal to $\gamma(S_3)$, a contradiction to our
assumption;  hence, either Case
b\oldstylenums{1} or Case b\oldstylenums{2} occurs. 
\medskip

Thus, from the analysis of Step 3, when there exist  $p^{j_1},
p^{j_2} \in P \cap F_3$ with $\#([p^{j_1},p^{j_2}] \cap \Z^2) \geq 2$, then $M(B')$ is a Type 2 triangle satisfying the statement of Case~b.  \medbreak

\noindent  \Step{4}\label{step:Step4} 
Suppose $P \not\subset \Z^2$ and there exist  $p^{j_1},
p^{j_2} \in P \cap F_i$ with $\#([p^{j_1},p^{j_2}] \cap \Z^2) \geq 2$, for $i=1$ or $i=2$.
After a relabeling of the facets of $M(B)$ and the rows of $B$, we can assume $i = 1$. Since $M(B)$ is a Type 2 triangle and $F_1$ is a facet with at most one integer point in its relative interior, we must have $\#(F_1 \cap \Z^2) \leq 2$.  In order for $\#([p^{j_1},p^{j_2}]\cap \Z^2) \geq 2$, it has to equal exactly two, and one of the points, say
$p^{j_1}$, must lie in $F_1 \cap F_3 \cap \Z^2$.  Thus, $p^{j_1}$ corresponds to a corner ray.

If $ P\cap\relint(F_2) \setminus \Z^2 \neq \emptyset$, then again, there must be a corner ray on~$F_2$; otherwise, Lemma \ref{lemma:simple_tilts} shows
that $\gamma(B)$ is a strict convex combination of points in $\Delta$. We can assume that this corner ray points from $f$ to $F_1 \cap F_2$, otherwise we are back to the assumptions in Step 3 and $M(B)$ will satisfy the conditions of Case b. Thus $p^{j_2}$ can be chosen such that $p^{j_2} \in F_1\cap F_2$.

As in Case b, if $P\cup \{f\} \subset M(S_1)$, then $\gamma(B)$ is dominated
by or equal to $\gamma(S_1)$. Hence, we are either in Case c\oldstylenums{1}
or Case c\oldstylenums{2}. \smallskip

\subsubsection*{Proof step 5: Bounding the cardinality of \boldmath $\Xi_2$ by a polynomial.}

Recall that we have a set of $k$ rays $\{r^1, \ldots, r^k\}$ and $P$ is the set of ray intersections. Given this set of rays, we count how many distinct vectors $\gamma(B)$ can arise when $M(B)$ satisfies the conditions in Cases a, b, c and d.  We will apply Lemma \ref{rem:ray_cone} to show that there are only polynomially many possibilities for $\gamma(B)$ in each case. 
\medskip

\noindent\textit{Case a.} $P(B, R) \subset \Z^2$.
Observation~\ref{obs:int_intersections} shows that 
there is a unique $\gamma\in \Xi_2$ corresponding to this case.\old{We need to count the vectors~$\gamma(B)$ with
corresponding $M(B)$ such that $P \subset \Z^2$. Consider the set~$Q = \{ q^1, \dots, q^k\}$ of
$k$ closest integer points that the rays $\{r^1, \ldots, r^k\}$ point to from $f$ such that $q^j = f + \lambda_j r^j$ and $\lambda_j = \min \{ \lambda\geq 0  \st f + \lambda r^j \in \Z^2\}$.  Recall that this definition is similar to the set $P$, that is, $p^j = f + \frac{1}{\psi_B(r^j)} r^j$ where $\mu_j = \frac{1}{\psi_B(r^j)}$ is exactly the scaling factor such that $p^j = f + \mu_j r^j \in \partial M(B)$.
If $\conv(Q)$ is a lattice-free set and it is contained in one or more Type~2
triangles, then we choose any such Type~2 triangle as $M(B)$ and we will have $P =
Q$ because $Q \subset \partial M(B)$. Moreover, all of these triangles yield the same vector $\gamma(B)$. If
$\conv(Q)$ is not lattice-free or is not contained in a Type 2 triangle, then
there does not exist a Type~2 triangle whose set of ray intersections
is~$P$. Therefore there is at most  one possibility for~$\gamma(B)$
which arises from Case~a.}

\begin{figure}
\centering
\ifpdf
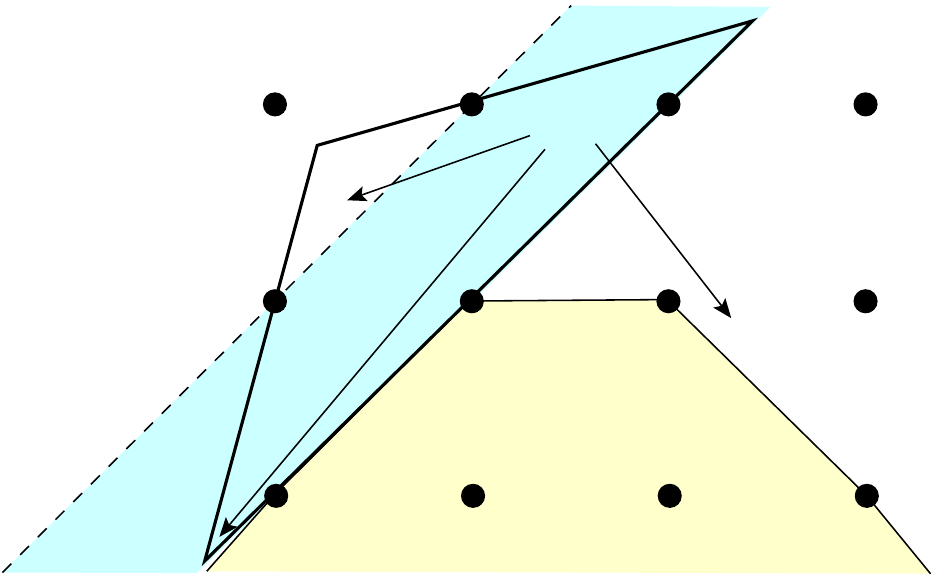
\else
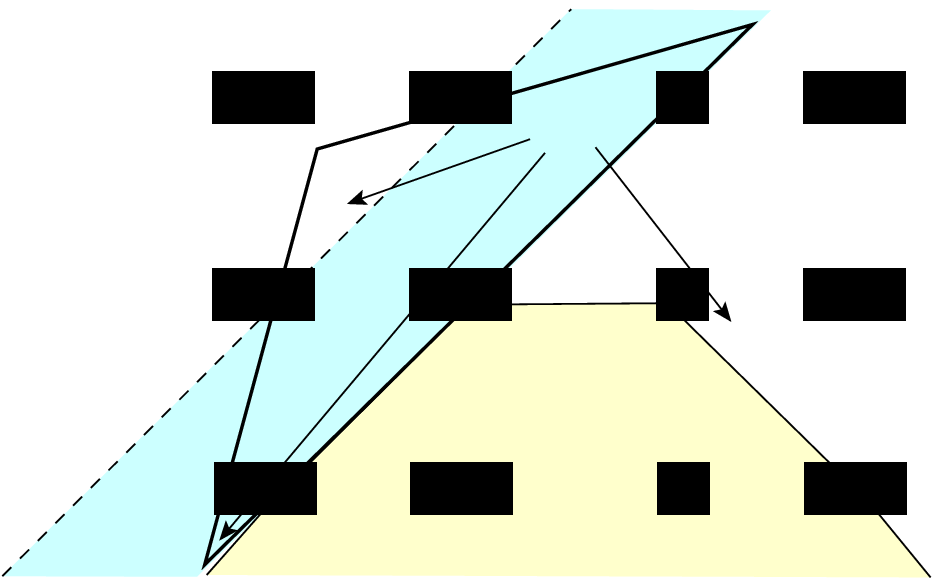
\fi
\caption{Proposition \ref{prop:type2-abstractly}, Step 5: Uniquely determining Type 2 triangles that are in Case~b using a facet of $(C(r^{j_1}, r^{j_2})_\IH)$ and the integer points $y^1$ and $y^2$.}
\label{fig:Type2ChoicesNew}
\end{figure}
\medbreak

\noindent\textit{Case b.} We now count the vectors $\gamma(B)$ such that
$M(B)$ satisfies the conditions of Case b with respect to our set of rays
$\{r^1, \ldots, r^k\}$. Consider any such $M(B)$. From the conditions stated
in Case b, we can assume that $M(B)$ has a corner ray on $F_3$. We label as
$r^{j_1}, r^{j_2}$ the two rays whose corresponding ray intersections are on $F_3$,
so that $r^{j_1}$ points to $F_1\cap F_3$ and the ray intersection of~$r^{j_2}$ is
closest to $F_2 \cap F_3$; and so  $r^{j_1}$ is a corner ray by the statement of
Case~b. There are $2\times {k \choose 2}$ ways to choose $r^{j_1}, r^{j_2}$ from the
set $\{r^1, \ldots, r^k\}$ with one of them as the corner ray.  See
Figure~\ref{fig:Type2ChoicesNew}. By Lemma~\ref{rem:cone_vertex}, $\aff(F_3)$ contains a facet of $(C(r^{j_1},r^{j_2}))_\IH$. By
Lemma~\ref{rem:ray_cone}, we have polynomially many choices for
$\aff(F_3)$. Once we choose $\aff(F_3)$, we consider the possible
choices for $y^1, y^2$, which are the integer points on $F_1, F_2$,
respectively.  

In Case b\oldstylenums{1}, where $f \not\in M(S_3)$, $y^1, y^2$ are given uniquely by where $f$ is. To see this, we observe a few things. Let $y^3$ and $y^4$ be the integer points on $F_3$ that are closest to $F_1\cap F_3$. The split with one side going through $y^1, y^3$ and the other side going through $y^2, y^4$ contains~$f$. Now consider the family of maximal lattice-free splits with one side going through $y^3$ and the other side going through $y^4$. Observe that since $f \not\in M(S_3)$, only one member of this family of splits contains~$f$. This uniquely determines $y^1$ and $y^2$.

In Case b\oldstylenums{2}, $P\not\subset M(S_3)$, which implies that there exists a ray $r^{j_3}$ such that $r^{j_3}$ points between $y^1$~and~$y^2$. Moreover, since $y^1, y^2$ have to lie on the lattice plane adjacent to $F_3$, we have a unique choice for $y^1, y^2$ once we choose $r^{j_3}$ from our set of $k$ rays. Now $r^{j_3}$ can be chosen in $O(k)$ ways and so there are $O(k)$ ways to pick $y^1, y^2$.

We already know there is a corner ray pointing to $F_1\cap F_3$. By the
statement of Case~b, either $P \cap \relint(F_2) \setminus \Z^2 \neq \emptyset$, in which case we have a corner ray of $M(B)$ pointing to
a different vertex, or $P \cap \relint(F_2) \setminus \Z^2 =\emptyset$.
If $M(B)$ has a corner ray pointing to a vertex different from $F_1\cap F_3$,
then we can choose it in $O(k)$ ways, and the triangle is uniquely determined
by these two corner rays, $\aff(F_3)$, $y^1$, and $y^2$. 

On the other hand, if $M(B)$ has corner rays pointing only to $F_1\cap F_3$
(one of which is~$r^{j_1}$), then the facet $F_2$ has no non-integer ray
intersections in its relative interior. Therefore, any possible choice of this
facet such that no ray points to $\relint(F_2)\setminus \Z^2$ will give a
triangle that yields the same vector~$\gamma(B)$.

Hence, there are only polynomially many possibilities for Case b.\smallbreak

\begin{figure}
\centering
\ifpdf
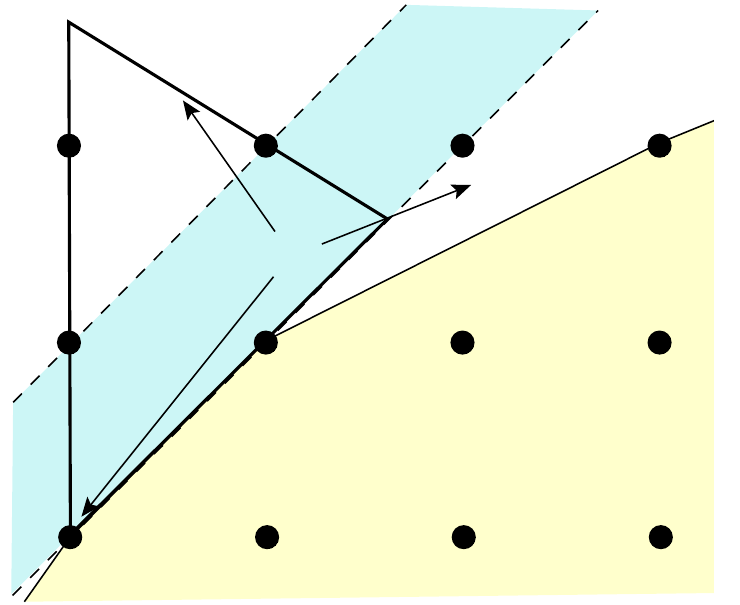
\else
\input{fig10-CasecCounting.pstex_t}
\fi
\caption{Proposition \ref{prop:type2-abstractly}, Step 5: Uniquely determining Type 2 triangles that are in Case~c using a facet of $(C(r^{j_1}, r^{j_2})_\IH)$ and the integer points $y^2$ and $y^4$.}
\label{fig:Type2ChoicesC}
\end{figure}

\noindent \textit{Case c.} We now count the vectors $\gamma(B)$ such that $M(B)$ satisfies the conditions of Case~c with respect to our set of rays $\{r^1, \ldots, r^k\}$. Consider any such $M(B)$. Then there exist $r^{j_1}, r^{j_2}$ such that $p^{j_1}, p^{j_2} \in F_1$ and $\#([p^{j_1},
p^{j_2}]\cap\Z^2) \geq 2$, where $p^{j_1}, p^{j_2}$ are the ray intersections for
$r^{j_1}, r^{j_2}$, respectively. Moreover,  $p^{j_1}$ is an integer point on the facet
$F_3$. There are $2 \times {k \choose 2}$ ways to choose $r^{j_1}, r^{j_2}$ from the
set $\{r^1, \ldots, r^k\}$ with $r^{j_1}$ pointing from $f$ to $F_1 \cap F_3$.  See
Figure~\ref{fig:Type2ChoicesC}.

We next choose  $\aff(F_1)$ as the affine hull of a facet of
$(C(r^{j_1},r^{j_2}))_\IH$, using Lemma~\ref{rem:cone_vertex}. Since $p^{j_1}$ is the integer point that $r^{j_1}$ points to, one of the facets of $(C(r^{j_1},r^{j_2}))_\IH$ that is incident to the vertex~$p^{j_1}$  is unbounded and lies on the infinite ray $f + \R_+r^{j_1}$, while the other facet is bounded.  Since $F_1$ must be bounded, there is a unique choice of $\aff(F_1)$ as the affine hull of the bounded facet of~$(C(r^{j_1},r^{j_2}))_\IH$ that is incident to the vertex~$p^{j_1}$.

%

Now we pick the integer points $y^2, y^4$ where $y^2$ is the integer point on the facet $F_2$ of $M(B)$ and $y^4$ is the integer point in the relative interior of $F_3$ that is closest to $p^{j_1}$. This analysis is the same as with Cases b\oldstylenums{1} and b\oldstylenums{2}. In Case c\oldstylenums{1}, these points are uniquely determined by $f$.  In Case c\oldstylenums{2}, these are uniquely determined by one of the rays pointing between them. There are $O(k)$ ways of choosing this ray.

The statement of Case c implies that either there is also a corner ray
pointing to $F_1\cap F_2$, 
or $P\cap \relint(F_2)\setminus\Z^2=\emptyset$.

If there is a corner ray pointing to $F_1\cap F_2$, then the triangle
is uniquely determined by the two corner rays, $\aff(F_1)$, $y^2$, and $y^4$. 

On the other hand, if $P\cap \relint(F_2)\setminus\Z^2=\emptyset$, then
$F_2$ can be chosen in any possible way such that no
ray points to $\relint(F_2)\setminus\Z^2$. Then the triangle is uniquely
determined by $r^1, \aff(F_1), \aff(F_2), y^2$, and~$y^4$. 

Therefore, there are only polynomially many possibilities for Case c.
\medbreak

\noindent \textit{Case d.} We consider Type 2 triangles with a corner ray $r^{j_1}$ pointing  from $f$ to $F_1\cap F_3$. We label the closest integer point in $\relint(F_3)$ to $F_1 \cap F_3$ as $y^3$, and the next closest integer point in $\relint(F_3)$ as $y^4$. Also, since $P \cap (y^3, y^4) \neq \emptyset$, there exists a ray $r^{j_3}$ that points from $f$  through  $(y^3,y^4)$ (we use the notation $r^{j_3}$ to remind ourselves that it points to $F_3$). Moreover, the condition that no two ray intersections on $F_3$ can contain two (or more) integer points between them implies that the ray intersections on $F_3$ are contained in the segment $[F_1\cap F_3, y^4]$. As before, $y^1$ and $y^2$ will denote the integer points on the facets $F_1$ and $F_2$.
\smallbreak
\indent \textit{Case d\oldstylenums{1} and Case d\oldstylenums{2}.}
For these two cases, there exists a ray $r^{j_2}$ that points from $f$ through
$(y^1, y^2)$ to $F_1$ (for example, in Case d\oldstylenums{1} this will be the
corner ray pointing from $f$ to $F_1\cap F_2$). 
Observe that $[p^{j_1},p^{j_2}] \cap \Z^2 = \{ y^1 \}$ and $[p^{j_1},p^{j_3}]
\cap \Z^2 = \{ y^3 \}$, where $y^1$ and $y^3$ lie in $(p^{j_1},p^{j_2})$ and $(p^{j_1},p^{j_3})$, respectively.
We now count the choices of these triangles. \smallbreak

First pick rays $r^{j_1}, r^{j_2}, r^{j_3}$, for which there are ${k \choose 3}$ ways to do this. 
Pick $y^1$ as a vertex of $(C(r^{j_1}, r^{j_2}))_\IH$ and pick $y^3$ as a vertex of  $(C(r^{j_1}, r^{j_3}))_\IH$, utilizing Lemma~\ref{rem:cone_vertex}~(i).  
By Lemma~\ref{rem:ray_cone}, there are only polynomially many ways to do this.

\begin{claim} The vector $\gamma(B)$ is uniquely determined by the choices of $r^{j_1}, r^{j_2}, r^{j_3}, y^1$, and $y^3$ in Case d\oldstylenums{1} and Case d\oldstylenums{2}.

\end{claim}
\begin{claimproof}
First note that $[y^2, y^4]$ is necessarily parallel to $[y^1, y^3]$.  Therefore, regardless of the choice of $y^2, y^4$, the half-space $H_{2,4}$ is already determined by $[y^1, y^3]$.  Recall that, by assumption, $f\notin H_{2,4}$. Define the family of splits $$
\mathcal{S} = \biggl\{\,  S \in\R^{3\times2} \mathrel{\bigg|}
\begin{array}{@{}l@{}}
  y^1,y^3 \in M(S), 
  \text{ and } M(S) \cap \intr(H_{2,4}) \neq \emptyset,\\
  M(S) \text{ is a maximal lattice-free split}
\end{array}
\, \biggr\}.
$$
For any distinct $M(S_1), M(S_2)$ with $S_1,S_2\in \mathcal{S}$, since they
both contain $[y^1, y^3]$, we  find that $M(S_1) \cap M(S_2) \setminus H_{2,4}
= \emptyset$.  Since $f \notin H_{2,4}$, there exists a unique $M(S)$ with $S
\in \mathcal{S}$ such that $f \in M(S)$.  Therefore, the unique choices for $y^2,
y^4$ are the two points given by $\partial M(S) \cap \partial H_{2,4}$.  
    
Now we show how to choose $M(B)$. The affine hull of facet $F_3$ is determined
by the segment $[y^3, y^4] \subset F_3$, and $\aff(F_1)$ is determined by
$[p^{j_1}, y^1] \subset F_1$, where  $p^{j_1}$ is the corner ray intersection of $r^{j_1}$ on $F_3$.  Lastly, $\aff(F_2)$  must be chosen.  For Case d\oldstylenums{1}, $r^{j_2}$ is chosen as a corner ray pointing from $f$ to $F_1 \cap F_2$, and therefore, $\aff(F_2)$ is determined by the ray intersection $p^{j_2}$ of $r^{j_2}$ on $F_1$, and by $y^2$, i.e., by the segment $[p^{j_2}, y^2]$.  For Case d\oldstylenums{2}, any choice of $\aff(F_2)$ such that there are no rays pointing from $f$ to  $\relint(F_2)\setminus \Z^2$ and such that $f \in M(B)$ will yield the same vector $\gamma(B)$, thus, we only need to consider one such triangle. 
\end{claimproof}
Since the vector $\gamma(B)$ is uniquely determined by these choices, there are only polynomially many possibilities for this case.
\medbreak
\indent \textit{Case d\oldstylenums{3}.}
For this case, there exists a corner ray $r^{j_2}$ that points from $f$ to $F_1 \cap F_2$.  Since $P \not \subset H_{2,4}$, there also exists a ray $r^{j_4}$ such that it points  from $f$ through $(y^2, y^4)$.  Since $r^{j_1}$ is a corner ray pointing  from $f$ to $F_1 \cap F_3$ and the ray intersections are contained in $[F_1\cap F_3, y^4]$, $r^{j_4}$ must be chosen to point from $f$  to  $F_2$.\smallbreak

We now count triangles of this description.  First pick rays $r^{j_1},
r^{j_2}, r^{j_3}, r^{j_4}$ from the set $\{r^1, \ldots, r^k\}$.  There are at
most ${k \choose 4}$ ways to do this. According to Lemma~\ref{rem:cone_vertex}~(i), pick $y^1$ as a vertex of $(C(r^{j_1}, r^{j_2}))_\IH$, $y^2$ as a vertex of $(C(r^{j_2}, r^{j_4}))_\IH$, and $y^3$ as a vertex of  $(C(r^{j_1}, r^{j_3}))_\IH$.  By Lemma \ref{rem:ray_cone}, there are only polynomially many ways to do this. 
Then $y^4$ is uniquely determined since $y^1, y^2, y^3, y^4$ form an area 1
parallelogram. The affine hull of~$F_3$ is uniquely determined, since it runs
along $[y^3, y^4]$. Since $r^{j_1}$ is a corner ray pointing to $F_1\cap F_3$,
the choice of $y^1$ fixes $\aff(F_1)$. Finally, since $r^{j_2}$ is a corner ray
pointing to $F_1 \cap F_2$, the choice of $y^2$ fixes $\aff(F_2)$. Therefore,
there are only polynomially many Type 2 triangles satisfying the conditions of this case.
\medskip

This concludes the proof of the fact that there are only a polynomial (in the binary encoding sizes of $f, r^1, \ldots, r^k$) number of vectors $\gamma(B)$ such that $M(B)$ is a Type 2 triangle satisfying Cases a, b, c and d.
\end{proof}


\subsection{Proof of Proposition~\ref{prop:almost-extreme}}\label{sec:proofs} 
\begin{proof}[Proof of Proposition~\ref{prop:almost-extreme}]
Let $\Xi = \Xi_0\cup\Xi_1\cup\Xi_2\cup\Xi_3$ using the sets $\Xi_i$ from
Propositions~\ref{prop:counting_type3}, \ref{prop:counting_splits}, \ref{prop:counting_type1} and \ref{prop:type2-abstractly}.
We show that for any $\gamma \in \Delta \setminus \Xi$, $\gamma$ is dominated by some $\gamma'\in \Delta$, or $\gamma$ is a strict convex combination of some $\gamma_1, \gamma_2\in \Delta$. If $\gamma \not\in \Pi\cup \Delta_1 \cup \Delta_2 \cup \Delta_3$, then $\gamma$ cannot be realized by a maximal lattice-free split or triangle and so $\gamma = \gamma(B)$ for some $B\in \R^{3\times 2}$ such that $M(B)$ is \emph{not} a maximal lattice-free convex set. This implies that there exists $B'\in \R^{3\times 2}$ such that $M(B')$ is a maximal lattice-free convex set containing $M(B)$ and $\gamma$ is dominated by $\gamma(B')$. 

So we consider $\gamma \in \Pi\cup \Delta_1 \cup \Delta_2 \cup \Delta_3$. Observe that $\Pi\cup \Delta_1 \cup \Delta_2 \cup \Delta_3 = \Pi
  \cup (\Delta_1 \setminus (\Pi\cup \Delta_2)) 
\cup \Delta_3
  \cup (\Delta_2\setminus(\Delta_3\cup \Pi))$ and so $\gamma$ is in one of the sets $\Pi$, $\Delta_1 \setminus (\Pi\cup \Delta_2)$, $\Delta_3$ or $\Delta_2\setminus(\Delta_3\cup \Pi)$. Since $\gamma \not\in \Xi_0\cup\Xi_1\cup\Xi_2\cup\Xi_3$, we have that $\gamma$
is in one of the four sets $(\Pi \setminus \Xi_0)$, $(\Delta_1 \setminus (\Xi_1\cup \Pi\cup \Delta_2))$, $(\Delta_3\setminus\Xi_3)$ or $(\Delta_2\setminus(\Xi_2\cup\Delta_3\cup \Pi))$.
Now it follows from Propositions~\ref{prop:counting_type3}, \ref{prop:counting_splits}, \ref{prop:counting_type1} and
\ref{prop:type2-abstractly} that $\gamma$ is dominated by some $\gamma' \in\Delta$, or $\gamma$ is a
strict convex combination of some $\gamma^1, \gamma^2\in \Delta$. Furthermore,
the cardinality $\# \Xi{}$, being bounded above by the sum of the
cardinalities of~$\Xi_i$, $i=0,\dots,3$, is polynomial in the binary encoding sizes of $f, r^1, \dots, r^k$. 
\end{proof}


\section{Proof of Theorem \ref{thm:main}}
\label{final}

In this section, we will complete the proof of Theorem~\ref{thm:main}. As stated in the introduction, the only result we will need from Section~\ref{sec:finiteness} is Theorem~\ref{thm:polynomial}. Apart from this, we will utilize the results proved in Section~\ref{preliminaries}.  We first state the following proposition.

\begin{prop}\label{prop:polarity}
  Let $\|\cdot\|$ be a norm on the space of matrices~$\R^{3\times 2}$. 
  Let $\mathcal{B}$ be a family of matrices in $\R^{3\times 2}$. If there exists $\epsilon > 0$ such that $\B(f, \epsilon) \subseteq M(B)$ for all $B \in \mathcal{B}$, then there exists a real number $M$ depending only on $\epsilon$ such that $\lVert B \rVert < M$ for all $B \in \mathcal{B}$.
\end{prop}

\begin{proof}
Since $\B(f, \epsilon) \subseteq M(B)$, the point $f + \epsilon b^i \in M(B)$, where $b^i$ is the $i$-th row of $B$. Therefore, $b^i\cdot(f + \epsilon b^i - f) \leq 1$. Therefore, $\lVert b^i \rVert_2 \leq \frac{1}{\sqrt\epsilon}$. Since this holds for every row $b^i$, there exists $M$ depending only on $\epsilon$ such that $\lVert B \rVert < M$.
\end{proof}
We will use the following set to derive a bound on a sequence of matrices to show there exists a convergent subsequence. For any vector $\gamma\in\R^k_+$, define

\[
\M_\gamma = \conv(\{f\}\cup \{\, f + \tfrac{1}{\gamma_j}r^j \st \gamma_j \neq 0\,\}) + \cone(\{\,r^j \st\gamma_j = 0\,\}).
\]

\begin{obs}\label{obs:inclusion}
For all $B\in \R^{3\times 2}$ we have the inclusion $\M_{\gamma(B)} \subseteq M(B)$. 
\end{obs}

\begin{proof}
Clearly $f \in M(B)$. Next observe that $f +  \frac{1}{\psi_B(r^j)}r^j \in M(B)$ if $\psi_B(r^j) > 0$. Finally, $\psi_B(r^j) = 0$ implies that $r^j$ is in the recession cone of $M(B)$. The claim follows. 
\end{proof}\begin{theorem}\label{thm:finite_extreme}
Assume that $f \in \Q^2$ and $r^j \in \Q^2$ for all $j \in \{1, \ldots, k\}$. If $\cone(\{r^1, \ldots, r^k\})  =\R^2$, then $\Delta'$ has a polynomial (in the binary encoding sizes of $f, r^1, \ldots, r^k$) number of extreme points.
\end{theorem}

\begin{proof}
Consider any extreme point $x$ of $\Delta'$. By Observation~\ref{obs:minimal}, $x \in \cl(\conv(\Delta))$. By Lemma~\ref{lem:seq_extreme}, there exists a sequence $(a^n)$ of points from $\Delta$ such that $(a^n)$ converges to $x$. 

\begin{claim} There exists a bounded sequence of matrices $(B_n) \in \R^{3 \times 2}$ such that $\gamma(B_n) = a^n$ and $M(B_n)$ is lattice-free.
\end{claim} 
\begin{claimproof}
Since $(a^n)$ converges to $x$, there exists $N \in \mathbb{N}$ such that $a^n_i \leq x_i + 1$ for every $n \geq N$ and $i \in \{1, \ldots, k\}$.
Since $(a^n)$ is a sequence in $\Delta$, there exists a sequence of matrices $(B_n)$ such that $(a^n)= (\gamma(B_n))$ and $M(B_n)$ is lattice-free for all $n \in \mathbb N$. Consider the sequence of polyhedra $\M_{\gamma(B_n)}$.  Let $\epsilon = \frac{1}{1 + \max_{i}{x_i}}$.

By the definition of $N$, for every $n \geq N$, we have that $\frac{1}{a^n_i} \geq \epsilon$. Since the conical hull of the rays $r^1, \ldots, r^k$ is $\R^2$, this implies that there exists $\bar \epsilon$ such that $\B(f, \bar \epsilon) \subseteq \M_{\gamma(B_n)}$ for all $n \geq N$. By Observation~\ref{obs:inclusion}, $\M_{\gamma(B_n)} \subseteq M(B_n)$. Therefore, for every $n \geq N$, $\B(f, \bar \epsilon) \subseteq M(B_n)$. Proposition~\ref{prop:polarity} implies that there exists a real number $M$ depending only on $\bar \epsilon$ such that $\lVert B_n \rVert \leq M$ for all $n \geq N$. This implies that $(B_n)$ is a bounded sequence.
\end{claimproof}

By the Bolzano--Weierstrass theorem, we can extract a convergent subsequence
$(\bar B_n)$ converging to a point~$\bar B$.  The map
$B\mapsto \gamma(B)$ is continuous because $\psi_B(r)$ is continuous in~$B$
for every fixed~$r$. Therefore, the sequence $(\gamma(\bar B_n))$ converges to $\gamma(\bar
B)$.  By assumption $(a^n) = (\gamma(B_n))$ converges to~$x$ and therefore
$\gamma(\bar B) = x$. Moreover, since $M(\bar B_n)$ is lattice-free for all
$n\in \N$, $M(\bar B)$ is also lattice-free and hence it is a lattice-free
triangle or a lattice-free split. Thus, $x = \gamma(\bar B) \in \Delta$ since $M(\bar B)$ is a lattice-free triangle or lattice-free split.  Therefore, we have shown that for every extreme point~$x$
of~$\Delta'$, we have that $x \in \Delta$. 

Let $\Xi{}$ be the set from Theorem~\ref{thm:polynomial}. Theorem~\ref{thm:polynomial} implies that the extreme point $x \not\in \Delta \setminus \Xi{}$. Since we show that $x \in \Delta$, this implies $x \in\Xi{}$.  Since $\#\Xi{}$ is polynomial in the encoding sizes of $f, r^1, \dots, r^k$, we have shown this property for the number of extreme points $\Delta'$ as well.
\end{proof}
   
\begin{corollary}\label{cor:polyhedron}
Assume that $f \in \Q^2$ and $r^j \in \Q^2$ for all $j \in \{1, \ldots, k\}$. If $\cone(\{r^1, \ldots, r^k\}) = \R^2$, then the triangle closure $T$ is a polyhedron with a polynomial (in the binary encoding sizes of $f, r^1, \ldots, r^k$) number of facets.
\end{corollary}

\begin{proof}
Lemma~\ref{lem:extreme_rep} and Theorem~\ref{thm:finite_extreme} together imply the corollary.
\end{proof}

We can now finally prove Theorem~\ref{thm:main}.

\begin{proof}[Proof of Theorem~\ref{thm:main}]
If $\cone(\{r^1, \ldots, r^k\}) = \R^2$, then Corollary~\ref{cor:polyhedron}
gives the result. Otherwise we add ``ghost'' rays $r^{k+1}, \ldots, r^{k'}$
such that $\cone(\{r^1, \ldots, r^k, r^{k+1}, \ldots, r^{k'}\}) = \R^2$. Now
consider the system~\eqref{M_fk} with the rays $r^1, \ldots, r^{k'}$. We can
similarly define the triangle closure $T'$ for this extended system. Given a matrix $B \in \R^{3\times 2}$, let $\alpha(B) =  (\psi_B(r^i))_{i = 1}^{k'} \in \R^{k'}$ (we will continue to use $\gamma(B) = (\psi_B(r^i))_{i = 1}^{k}$). So $T'$ is defined as
\[
T' = \bigl\{\,s \in \R^{k'}_+ \bigst \alpha(B)\cdot s \geq 1 \textrm{ for all } B\textrm{ such that }M(B)\textrm{ is a lattice-free triangle}\,\bigr\}.
\]

\begin{claim}
$T' \cap \{s_{k+1} = 0, \ldots, s_{k'}=0\} = T \times \{0^{k' - k}\}$. 
\end{claim}

\begin{claimproof}
Consider any point $s \in T' \cap \{s_{k+1} = 0, \ldots, s_{k'} = 0\}$ and let $s^k = (s_1, \ldots, s_k)$ be the truncation of $s$ to the first $k$ coordinates. Consider any $a \in \R^k$ such that $a = \gamma(B)$ for some matrix $B$ where $M(B)$ is a lattice-free triangle. Consider $a' = \alpha(B)$. Clearly, $a'_i = a_i$ for $i \in \{1, \ldots, k\}$. Since $a'\cdot s \geq 1$ and $a'\cdot s = a\cdot s^k$, we have that $a\cdot s^k \geq 1$. So, $s \in T \times \{0^{k' - k}\}$.

For the reverse inclusion, consider a point $s \in T \times \{0^{k' - k}\}$ and let $s^k = (s_1, \ldots, s_k)$ be the truncation of $s$ to the first $k$ coordinates. Consider any $a' \in \R^{k'}$ such that $a' = \alpha(B)$ for some matrix $B$ where $M(B)$ is a lattice-free triangle. Let $a = \gamma(B)$. As before, $a'_i = a_i$ for $i \in \{1, \ldots, k\}$. Since $a\cdot s^k \geq 1$ and $a'\cdot s = a\cdot s^k$, we have that $a'\cdot s \geq 1$. So, $s \in T' \cap \{s_{k+1} = 0, \ldots, s_{k'} = 0\}$.
\end{claimproof}

Since  $\cone(\{r^1, \ldots, r^k, r^{k+1}, \ldots, r^{k'}\}) = \R^2$, Corollary~\ref{cor:polyhedron} says that $T'$ is a polyhedron with a polynomial (in the binary encoding sizes of $f, r^1, \ldots, r^k$) number of facets. Since $T' \cap \{s_{k+1} = 0, \ldots, s_{k'}=0\} = T \times \{0^{k' - k}\}$, this shows that $T$ is a polyhedron with a polynomial number of facets.
\end{proof}

This concludes the part of the paper which deals with the result that the triangle closure is a polyhedron.

\section{Proof of Theorems~\ref{THM:POLYFACETS} and~\ref{thm:enumerate}}\label{sec:polynomiality}

We now complete our second result showing that the mixed integer hull $\conv(R_f)$ has only polynomially many facets.
We first make a counting argument for quadrilaterals that is similar to the counting arguments in Section~\ref{sec:counting}. For quadrilaterals, Cornu\'ejols and Margot~\cite{cm} defined the \emph{ratio condition}
as a necessary and sufficient condition to yield an extreme inequality when all corner rays
are present.  Suppose $p^{j_1}, p^{j_2}, p^{j_3}, p^{j_4}$ are the corner ray intersections
assigned in a counter-clockwise orientation, and $y^i$ is the integer point
contained in $[p^{j_i}, p^{j_{i+1}}]$, where we set $j_5 = j_1$.  The ratio condition holds if there does not
exist a scalar $t > 0$ such that 
\begin{equation}\label{eq:ratio-cond}
\frac{\Vert y^i - p^{j_i}\Vert_2}{\Vert y^i - p^{j_{i+1}}\Vert_2} = \begin{cases}
t & \text{for} \ i = 1,3\\
\frac{1}{t} & \text{for} \ i = 2,4.
\end{cases}
\end{equation}
This is illustrated in Figure~\ref{fig:ratio-cond}.

\begin{figure}
\centering
\ifpdf
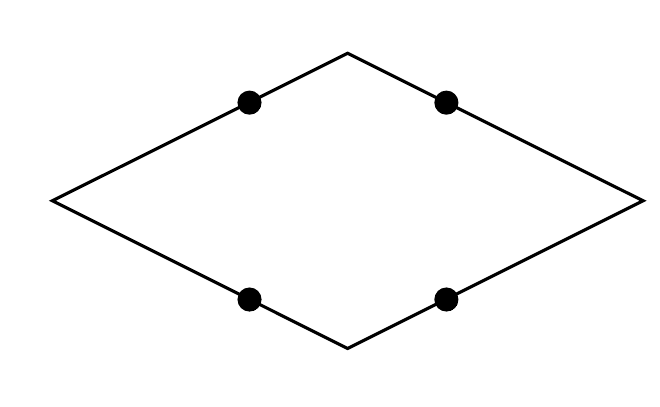
\else
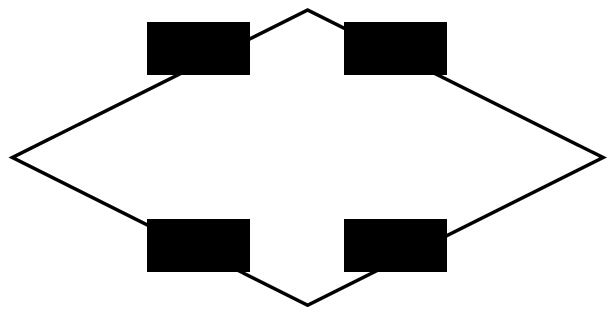
\fi
\caption{Example of a quadrilateral for which the ratio condition does \emph{not}
  hold, i.e., there exists $t>0$ satisfying \eqref{eq:ratio-cond}. Here
  $\dim \mathcal N(\Y) \neq 0$.}
\label{fig:ratio-cond}
\end{figure}

\begin{proposition}[Counting Quadrilaterals]\label{prop:counting_quads} There exists a finite set $\Xi_4 \subseteq \R^k_+$ such that if $\gamma(B)\not\in \Xi_4$ for some $B \in \R^{4 \times 2}$ where $M(B)$ is a maximal lattice-free quadrilateral, then $\gamma(B)$ is not extreme. Moreover, the cardinality of $\Xi_4$ is bounded by a polynomial in the binary encoding size of $f, r^1, \ldots, r^k$.
Specifically,  $\Xi_4$ can be chosen
as the set of all $\gamma(B)$ such that $M(B)$ is a maximal lattice-free quadrilateral
satisfying one of the following:\\ 
\textbf{Case a.} $P \subset \Z^2$. \\
\textbf{Case b.} $M(B)$ has four distinct corner rays and the ratio condition holds.
\end{proposition}


\begin{proof}
  \Step{1}\
Suppose that $\gamma(B)$ is extreme and that we are not in Case a.
Lemma~\ref{lemma:corner_rays_almost} shows that all four corner rays must
exist.   Suppose, for the sake of contradiction, that the ratio condition does
not hold.  Lemma~\ref{lemma:ratio_condition} in the Appendix then shows that
$\dim\Ny \geq 1$.   Let $\bar A = (\bar a^1; \bar a^2; \bar a^3; \bar a^4) \in
\mathcal \Ny\setminus \{0\}$. Since $\bar A\neq 0$, we have $\bar a^i \neq 0$ for some $i = 1, \dots, 4$.  
Since there are 4 corner rays, $p^j\in F_i$ for some $j = 1, \dots, k$.  Lastly, since $\{y^i\} = Y_i$, Lemma~\ref{obs:dimension} shows that for some $\epsilon >0$, $\gamma(B)$ is a strict convex combination of $\gamma(B + \epsilon \bar A)$ and $\gamma(B - \epsilon \bar A)$ and $M(B + \epsilon \bar A)$, $M(B - \epsilon \bar A)$ are lattice-free quadrilaterals.  Therefore, $\gamma(B)$ is not extreme, which is a contradiction.  Hence, the ratio condition must hold.  \smallbreak

\noindent\Step{2}
We now bound the number of possible vectors in $\Xi_4$.\\
\indent \textit{Case a.} Observation~\ref{obs:int_intersections} shows that
there is a unique $\gamma\in\Xi_3$ corresponding to this case.\smallskip

\indent \textit{Case b.} $M(B)$ has all four corner rays and the ratio
condition holds. By Lemma~\ref{lemma:ratio_condition}, the quadrilateral we construct must be uniquely determined by the choice of four corner rays and choice of four integer points that lie on the facets of the quadrilateral.  
There are $O(k^4)$ ways to pick four rays $r^{j_1}, r^{j_2}, r^{j_3}, r^{j_4}$ to be corner rays.   By Lemma \ref{rem:cone_vertex}~(i),  the four integer points $y^1, y^2, y^3, y^4$ are such that $y^i$ is a vertex of $(C(r^{j_i}, r^{j_{i+1}}))_\IH$, with $i=1,2,3$ and $y^4$ a vertex of $(C(r^{j_4}, r^{j_1}))_\IH$.  Lemma~\ref{rem:ray_cone} shows that there are polynomially many possibilities for $y^1, \ldots, y^4$.
\end{proof}

\begin{proof}[Proof of Theorem~\ref{THM:POLYFACETS}]
We now introduce the set~$\Gamma$ of all vectors~$\gamma(B)$ that come from arbitrary
(not necessarily maximal) lattice-free polyhedra in~$\R^2$,
$$
\Gamma = \bigcup_{n \in \N}\{\,\gamma(B) \st B\in \R^{n\times 2}\textrm{ such
  that }M(B)\textrm{ is a lattice-free convex set} \,\}.
$$ 
Since we consider $B \in \R^{n\times 2}$ for all $n \in \N$, this includes all
$\gamma(B)$ such that $M(B)$ is a lattice-free split, triangle, or
quadrilateral and all other polyhedra that are lattice-free in $\R^2$. It is
easy to verify (see Lemma 1.6 and its proof in~\cite{cm}) that $\conv(R_f)$ is
a polyhedron of {\em blocking type} (see Section 9.2 in~\cite{sch} for a
discussion of blocking polyhedra). In fact, because of the
correspondence~\eqref{eq:mixed-int-hull-by-lattice-free} 
between valid inequalities for the mixed integer
hull and lattice-free sets, one can show that $\Gamma$ is actually the {\em blocking polyhedron} of $\conv(R_f)$, i.e., $\conv(R_f) = \{s \in \R^k_+ \st \gamma\cdot s \geq 1 \textrm{ for all } \gamma \in \Gamma\}$ and $\Gamma = \{\gamma\in \R^k_+ \st \gamma\cdot s \geq 1 \textrm{ for all } s \in \conv(R_f)\}$. Hence, by Theorem 9.2\,(iv) in~\cite{sch}, the facets of $\conv(R_f)$ are given by $\gamma\cdot s \geq 1$ where $\gamma$ is an extreme point of $\Gamma$. So we need to enumerate the extreme points of $\Gamma$. Moreover, if $\gamma$ is an extreme point of $\Gamma$, then there does not exist $\gamma' \in \Gamma$ such that $\gamma$ is dominated by $\gamma'$. By Propositions~\ref{prop:counting_type3}, \ref{prop:counting_splits}, \ref{prop:counting_type1}, \ref{prop:type2-abstractly} and \ref{prop:counting_quads}, the extreme points of $\Gamma$ can only be in $\Xi_0\cup\Xi_1\cup\Xi_2\cup\Xi_3\cup\Xi_4$, whose cardinalities are bounded above by a polynomial in the binary encoding sizes of $f, r^1, \ldots, r^k$.\end{proof}

\begin{proof}[Proof of Theorem~\ref{thm:enumerate}]
As established in the proof of Theorem~\ref{THM:POLYFACETS}, we only need to search in the set $\Xi_0\cup\Xi_1\cup\Xi_2\cup\Xi_3\cup\Xi_4$ to find the facet-defining inequalities. The conditions defining these five sets and the counting arguments presented in Section~\ref{sec:finiteness} and Proposition~\ref{prop:counting_quads} can be converted into an algorithm for enumerating all the points in $\Xi_0\cup\Xi_1\cup\Xi_2\cup\Xi_3\cup\Xi_4$. 
This relies on the algorithm by Hartmann~\cite{hartmann-1989-thesis} for
computing the facets and vertices of integer hulls, which runs in polynomial
time in fixed dimension (cf.~Lemma \ref{rem:ray_cone}).
Thus we  generate a set of valid inequalities (of polynomial cardinality)
that is a superset of all the facets. We can then use standard LP techniques
to select the facet-defining ones from these. 
\end{proof}

\section{Conclusion}
We conclude this paper with a discussion of some interesting open problems.

\paragraph{Generalized triangle closures.}
A drawback of the triangle closure result presented in this paper is that it only applies to a system with two integer variables. Here is one way to generalize to general mixed integer linear programs. Consider the polyhedron $C = \{(x,y) \in \R^p \times\R^q : Px + Qy \leq d\}$ for some matrices $P \in \R^{m \times p}, Q\in\R^{m \times q}$. We are interested in the mixed integer hull $C_\IH = \conv(C \cap (\Z^p \times \R^q))$. We define the generalized triangle closure in the following way. Consider any two-dimensional lattice subspace of $\Z^p$ and consider a lattice-free triangle in this subspace. Let $\mathcal{T}$ be the family of all such triangles from all possible two-dimensional lattice subspaces. For any $T\in \mathcal{T}$, let $b^1, \ldots, b^p$ be a lattice basis for $\Z^p$ such that $b^1$ and $b^2$ are a basis for the lattice subspace containing $T$. Let $L(T)$ be the linear subspace spanned by $b^3, \ldots, b^p$. Then $R(T) = (T \oplus L(T)) \times \R^q$ is a polyhedron which contains no point from $\Z^p \times \R^q$ in its interior. We define the {\em generalized triangle closure} as $$\bigcap_{T \in \mathcal{T}} \conv(C \setminus \intr(R(T))).$$ We would like to show that the generalized triangle closure is also a polyhedron.


\paragraph{Quadrilateral closures.}
Even for the case $m=2$, one can ask about the quadrilateral closure. If
  one considers this to be the intersection of inequalities derived from all
  possible lattice-free quadrilaterals (and not just the maximal ones), then
  this convex set can be seen to be the same as the mixed integer hull
  $\conv(R_f)$, since every maximal lattice-free triangle and split can be
  arbitrarily well approximated by a quadrilateral (but not necessarily a
  maximal one). Of course, $\conv(R_f)$ is known to be a polyhedron, and so
  the question for the quadrilateral closure becomes interesting only if we
  restrict ourselves to the maximal quadrilaterals. We conjecture that this
  is also a polyhedron, but it does not seem to be an immediate corollary of
  the results of this paper. 

On a related note, we mention that in this paper, we included non-maximal
triangles to define the triangle closure. However, one can see that if we
restrict ourselves to only maximal triangles, we would obtain the same convex
set. This is because an inequality derived from a non-maximal triangle 
is equal to or dominated by one derived from a maximal triangle
or a split, and a split can be obtained as the limit of maximal triangles. 
In this respect, the quadrilateral closure differs from the triangle
closure: an inequality derived from a non-maximal quadrilateral
may not be equal to or dominated by an inequality from a maximal
quadrilateral or a limit of such inequalities.

\paragraph{Higher dimensions.}
Many of the tools described in Subsection 3.1 readily extend to $m\geq3$.
This can be used to study the extremality of inequalities arising from maximal
lattice-free convex sets in higher dimension.  Unfortunately, it is unclear
what results can be obtained in higher dimensions due to the difficult task of
first classifying all maximal lattice-free convex sets in higher dimensions.
Such a classification for $m=3$ is only partially known~\cite{averkov2011} while even less is known
for $m>3$.  That said, the tools given in this paper may be found useful for
studying specific classses of lattice-free convex sets in higher dimensions
such as simplices or cross-polytopes.  Studying such classes of inequalities
may produce stronger valid inequalities simply because more rows of the
simplex tableau are utilized.

\appendix
{\small
\section{Appendix: Uniqueness of a triangle defined by 3 corner rays and a point on the
  relative interior of each facet}

\begin{proposition}\label{prop:unique-triangle}
Any triangle defined by 3 corner rays and 3 points (one on the relative interior of each facet) is uniquely defined. 
\end{proposition}
\begin{proof}
The space of  triangles with these three corner rays and 3 points is exactly the tilting space
of any such triangle satisfying this.  For convenience we define $\y^i:= y^i - f$
and $\shiftedp^i := p^i - f$, where $p^i
$ are the ray intersections.  Then $\shiftedp^i = 
\frac{1}{\psi_B(r^i)} r^i$. 

We want to show that the solution to the following systems of equations is unique.
\begin{equation*}
	\begin{array}{ccc}
		\begin{array}{rl}
		 	a^1\cdot \y^1  &= 1            \\			
			 a^1\cdot \shiftedp^2  &=a^2\cdot  \shiftedp^2   \\
			a^2\cdot  \y^2  &= 1            \\
			a^2\cdot  \shiftedp^3  &= a^3\cdot  \shiftedp^3  \\
			a^3\cdot  \y^3  &= 1            \\
			a^3\cdot  \shiftedp^1  &=a^1\cdot  \shiftedp^1 
		\end{array}
	& \Rightarrow &
	\begin{bmatrix}
		\y^1                         \\
		 \shiftedp^2 & -\shiftedp^2              \\
			 & \y^2               \\
			 & \shiftedp^3 & - \shiftedp^3    \\
			 &       & \y^3      \\
		-\shiftedp^1 &       & \shiftedp^1
	\end{bmatrix}
	\begin{bmatrix}
		 a^1 \\ a^2 \\ a^3
	\end{bmatrix}
	= 
	\begin{bmatrix}
		1\\0\\1\\0\\1\\0
	\end{bmatrix}
\end{array}
\end{equation*}
We then write this down as a matrix equation where every vector in the matrix is a row vector of size 2, therefore we have a $6 \times 6$ matrix.  We will analyze the determinant of the matrix.

Since the points $\y^1, \y^2, \y^3$ are on the interior of each facet, they can be written as convex combinations of $\shiftedp^1, \shiftedp^2, \shiftedp^3$.

\begin{equation*}
\begin{array}{ccc}
\y^1 = \frac{1}{\alpha'} \shiftedp^1 + \frac{\alpha}{\alpha'} \shiftedp^2 &  & \shiftedp^1 = \alpha' \y^1 - \alpha \shiftedp^2\\
\y^2  =  \frac{1}{\beta'} \shiftedp^2 + \frac{\beta}{\beta'}\shiftedp^3&  \Rightarrow & \shiftedp^2 = \beta' \y^2 - \beta \shiftedp^3\\
\y^3 =  \frac{1}{\gamma'} \shiftedp^3 +\frac{\gamma}{\gamma'} \shiftedp^1&  & \shiftedp^3 = \gamma' \y^3 - \gamma \shiftedp^1
\end{array}
\end{equation*}

Therefore, we can perform row reduction on the last row.  Just tracking the last row, we have
$$
\begin{bmatrix}
-\shiftedp^1 & 0 & \shiftedp^1
\end{bmatrix}\rightarrow
\begin{bmatrix}
0 & \alpha \shiftedp^2 & \shiftedp^1
\end{bmatrix}
\rightarrow
\begin{bmatrix}
0& 0& \shiftedp^1 - \alpha \beta \shiftedp^3
\end{bmatrix}.
$$

This matrix now has an upper block triangular form, and the determinant is easily computed as 
$$
\det (\y^1; \shiftedp^2) \det(\y^2; \shiftedp^3) \det(\y^3; \shiftedp^1 - \alpha \beta \shiftedp^3).
$$
The first two determinants are non-zero because those vectors are linearly independent.  The last determinant requires some work:
$$
\begin{bmatrix}
\y^3\\
 \shiftedp^1 - \alpha \beta \shiftedp^3
\end{bmatrix} =
 \begin{bmatrix}
\frac{1}{\gamma'} \shiftedp^3 +\frac{\gamma}{\gamma'} \shiftedp^1\\
 \shiftedp^1 - \alpha \beta \shiftedp^3
\end{bmatrix}
= 
\begin{bmatrix}
\frac{\gamma}{\gamma'} & \frac{1}{\gamma'} \\
1 & -\alpha \beta
\end{bmatrix}
\begin{bmatrix}
\shiftedp^1\\
\shiftedp^3
\end{bmatrix}.
$$
Since all the coefficients are positive, the determinant of the first matrix is strictly negative, and since $\shiftedp^1, \shiftedp^3$ are linearly independent, the determinant of the second matrix is non-zero.\smallbreak

Hence, the determinant of the original matrix is non-zero, and therefore the system of equations has a unique solution.
\end{proof}

\section{Appendix: Ratio condition for quadrilaterals}\label{sec:ratio_cond}

\begin{lemma}
\label{lemma:ratio_condition}
Suppose $\M(B)$ is a quadrilateral with four corner
rays~$r^{j_1},r^{j_2},r^{j_3},r^{j_4}$.  
The following are equivalent:
\begin{enumerate}[\rm(i)]
\item The ratio condition holds.
\item $\dim \mathcal N(\Y) \neq 0$.
\item The quadrilateral is uniquely defined by these corner rays and the integer points lying on the
  boundary.
\end{enumerate}

\end{lemma}
\begin{proof} 
We will first analyze the tilting space equations with four corner rays, and then apply the assumption that the ratio condition does not hold.  Label the integer points on the facets of $M(B)$ such that $y^i \in [p^{j_i}, p^{j_{i+1}}]$ for $i=1,2,3,4$ and $j_5 = j_1$, as in Figure~\ref{fig:ratio-cond}.
For convenience we define $\y^i:= y^i - f$ and $\shiftedp^{i} := p^{j_i} - f$. Then $\shiftedp^i = 
\frac{1}{\psi_B(r^{j_i})} r^{j_i}$. 

We want to determine when there is not a unique solution to the following system of equations that come from the tilting space:
\begin{equation*}
\begin{array}{ccc}
\begin{array}{r@{\;}l}
a^1 \cdot \y^1&= 1\\
a^1 \cdot \shiftedp^2&= a^2\cdot  \shiftedp^2 \\
a^2\cdot \y^2  &= 1\\
a^2\cdot \shiftedp^3  &= a^3\cdot \shiftedp^3 \\
a^3\cdot \y^3  &= 1\\
a^3\cdot \shiftedp^4  &=a^4\cdot  \shiftedp^4 \\
a^4 \cdot \y^4 &= 1\\
a^4\cdot \shiftedp^1  &=a^1\cdot  \shiftedp^1 
\end{array}
&
\quad\text{or}\quad
&
\begin{bmatrix}
\y^1 \\
\shiftedp^2 & -\shiftedp^2 \\
& \y^2\\
& \shiftedp^3 & - \shiftedp^3 \\
& & \y^3\\
 & & \shiftedp^4 & -\shiftedp^4\\
 & & & \y^4\\
 - \shiftedp^1 & & & \shiftedp^1
\end{bmatrix}
\begin{bmatrix}
a^1\\ a^2 \\ a^3\\a^4
\end{bmatrix}
= 
\begin{bmatrix}
1\\0\\1\\0\\1\\0\\1\\0
\end{bmatrix}
\end{array}
\end{equation*}
as an $8 \times 8$  matrix equation where every vector shown in the matrix is a row vector of
size 2.  We will analyze the determinant of the matrix.

Since the points $\y^1, \y^2, \y^3, \y^4$ are on the interior of each facet,
they can be written as certain convex combinations of $\shiftedp^1, \shiftedp^2, \shiftedp^3, \shiftedp^4$.  We write this in a complicated form at first to simplify resulting calculations.  Here, $\alpha' = 1 + \alpha$, and $\alpha >0$, and similarly for $\beta, \gamma$, and~$\delta$.
$$
\begin{array}{r@{\;}lcr@{\;}l}
\y^1 &= \frac{1}{\alpha'} \shiftedp^1 + \frac{\alpha}{\alpha'} \shiftedp^2 & &\shiftedp^1 &= \alpha' \y^1 - \alpha \shiftedp^2\\
\y^2 &= \frac{1}{\beta'} \shiftedp^2 + \frac{\beta}{\beta'} \shiftedp^3 &\Leftrightarrow &\shiftedp^2 &= \beta' \y^2 - \beta \shiftedp^3\\
\y^3 &= \frac{1}{\gamma'} \shiftedp^3 + \frac{\gamma}{\gamma'} \shiftedp^4 & &\shiftedp^3 &= \gamma' \y^3 - \gamma \shiftedp^4\\
\y^4 &=\frac{1}{\delta'} \shiftedp^4 + \frac{\delta}{\delta'} \shiftedp^1 & &\shiftedp^4 &= \delta' \y^4 - \delta \shiftedp^1
\end{array}
$$
Now just changing the last row using the above columns
$$
\begin{bmatrix}
 - \shiftedp^1 & 0& 0& \shiftedp^1
\end{bmatrix} \rightarrow
\begin{bmatrix}
  0&\alpha \shiftedp^2 &0 & \shiftedp^1
\end{bmatrix} \rightarrow
\begin{bmatrix}
 0 &0 &-\alpha \beta \shiftedp^3 & \shiftedp^1
\end{bmatrix} \rightarrow
\begin{bmatrix}
0  &0 &0 & \alpha \beta \gamma \shiftedp^4 + \shiftedp^1 
\end{bmatrix}
$$
The resulting matrix, after adding this last row and substituting in $\y^4$, is

$$
\left[
\begin{array}{cccc}
\cellcolor{lgray} \y^1  & & &\\
\cellcolor{lgray}\shiftedp^2 & -\shiftedp^2 & &\\
& \cellcolor{lgray}\y^2 & & \\
& \cellcolor{lgray}\shiftedp^3 & - \shiftedp^3 & \\
& & \cellcolor{lgray}\y^3 &\\
 & & \cellcolor{lgray}\shiftedp^4 & -\shiftedp^4\\
 & & & \cellcolor{lgray}\frac{1}{\delta'} \shiftedp^4 + \frac{\delta}{\delta'} \shiftedp^1\\
  & & & \cellcolor{lgray} \alpha \beta \gamma \shiftedp^4 + \shiftedp^1 
\end{array}
\right]
$$

This is now an upper block triangular matrix, which is emphasized by shading.  The first three blocks are all non-singular, and the last block is non-singular if and only if there does not exist a $t$ such that 
$$
  \frac{1}{\delta'} \shiftedp^4 + \frac{\delta}{\delta'} \shiftedp^1 =  t(\alpha \beta \gamma \shiftedp^4 + \shiftedp^1 )\ \Rightarrow \ \Big(\frac{\delta}{\delta'} - t\Big) \shiftedp^1 + \Big(\frac{1}{\delta'} - t \alpha \beta \gamma\Big)\shiftedp^4 = 0.
$$
If such a $t$ exists, then $t = \frac{\delta}{\delta'}$ since $\shiftedp^1$ and $\shiftedp^4$ are linearly independent.    It follows that $\alpha \beta \gamma \delta = 1$ if and only if $\dim \mathcal N(\Y) \neq 0$. 

It is easy to see that the ratio condition does not hold if and only if $\alpha= \frac{1}{\beta} = \gamma  = \frac{1}{\delta}$.
Therefore, it remains to show that $\alpha= \frac{1}{\beta} = \gamma  = \frac{1}{\delta}$ if and only if $\alpha \beta \gamma \delta = 1$.  The forward direction is obvious.  We will show that if it is not true that $
\alpha = \frac{1}{\beta} = \gamma = \frac{1}{\delta}
$, then $\alpha \beta \gamma \delta \neq 1$.
To do so, we use ideas from the proof by Cornu\'ejols and Margot of Theorem 3.10 in \cite{cm}.

Let $Y, X \in \R^{2 \times 4}$ be the matrices  with columns $\bar y^1, \bar
y^2, \bar y^3, \bar y^4$ and $\bar p^1, \bar p^2, \bar p^3, \bar p^4$,
respectively.  Next let $\bar Y, \bar X \in \R^{3 \times 4}$ be the matrices $Y$ and $X$, respectively, after adding a row of ones at the bottom.   Observe that $\rank(\bar Y) = \rank(\bar X) = 3$.

Define the matrix $S$ of coefficients as
$$
S =   
\begin{bmatrix}
\frac{1}{\alpha'}&  & & \frac{\delta}{\delta'}\\
 \frac{\alpha}{\alpha'}  &\frac{1}{\beta'}&& \\
& \frac{\beta}{\beta'} & \frac{1}{\gamma'}  &\\
   &  & \frac{\gamma}{\gamma'}  &\frac{1}{\delta'} 
\end{bmatrix}
$$
 and note that $\bar Y = \bar X S$.  Since $y^1, \dots, y^4$ form a
 parallelogram by Theorem~\ref{mlfcb2}, $\bar y^1 + \bar y^3 = \bar y^2 + \bar
 y^4$, or equivalently,
$
\bar Y u = 0
$
where $u = (1;-1;1;-1)$.
This implies that
$
\bar Y u = \bar X S u = 0.
$
Since we assume that it is not true that $
\alpha = \frac{1}{\beta} = \gamma = \frac{1}{\delta},
$ it follows that $Su \neq 0$.  Therefore, $Su \in \ker  (\bar X) \setminus
\{0\}$, where $\ker$ denotes the kernel.   Some simple linear algebra, made
more explicit in \cite[Lemma 3.5, proof of Theorem~3.10]{cm}, shows that $\ker(\bar X) \subseteq \im(S)$, where $\im(S)$ denotes the column space,   and that 
$$
\rank(\bar X S) = \rank(S) - \dim(\ker(\bar X) \cap \im(S))
$$
from which it follows that $\rank(S) = 4$.  Therefore, $\det(S) = d$ for some $d \neq 0$, where $
\det(S)  
= \frac{1 - \alpha \beta \gamma \delta}{\alpha' \beta ' \gamma' \delta'}
$.
 Therefore
$$
\alpha \beta \gamma \delta = 1 - d \alpha' \beta' \gamma' \delta'
$$
and since  $\alpha', \beta', \gamma', \delta' \neq 0$, we conclude that 
$\alpha \beta \gamma \delta \neq 1$.

We have shown that the ratio condition holds if and only if $\dim \mathcal N(\Y) = 0$.  Lastly, note that the set $\{\, M(B + \bar A) \st \bar A \in \mathcal N(\Y)\,\}$ is the set of all  quadrilaterals defined by the four corner rays $r^{j_1}, r^{j_2}, r^{j_3}, r^{j_4}$ and the four integer points $y^1, y^2, y^3, y^4$.  Hence, $\dim \mathcal N(\Y) = 0$ if and only if $M(B)$ is the unique quadrilateral defined by these corner rays and integer points.   
\end{proof} 

\scriptsize
\bibliography{../bib/MLFCB_bib}{}
\bibliographystyle{../amsabbrv}

\end{document}